\documentclass[11pt]{amsart}
\usepackage{geometry}
\usepackage{graphicx}
\usepackage{amssymb}
\usepackage{epstopdf}
\usepackage{amsmath,amscd}
\usepackage{amsthm}
\usepackage{enumerate}
\usepackage{url,verbatim}
\usepackage{subfig}
\usepackage{enumitem}

\usepackage[normalem]{ulem}
\usepackage{fixme}

\RequirePackage[colorlinks,citecolor=blue,urlcolor=blue]{hyperref}
\usepackage{breakurl}
\theoremstyle{plain}
\newtheorem{theorem}{Theorem}
\newtheorem{definition}[theorem]{Definition}
\newtheorem{lemma}[theorem]{Lemma}
\newtheorem{proposition}[theorem]{Proposition}
\newtheorem{corollary}[theorem]{Corollary}

\newtheorem{assumption}[theorem]{Assumption}

\newtheorem{claim}[theorem]{Claim}

\newcommand\ol{\overline}

\newcommand\RR{{\mathbb R}}
\newcommand\ZZ{{\mathbb Z}}
\newcommand\NN{{\mathbb N}}

\renewcommand\ell{l}

\newcommand\CC{\mathbb{C}}

\newcounter{mycount}

\numberwithin{equation}{section}
\numberwithin{theorem}{section}
\numberwithin{figure}{section}

\title{Exact Recovery of Community Detection in k-partite Graph Models}

\author{Zhongyang Li}
\address{Department of Mathematics,
University of Connecticut,
Storrs, Connecticut 06269-3009, USA}
\email{zhongyang.li@uconn.edu}
\urladdr{\url{https://mathzhongyangli.wordpress.com}}

\begin{document}
\maketitle

\begin{abstract}
We study the vertex classification problem on a graph whose vertices are in $k\ (k\geq 2)$ different communities,  edges are only allowed between distinct communities, and the number of vertices in different communities are not necessarily equal. The observation is a weighted adjacency matrix, perturbed by a scalar multiple of the Gaussian Orthogonal Ensemble (GOE), or Gaussian Unitary Ensemble (GUE) matrix. For the exact recovery of the maximum likelihood estimation (MLE) with various weighted adjacency matrices, we prove sharp thresholds of the intensity $\sigma$ of the Gaussian perturbation. These weighted adjacency matrices may be considered as natural models for the electric network. Surprisingly, these thresholds of $\sigma$ do not depend on whether the sample space for MLE is restricted to such classifications that the number of vertices in each group is equal to the true value.  In contrast to the $\ZZ_2$-synchronization, a new complex version of the semi-definite programming (SDP) is designed to efficiently implement the community detection problem when the number of communities $k$ is greater than 2, and a common region (independent of $k$) for $\sigma$ such that SDP exactly recovers the true classification is obtained.
\end{abstract}

\section{Introduction}
Most graphs of interest display community structure, i.e., their vertices are organized into groups, called communities, clusters or modules. In some cases, edges are concentrated within groups. For example, vertices of a graph  may represent scientists, edges join coauthors.  Each group consists of vertices representing scientists working on the same research topic, where collaborations are more frequent.  Likewise, communities could represent molecules with similar structure in molecule interaction networks among which interactions are more likely, groups of friends in social networks who communicate more often, websites on similar topics in the web graph where there are more hyperlinks in between, and so on. In some other cases, edges may only be possible between vertices in distinct groups. For instance, in an electrical network, electrical current can only be observed between two sites with different electrical potential; commercial trades can only occur when two individuals own different goods.  Identifying communities may offer insight on how the network is organized. It allows us to focus on regions having some degree of similarity within the graph.  It helps to classify the vertices, based on their role with respect to the communities they belong to.  For instance we can distinguish vertices in the interior of their clusters from vertices at the boundary of the clusters, which may act  as  brokers  between  the  modules  and,  in  that  case, could play a crucial role both in holding the modules together and in the dynamics of spreading processes across the network. 

Identifying different communities in the stochastic block model is a central topic in many fields of science and technology; see \cite{EA18} for a summary. A lot of spectacular work has been done when the graph has two equal-sized communities, see, for example, \cite{LM14,MNS13,ABH15} for an incomplete list. Community detection with two equal-sized communities has also been studied on hyper-graphs, see \cite{KBG}. 

In this paper, we instead study the community detection on a graph in which there are $k (k\geq 2)$ distinct clusters, not necessarily equal-sized, and edges are only allowed between vertices in different communities. This corresponds to the famous $k$-partite graph in graph theory. The observation is a weighted adjacency matrix, perturbed by a $\sigma$-multiple of the Gaussian Orthogonal Ensemble (GOE), or the Gaussian Unitary Ensemble (GUE) matrix. Here $\sigma$ is a positive number representing the intensity of the Gaussian perturbation. These weighted adjacency matrices, as will be explained later, may be considered as natural models for the electric network. Given such observations, we apply the maximum likelihood estimation (MLE) to determine which vertex belongs to which group, or community. We obtain a division, or an assignment of vertices of the graph into communities by the MLE, if this assignment of vertices is the same as the true assignment of vertices into communities for every vertex, we say that the MLE exactly recovers the true community structure of the graph, or the exact recovery occurs for the MLE.

The main goal of the paper is to investigate the condition when the MLE exactly recovers the true community structure of the graph.
We prove sharp phase transition results with respect to the intensity $\sigma$ of the Gaussian perturbation for the exact recovery of the MLE. More precisely, we explicitly find the critical value for $\sigma$, such that if $\sigma$ is less than the critical value, then as the size of the graph goes to infinity, with probability tending to 1 exact recovery occurs. On the other hand, if $\sigma$ is greater than the critical value, then as the size of the graph goes to infinity, with probability tending to 0 exact recovery occurs. Interestingly, the threshold, or critical value, of $\sigma$ does not depend on whether or not we restrict the sample space for MLE  to those classifications in which the number of vertices of each group, or community, is the same as the true value. These results are obtained by analyzing the Gaussian distribution through various inequalities.

Semidefinite programming (SDP) is one of the most exciting developments in math­ematical programming  in the 1990’s. SDP has  applications  in  diverse fields  including, but not restricted to,  traditional  convex  constrained  optimization,  control  theory,  and combinatorial  optimization.  A linear programming (LP) problem is one in which we wish to maximize or minimize a linear objective function of real variables over a polytope. In SDP, we instead use real-valued vectors and are allowed to take the dot product of vectors; non-negativity constraints on real variables in LP are replaced by semi-definiteness constraints on matrix variables in SDP.  Because SDP is  solvable  via  interior  point methods,  most  of  these  applications  can  usually  be  solved  very  efficiently in  practice  as  well  as  in  theory. 

It is well-known that the community detection problem with $k=2$ equal-sized communities may be efficiently solved by a semi-definite programming algorithm; see, for instance, \cite{JMT16S, JMRT16}. When there are $k\geq 3$ different communities, we can design a ``complex version'' of the semi-definite programming for efficient recovery. The idea is to relax the constraint on the rank of the optimal solution, solve the optimization problem on a larger space of the semi-definite matrices, and then achieve efficient recovery. We also obtain an interval of $\sigma$ to guarantee the exact recovery of the SDP, by applying the celebrated result of the Tracy-Widom fluctuation of the maximal eigenvalue of the GOE matrix; see \cite{TW}.

\section{Main Results.}
In this section, we state the main results proved in the paper.
We first discuss the basic definition and notation of the $k$-partite graph, where $k$ is a positive integer with value at least 2.

A $k$-partite graph $G=(V,E)$ is a graph whose vertices can be colored in $k$ different colors such that any two vertices of the same color cannot be adjacent, or joined by an edge. Assume $V=[n]:=\{1,2,\ldots,n\}$ is the vertex set of $G$. Define the set of colors 
\begin{eqnarray*}
R_k:=\{c_1,\ldots,c_k\}\subset \RR
\end{eqnarray*}
to be a set consisting of $k$ distinct real numbers representing $k$ different colors.

Let $x:[n]\rightarrow R_k$ be a mapping from the set of vertices to the set of colors, i.e., it assigns a unique color in $R_k$ to each vertex in $[n]$. Such a mapping $x$ is also called a color assignment mapping of a graph.
For $1\leq i\leq k$, let 
\begin{eqnarray*}
x^{-1}(c_i)=\{j\in [n]:y(j)=c_i\};
\end{eqnarray*}
that is, $y^{-1}(c_i)$ is the set consisting of all the vertices with color $c_i$ under the mapping $y$.
For each color $c_i\in R_k$ ($1\leq i\leq k$), let
\begin{eqnarray*}
n_i(x):=\left|x^{-1}(c_i)\right|
\end{eqnarray*}
In other words, $n_i(x)$ is the total number of vertices in $[n]$ with color $c_i$ under the mapping $x$. It is straightforward to see that $n_i(x)$'s are positive integers satisfying 
\begin{eqnarray}
\sum_{i=1}^{k} n_i(x)=n.\label{sc}
\end{eqnarray} 
Given these definitions, to determine the colors of all the vertices of a graph is the same as to identify the color assignment mapping of the the graph.
In vertex-color-detection problems that will be discussed later, we shall find color assignment mappings of graphs from the following spaces of color assignment mappings:
\begin{enumerate}
\item For each positive integer $n$, let $n_1\geq n_2\geq \ldots\geq n_k$ be fixed and satisfy (\ref{sc}). Let $\Omega_{n_1,\ldots,n_k}$ be the set of all the color assignment mappings under which the number of vertices with color $c_i$ is exactly $n_i$ for each $1\leq i\leq k$; that is,
\begin{eqnarray*}
\Omega_{n_1,\ldots,n_k}=\{\left.x:[n]\rightarrow R_k\right| |x^{-1}(c_i)|=n_i,\ \forall\ 1\leq i\leq k\}.
\end{eqnarray*}

\item Let $\Omega$ be the set of all the color assignment mappings; that is,
\begin{eqnarray*}
\Omega=\{x: [n]\rightarrow R_k\}.
\end{eqnarray*}

\item Let $c>0$.  Let $\Omega_{c}$ be the set of all the color assignment mappings in $\Omega$ such that the number of vertices in each color is at least $cn$, i.e.
\begin{eqnarray*}
\Omega_{n_1,\ldots,n_k}=\{\left.x: [n]\rightarrow R_k\right|:\min_{j\in [k]}\left|x^{-1}(c_j)\right|\geq cn \}.
\end{eqnarray*}

\end{enumerate}

\subsection{Real Weighted Adjacency Matrix with Gaussian Perturbation}

We first consider the community detection problem when the observation is a real weighted adjacency matrix with Gaussian perturbation.

In Theorems \ref{m1} and \ref{m2}, we observe different weighted adjacency matrices for a $k$-partite graph,  both of which are perturbed by a $\sigma$-multiple of a matrix with i.i.d.~standard Gaussian entries. The weighted adjacency matrix in \ref{m1} can be considered as a natural model for an electrical network, where each one of the $n$ vertices has one of the $k$ distinct electric potentials. The weight of each (oriented) edge is the difference of electric potentials between its initial point and its terminal point. This difference in potentials is proportional to the intensity of electric current on the edge. Given the observation, the goal is to find the difference of electric potentials between each pair of vertices; then we can determine the electrical potential of each vertex up to an additive constant. We consider the probability of the exact recovery of MLE, and find a sharp threshold with respect to the intensity $\sigma$ of the Gaussian perturbation. More precisely, there is a critical value $\sigma_c$ depending on $n$, such that if $\lim_{n\rightarrow\infty}\frac{|\sigma|}{\sigma_c}<1$, the limit of the probabilities for the exact recoveries of MLE as $n\rightarrow\infty$ is 1; while if $\lim_{n\rightarrow\infty}\frac{|\sigma|}{\sigma_c}>1$, the limit of the probabilities for the exact recovery of MLE as $n\rightarrow\infty$ is 0.

Before stating Theorem \ref{m1}, recall that the Frobenius norm of an $m\times m$ complex matrix $A=\{A_{i,j}\}_{i,j=1}^m\in \CC^{m\times m}$ is defined by
\begin{eqnarray*}
\|A\|_F:=\sqrt{\sum_{i,j=1}^m |A_{ij}|^2}
\end{eqnarray*}

\begin{theorem}\label{m1}Let $n_1\geq \ldots\geq n_k$ be the numbers of vertices in $k$ different colors $c_1,\ldots,c_k$, respectively. For an arbitrary mapping $x\in \Omega$, let $\mathbf{G}(x)$ be the $n\times n$ square matrix whose entries are defined by
\begin{eqnarray}
\mathbf{G}_{i,j}(x)=x(i)-x(j);\label{gij}
\end{eqnarray}
where $1\leq i,j\leq n$.
Let $y\in \Omega_{n_1,\ldots,n_k}$ be the true color assignment function.  Assume the observation is given by 
\begin{eqnarray}
\mathbf{T}=\mathbf{G}(y)+\sigma \mathbf{W}\label{tgw}
\end{eqnarray}
where $\mathbf{W}$ is a random $n\times n$ matrix with i.i.d.~standard Gaussian entries, and $\sigma\in \RR$ is deterministic.
 Let $k$ and $\{c_1,\ldots,c_k\}$ be fixed as $n\rightarrow\infty$. Assume that there exists a constant $c>0$ independent of $n$, such that $\frac{n_k}{n}\geq c$ for all $n$, Let
\begin{eqnarray}
\hat{y}&=&\mathrm{argmin}_{x\in \Omega_{n_1,\ldots,n_k}}\|\mathbf{T}-\mathbf{G}(x)\|_F^2;\label{hy}\\
\check{y}&=&\mathrm{argmin}_{x\in \Omega}\|\mathbf{T}-\mathbf{G}(x)\|_F^2.\label{cy}
\end{eqnarray}
We have
\begin{enumerate}
\item  Assume there exists a constant $\delta>0$ independent of $n$, such that
\begin{eqnarray}
\sigma^2<\frac{(1-\delta)n\min_{1\leq i<j\leq k}(c_i-c_j)^2}{4\log n}.\label{ss2}
\end{eqnarray}
Then 
\begin{eqnarray*}
\lim_{n\rightarrow\infty}\mathrm{Pr}(\hat{y}=y)=1;\ 
 \mathrm{and}\ \lim_{n\rightarrow\infty}\mathrm{Pr}(\check{y}=y)=1.
\end{eqnarray*}
\item Assume there exists a constant $\delta>0$ independent of $n$, such that
\begin{eqnarray}
\sigma^2>\frac{(1+\delta)n\min_{1\leq i<j\leq k}(c_i-c_j)^2}{4\log n}.\label{s2}
\end{eqnarray}
Then
\begin{eqnarray*}
\lim_{n\rightarrow\infty}\mathrm{Pr}(\hat{y}=y)=0;\ \mathrm{and}\ \lim_{n\rightarrow\infty}\mathrm{Pr}(\check{y}=y)=0.
\end{eqnarray*}
\end{enumerate}

\end{theorem}

Note that $\hat{y}$ (resp.\ $\check{y}$) are actually the maximum likelihood estimation (MLE) of the true value $y$ in the sample space $\Omega_{n_1,\ldots,n_k}$ (resp.\ $\Omega$) with respect to the given observation $\mathbf{T}$, since for each given color assignment mapping $x\in \Omega$, the probability density at each $\mathbf{T}$ is proportional to $e^{-\frac{\|\mathbf{T}-\mathbf{G}(x)\|_F^2}{2}}$.

More general versions of Theorem \ref{m1} are proved in Sections \ref{pm11} and \ref{pm12}; see Propositions \ref{pp35}, \ref{pp37}, \ref{pp43}, \ref{pp44}. In these propositions, we also allow the total number of colors $k$ to change with $n$. It is straight forward to check that when $k$ and $R_k$ are fixed as $n\rightarrow\infty$, Theorem \ref{m1} is a special case of these propositions.

 In Theorem \ref{m2}, we observe the uniformly-weighted adjacency matrix for the undirected $k$-partite graph, perturbed by a noise which is a $\sigma$-multiple of a matrix with i.i.d.~standard Gaussian entries. Given the observation, the goal is to determine whether two vertices have the same color or not. For $x,y\in \Omega$, we say $x$ and $y$ are equivalent if for all $i,j\in[n]$, $x(i)=x(j)$ if and only fi $y(i)=y(j)$. We write $x\in C(y)$ if $x$ and $y$ are equivalent. Now we only need the algorithm to find a color assignment mapping which is equivalent to the true color assignment mapping.
 Again we find a sharp threshold for the probability of the exact recovery of the MLE with respect to the parameter $\sigma$.

\begin{theorem}\label{m2}Let $n_1\geq \ldots\geq n_k$ be the numbers of vertices in the $k$ different colors, respectively. For an arbitrary mapping $x\in \Omega$, let $\mathbf{K}(x)$ be the $n\times n$ square matrix with entries defined by
\begin{eqnarray}
\mathbf{K}_{i,j}(x)=\begin{cases} 1& \mathrm{if}\ x(i)\neq x(j)\\ 0&\mathrm{if}\ x(i)=x(j)\end{cases};\label{kij}
\end{eqnarray}
where $1\leq i,j\leq n$.
Let $y\in \Omega_{n_1,\ldots,n_k}$ be the true color assignment function. Assume the observation is given by
\begin{eqnarray}
\mathbf{R}=\mathbf{K}(y)+\sigma \mathbf{W}\label{rgw}
\end{eqnarray}
where $\mathbf{W}$ is a random $n\times n$ matrix with i.i.d.~standard Gaussian entries as before. 
 Let the total number of colors $k$ and the set of all colors $\{c_1,\ldots,c_k\}$ be fixed as $n\rightarrow\infty$. Assume there exists a constant $c>0$, such that $n_k\geq cn$ for all $n$.
 Let
\begin{eqnarray}
\tilde{y}&=&\mathrm{argmin}_{x\in \Omega_{\frac{2c}{3}}}\|\mathbf{R}-\mathbf{K}(x)\|^2_F\label{ty}\\
\overline{y}&=&\mathrm{argmin}_{x\in \Omega_{n_1,\ldots,n_k}}\|\mathbf{R}-\mathbf{K}(x)\|^2_F\label{ty}.\label{by}
\end{eqnarray}
We have
\begin{enumerate}
\item 
Assume there exists $\delta>0$, such that
\begin{eqnarray}
\sigma^2<\frac{(1-\delta)(n_k+n_{k-1})}{4\log n}\label{ss2}
\end{eqnarray}
Then 
\begin{eqnarray*}
\lim_{n\rightarrow\infty}\mathrm{Pr}(\tilde{y}\in C(y))=1;\ \mathrm{and}\ \lim_{n\rightarrow\infty}\mathrm{Pr}(\overline{y}\in C(y))=1.
\end{eqnarray*}
\item Assume there exists $\delta>0$, such that
\begin{eqnarray}
\sigma^2>\frac{(1+\delta)(n_k+n_{k-1})}{4\log n}\label{s2}
\end{eqnarray}
then
\begin{eqnarray*}
\lim_{n\rightarrow\infty}\mathrm{Pr}(\tilde{y}\in C(y))=0;\ \mathrm{and}\ \lim_{n\rightarrow\infty}\mathrm{Pr}(\overline{y}\in C(y))=0.
\end{eqnarray*}
\end{enumerate}

\end{theorem}

More general versions of Theorem \ref{m2} are proved in Sections \ref{pm21} and \ref{pm22}; see Propositions \ref{l2a}, \ref{l2b}, \ref{p67}, \ref{p68}. In these propositions, we also allow the total number of colors $k$ to change with $n$. It is straight forward to check that when $k$ and $R_k$ are fixed as $n\rightarrow\infty$, Theorem \ref{m2} is a special case of these propositions. The proofs of Theorems \ref{m1} and {\ref{m2}} are based on various inequalities of Gaussian distributions.

 From these two theorems, we can see that we may either choose the sample space for MLE as all the possible assignments of $k$ colors, or potentials, to $n$ vertices in Theorem \ref{m1} (all the possible classifications of $n$ vertices in $k$ distinct groups in Theorem \ref{m2}), or choose the sample space for MLE to be restricted on all the classifications such that the number of vertices of each type coincides with that of the true value - either way we obtain the same threshold.

 \subsection{Complex Unitary Matrix with Gaussian Perturbation}
 
 Now we consider the community detection problem with $k\geq 2$ different communities when the observation is a complex unitary matrix with Gaussian perturbation. Community detection problems with such an observation matrix may be efficiently recovered by the SDP.
 
 Let $d_1,\ldots,d_k\in[0,2\pi)$ be $k$ distinct real numbers. Let $\mathbf{i}$ satisfy $\mathbf{i}^2=-1$ be the imaginary unit. Let $x:[n]\rightarrow\{e^{\mathbf{i}d_1},\ldots,e^{\mathbf{i}d_k}\}$ be a mapping which assigns each vertex in $[n]$ a unique color represented by a complex number of modulus 1. Let $\Theta$ be the set consisting of all such mappings, that is
 \begin{eqnarray*}
 \Theta:=\{x:[n]\rightarrow\{e^{\mathbf{i}d_1},\ldots,e^{\mathbf{i}d_k}\}\}
 \end{eqnarray*}

For a mapping $x\in \Theta$, let $\mathbf{P}(x)$ be an $n\times n$ matrix whose entries are given by
\begin{eqnarray*}
\mathbf{P}_{a,b}(x)=x(a)\overline{x(b)}=e^{\mathrm{Log} [x(a)]-\mathrm{Log} [x(b)]},\qquad 1\leq a,b\leq n,
\end{eqnarray*}
where $\ol{x(b)}$ is the complex conjugate of $x(b)$ and $\mathrm{Log}[\cdot]$ is the principal branch of the complex logarithmic function. 

For each $x\in \Theta$, if we consider $\mathbf{x}$ as an $n\times 1$ vector given by
\begin{eqnarray*}
\mathbf{x}=(x(1),\ldots,x(n))^t,
\end{eqnarray*}
then
\begin{eqnarray*}
\mathbf{P}(x)=\mathbf{x}\ol{\mathbf{x}}^t;
\end{eqnarray*}
which is a rank-1 positive semi-definite Hermitian matrix  whose diagonal entries are 1.

We may further restrict the MLE to a subspace of $\Theta$ satisfying the following assumptions. 

\begin{assumption}\label{ap1}
\begin{enumerate}
\item The number of vertices in each color is the same. In particular, this implies that the total number of vertices $n$, is an integer multiple of the total number of colors $k$.
\item $e^{\mathbf{i}d_1},\ldots, e^{\mathbf{i}d_k}$ are the $k$th roots of unity. Without loss of generality, assume that $d_l=\frac{2(l-1)\pi}{k}$, for $l=1,\ldots,k$.
\end{enumerate}
Let $\Theta_{A}$ be the subset of $\Theta$ consisting of all the mappings satisfying the above two assumptions. That is,
\begin{eqnarray*}
\Theta_{A}:=\left\{x\in \Theta : \left|x^{-1}\left(e^{\mathbf{i}d_j}\right)\right|=\frac{n}{k},\ \forall 1\leq j\leq k\right\}.
\end{eqnarray*}
We define an equivalence class on $\Theta_A$ as follows. We say $x,y\in\Theta_A$ are equivalent if there exists a fixed angle $\alpha$, such that $e^{\mathbf{i}\alpha}\mathbf{x}=\mathbf{y}$. For each $z\in \Theta_{A}$, let $\theta(z)$ be the equivalence class containing $y$.
\end{assumption}

Let $y\in \Theta_A$ be the true color assignment function. 
 Define the observation by
\begin{eqnarray*}
\mathbf{U}=\mathbf{P}(y)+\sigma\mathbf{W}_c
\end{eqnarray*}
where $\mathbf{W}_c$ is the standard GUE random matrix. More precisely, $\mathbf{W}_c$ is a random Hermitian matrix whose diagonal entries are i.i.d. standard real Gaussian random variables ($\mathcal{N}(0,1)_{\mathbb{R}}$), and upper triangular entries are i.i.d.~standard complex Gaussian random variables ($\mathcal{N}(0,1)_{\mathbb{C}}$).

Given each observation $\mathbf{U}$, the goal is to determine the true color assignment function $y$, up to a multiplicative constant.  Let 
\begin{eqnarray}
y^{A}=\mathrm{argmin}_{x\in \Theta_A}\|\mathbf{U}-\mathbf{P}(x)\|_F^2\label{ys}
\end{eqnarray}
Note that for any $x\in\Theta$,
\begin{eqnarray*}
\|\mathbf{P}(x)\|_F^2=\sum_{1\leq a,b\leq n}x(a)\ol{x(b)}x(b)\ol{x(a)}=n^2,
\end{eqnarray*}
which is independent of $x$. Hence we have
\begin{eqnarray*}
y^{A}=\mathrm{argmax}_{x\in\Theta_A}\Re\langle\mathbf{U},\mathbf{P}(x) \rangle
\end{eqnarray*}
where $\Re$ denotes the real part of a complex number, and $\langle\cdot,\cdot\rangle$ denotes the inner product of two matrices defined by
\begin{eqnarray}
\langle M_1, M_2 \rangle=\sum_{i,j\in[n]}M_1(i,j)\ol{M_2(i,j)}.\label{cip}
\end{eqnarray}
where $M_1,M_2\in \CC^{n\times n}$.

It is not hard to see that for any $x,z\in \Theta_{A}$, if $x\in \theta(z)$, then $\mathbf{P}(x)=\mathbf{P}(z)$. So given any observation $\mathbf{U}$, we cannot distinguish color assignment mappings in the same equivalence class of $\theta_{A}$. Therefore the best we can do by an MLE algorithm is to recover the equivalence class of the true color assignment function. We have the following theorem.

 \begin{theorem}\label{m3}
\begin{enumerate}Let $y\in\Theta_A$ be the true color assignment function satisfying Assumption \ref{ap1}. Let $k$ be fixed and $n\rightarrow \infty$. 
\begin{enumerate}
\item If there exists $\delta>0$, such that 
\begin{eqnarray}
\sigma^2<\frac{(1-\delta)\left[n(1-\cos\frac{2\pi}{k})\right]}{2\log n}\label{sc1}
\end{eqnarray}
then 
\begin{eqnarray*}
\lim_{n\rightarrow\infty}\mathrm{Pr}(y^A\in \theta(y))=1
\end{eqnarray*}
\item If there exists $\delta>0$, such that 
\begin{eqnarray*}
\sigma^2>\frac{(1+\delta)\left[n(1-\cos\frac{2\pi}{k})\right]}{2\log n}
\end{eqnarray*}
then 
\begin{eqnarray*}
\lim_{n\rightarrow\infty}\mathrm{Pr}(y^A\in \theta(y))=0.
\end{eqnarray*}
\end{enumerate}

\end{enumerate}
\end{theorem}

Now we describe a complex semi-definite programming for community detection with multiple (more than 2) communities.
For a true color assignment function $y\in \Theta$, let the observation be given by
\begin{eqnarray*}
\mathbf{V}=\mathbf{P}(y)+\sigma\mathrm{diag}(\mathbf{y})\mathbf{W}_s\mathrm{diag}(\ol{\mathbf{y}})
\end{eqnarray*}
where $\mathbf{W}_s$ is the standard GOE random matrix. More precisely, $\mathbf{W}_s$ is a random symmetric matrix whose diagonal entries and upper triangular entries are i.i.d. standard real Gaussian random variables ($\mathcal{N}(0,1)_{\mathbb{R}}$). Note that $\mathbf{V}$ is a Hermitian matrix.

Given each observation $\mathbf{V}$, the goal is to find the true color assignment function $y$, up to a multiplicative constant.  We may consider the following optimization problem
\begin{eqnarray}
&&\max \Re\langle \mathbf{V},X \rangle\label{co}\\
\mathrm{subject\ to}\ &&X_{ii}=1\notag\\
\mathrm{and}\ &&X\succeq 0\notag
\end{eqnarray}
where $X\succeq 0$ means that $X$ is a positive semi-definite Hermitian matrix.

For any mapping $x\in \Theta$, the matrix $P(x):=\mathbf{x}\overline{\mathbf{x}}^t$ is a rank-1 Hermitian matrix. The idea for the MLE is to find the minimizer 
\begin{eqnarray*}
\mathrm{argmin}_{P(x)}\|P(x)-\mathbf{V}\|^2
\end{eqnarray*}
among all the rank-1 positive semi-definite Hermitian matrices with diagonal entries $1$ and the form $\mathbf{x}\mathbf{x}^t$. To achieve efficient recovery, we may relax the rank-1 constraint and consider instead an optimization over all the positive semi-definite Hermitian matrix with diagonal entries 1.

Then we have the following theorem
\begin{theorem}\label{m5}Let $p(Y;\sigma)$ be the probability that the solution $Y$ of (\ref{co}) coincides with $\mathbf{y}\ol{\mathbf{y}}^{t}$, where $y$ is the true color assignment mapping for vertices. If there exists a constant $\delta>0$ independent of $N$, such that $\sigma<\frac{(1-\delta)\sqrt{n}}{\sqrt{2\log n}}$, then
\begin{eqnarray*}
\lim_{n\rightarrow\infty}p(Y;\sigma)=1
\end{eqnarray*}
\end{theorem}

In Theorem \ref{m5}, it turns out that the bound $\frac{(1-\delta)\sqrt{n}}{\sqrt{2\log n}}$ on $\sigma$ to guarantee the exact recovery of the SDP is independent of $k$ - the total number of communities. However, if we instead use the GUE matrix $\mathbf{W}_c$ instead of $\mathrm{diag}(\mathbf{y})\mathbf{W}_s\mathrm{diag}(\mathbf{\ol{y}})$ to represent the noise, Theorem \ref{m3} shows that threshold of $\sigma$ to guarantee the exact recovery of MLE does depend on $k$.  This threshold is of order $O\left(\frac{1}{\sqrt{n\log n}}\right)$ when $k\sim n$. Since the SDP is an algorithm obtained from the relaxation of the rank constraint of the MLE, one may naturally expect a smaller common upper bound for $\sigma$ to guarantee the exact recovery of the SDP for all $k$, when the noise is represented by a $\sigma$-multiple of $\mathbf{W}_c$. 

Theorem \ref{m3} is proved in Section \ref{pm3}; and Theorem \ref{m5} is proved in Section \ref{pm5}.

\section{Proof of Theorem \ref{m1} when the number of vertices in each color is fixed in the sample space}\label{pm11}

We first consider the MLE in Theorem \ref{m1} with sample space $\Omega_{n_1,\ldots,n_k}$, where the number of vertices with color $c_i$ is fixed to be $n_i$ - the same as the true value, for $1\leq i\leq k$. For a mapping $x\in \Omega_{n_1,\ldots,n_k}$, let $\mathbf{G}(x)$, $\mathbf{T}$ be defined as in (\ref{gij}), (\ref{tgw}), respectively.

Given a sample $\mathbf{T}$, the goal is to determine the color assignment function $y$. Let $\hat{y}$ be defined by (\ref{hy}). Note that for all $x\in \Omega_{n_1,\ldots,n_k}$, 
\begin{eqnarray*}
\|\mathbf{G}(x)\|_F^2=2\sum_{1\leq i<j\leq k} n_in_j(c_i-c_j)^2,
\end{eqnarray*}
which depends only on $n_1,\ldots,n_k$, but is independent of $x$.
Then we have
\begin{eqnarray*}
\hat{y}=\mathrm{argmax}_{x\in \Omega_{n_1,\ldots,n_k}}\langle \mathbf{G}(x),\mathbf{T} \rangle
\end{eqnarray*}
Let 
\begin{eqnarray}
p(\hat{y},\sigma)=\mathrm{Pr}(\hat{y}=y),\label{phy}
\end{eqnarray}
where $y\in \Omega_{n_1,\ldots,n_k}$ is the true color assignment function.

For each $x\in \Omega_{n_1,\ldots,n_k}$, define
\begin{eqnarray*}
f(x)=\langle \mathbf{G}(x),\mathbf{T} \rangle
\end{eqnarray*}

 Then 
\begin{eqnarray*}
p(\hat{y},\sigma)=\mathrm{Pr}\left(f(y)>\max_{x\in \Omega_{n_1,\ldots,n_k}\setminus \{y\}}f(x)\right)
\end{eqnarray*}
Note that
\begin{eqnarray*}
f(x)-f(y)=\langle\mathbf{G}(y), \mathbf{G}(x)-\mathbf{G}(y) \rangle+\sigma\langle\mathbf{W}, \mathbf{G}(x)-\mathbf{G}(y) \rangle.
\end{eqnarray*}
The expression above shows that $f(x)-f(y)$ is a Gaussian random variable with mean $\langle\mathbf{G}(y), \mathbf{G}(x)-\mathbf{G}(y) \rangle$ and variance $\sigma^2\|\mathbf{G}(x)-\mathbf{G}(y)\|_F^2$.

For $i,j\in[k]$, $x,y\in \Omega_{n_1,\ldots,n_k}$, let 
\begin{eqnarray}
S_{i,j}(x,y)=\{l\in[n]:x(l)=c_i, y(l)=c_j\};\label{1sij}
\end{eqnarray}
i.e., $S_{i,j}(x,y)$ consists of all the vertices which have color $c_i$ in $x$ and color $c_j$ in $y$.

 Let $t_{i,j}(x,y)=|S_{i,j}(x,y)|$, i.e., $t_{i,j}$ is the total number of vertices which have color $c_i$ in $x$ and color $c_j$ in $y$. We may write $t_{i,j}$ (resp.\ $S_{i,j}$) instead of $t_{i,j}(x,y)$ (resp.\ $S_{i,j}(x,y)$) when there is no confusion.
Note that since $x,y\in \Omega_{n_1,\ldots,n_k}$, we have
\begin{eqnarray}
&&\sum_{j\in[k]}t_{i,j}=n_i,\ \ \forall i\in[k]\label{t1}\\
&&\sum_{i\in[k]}t_{i,j}=n_j,\ \ \forall j\in[k]\label{t2}
\end{eqnarray}
For each vertex $l\in S_{i,j}$, the inner product of the row in $\mathbf{G}(x)$ corresponding to $l$ and the row in $\mathbf{G}(y)$ corresponding to $l$ is
\begin{eqnarray*}
\langle \mathbf{G}_l(x),\mathbf{G}_l(y) \rangle&=&\sum_{r\in [n]}\mathbf{G}_{l,r}(x)\overline{\mathbf{G}_{l,r}(y)}\\
&=&\sum_{r\in [n]}\left[x(l)-x(r)\right]\left[y(l)-y(r)\right]\\
&=&\sum_{u=1}^{k}\sum_{v=1}^{k}t_{u,v}(c_i-c_u)(c_j-c_v)
\end{eqnarray*}
Then
\begin{eqnarray}
\langle\mathbf{G}(x),\mathbf{G}(y) \rangle&=&\sum_{l=1}^{n}\langle\mathbf{G}_l(x),\mathbf{G}_l(y) \rangle\notag\\
&=&\sum_{i,j,u,v\in [k]}t_{i,j}t_{u,v}(c_i-c_u)(c_j-c_v)\label{pgxy}\\
&=&2\left[\sum_{i,j\in [k]} t_{i,j}\cdot c_i\cdot c_j\right]\left[\sum_{u,v\in[k]}t_{u,v}\right]-2\left[\sum_{i,j\in [k]} t_{i,j}\cdot c_i \right]\left[\sum_{u,v\in[k]}t_{u,v}\cdot c_v\right]\notag\\
&=&2n\left[\sum_{i,j\in [k]} t_{i,j}\cdot c_i\cdot c_j\right]-2\left[\sum_{i\in [k]} n_i\cdot c_i \right]^2,\notag
\end{eqnarray}
where the last identity follows from (\ref{t1}) and (\ref{t2}).

Note that for $i,j\in[k]$ and $x\in \Omega_{n_1,\ldots,n_k}$,
\begin{eqnarray*}
t_{i,j}(x,x)=\begin{cases}n_i&\mathrm{if}\ i=j;\\ 0&\mathrm{else}.\end{cases}
\end{eqnarray*}
Therefore
\begin{eqnarray*}
\langle \mathbf{G}(y),\mathbf{G}(y)\rangle=2n \sum_{i\in [k]}\left[n_i \cdot c_i^2\right]-2\left[\sum_{i\in [k]}n_i\cdot c_i\right]^2
\end{eqnarray*}
For $x,y\in \Omega_{n_1,\ldots,n_k}$, let
\begin{eqnarray}
M(x,y):&=&-\mathbb{E}[f(x)-f(y)]=-\langle \mathbf{G}(y),\mathbf{G}(x)-\mathbf{G}(y)\rangle\notag\\
&=&2n \sum_{i\in [k]}\left[n_i \cdot c_i^2\right]-2n\left[\sum_{i,j\in [k]} t_{i,j}\cdot c_i\cdot c_j\right]\label{mxy}
\end{eqnarray}

Then
\begin{eqnarray*}
\mathbf{Var}[f(x)-f(y)]=2\sigma^2 M(x,y)
\end{eqnarray*}
Therefore we have
\begin{eqnarray*}
f(x)-f(y)\sim \mathcal{N}(-M(x,y),2\sigma^2 M(x,y))
\end{eqnarray*}
and
\begin{eqnarray*}
\xi:=\frac{f(x,y)+M(x,y)}{\sigma\sqrt{2M(x,y)}}\sim \mathcal{N}(0,1).
\end{eqnarray*}
Then for $x\in \Omega_{n_1,\ldots,n_k}\setminus \{y\}$
\begin{eqnarray*}
\mathrm{Pr}\left(f(x)-f(y)>0\right)
&=&\mathrm{Pr_{\xi\sim \mathcal{N}(0,1)}}\left(\xi>\frac{\sqrt{M(x,y)}}{\sqrt{2}\sigma}\right)
\end{eqnarray*}
Using the standard Gaussian tail bound $\mathrm{Pr}_{\xi\in \mathcal{N}(0,1)}(\xi>x)<e^{-\frac{1}{2}x^2}$, we obtain
\begin{eqnarray*}
\mathrm{Pr}\left(f(x)-f(y)>0\right)\leq e^{-\frac{M(x,y)}{4\sigma^2}}
\end{eqnarray*}

Then
\begin{eqnarray}
1-p(\hat{y};\sigma)&\leq& \sum_{x\in \Omega_{n_1,\ldots,n_k}\setminus \{y\}}\mathrm{Pr}(f(x)-f(y)> 0)\label{wmp}\\
&=&\sum_{x\in \Omega_{n_1,\ldots,n_k}\setminus \{y\}}\mathrm{Pr_{\xi\sim \mathcal{N}(0,1)}}\left(\xi> \frac{\sqrt{M(x,y)}}{\sqrt{2}\sigma}\right)\notag\\
&\leq &\sum_{x\in \Omega_{n_1,\ldots,n_k}\setminus \{y\}} e^{-\frac{M(x,y)}{4\sigma^2}}\notag
\end{eqnarray}
\begin{lemma}Let $y\in \Omega_{n_1,\ldots,n_k}$ be the fixed color assignment mapping, and let $x$ change over $\Omega_{n_1,\ldots,n_k}$.
Under the constraints (\ref{t1}) and (\ref{t2}), $M(x,y)$ achieves its minimum if and only if
\begin{eqnarray*}
&&t_{i,i}=n_i\\
&&t_{i,j}=0,\qquad\mathrm{if}\ i\neq j.
\end{eqnarray*}
and the minimal value of $M(x,y)$ is 0.
\end{lemma}
\begin{proof}Note that 
\begin{eqnarray*}
\sum_{i,j\in [k]} t_{i,j}\cdot c_i\cdot c_j\leq \sum_{i,j\in [k]}  \frac{(c_i^2+c_j^2) t_{i,j}}{2}=\sum_{i\in [k]}[n_i\cdot c_i^2];
\end{eqnarray*}
where the identity holds if and only if 
\begin{itemize}
\item $c_i=c_j$ whenever $t_{i,j}\neq 0$. 
 \end{itemize}
 Then the lemma follows from the assumption that $c_i\neq c_j$ whenever $i\neq j$.
\end{proof}

Before proving Theorem \ref{m1}, we introduce a definition.
\begin{definition}Let $l\geq 2$ be a positive integer. Let $x,y\in \Omega_{n_1,\ldots,n_k}$. We say
$l$ distinct colors $(i_1,\ldots,i_{l})\in[k]^l$ is an $l$-cycle for $(x,y)$, if $t_{i_{s-1},i_s}(x,y)>0$ for all $2\leq s\leq l+1$, where $i_{l+1}:=i_1$.
\end{definition}

\begin{lemma}\label{lm33}Let $x,y\in \Omega_{n_1,\ldots,n_k}$ and $x\neq y$. Then there exists an $l$-cycle for $(x,y)$ with $2\leq l\leq k$. 
\end{lemma}

\begin{proof}Since $x\neq y$ and $x,y\in \Omega_{n_1,\ldots,n_k}$, there exists $i_1\in [k]$, such that $t_{i_1,i_1}(x,y)<n_{i_1}$. Since $\sum_{j\in[k]}t_{i_1,j}(x,y)=n_{i_1}$, there exists $i_2\in \left([k]\setminus\{i_1\}\right)$, such that $t_{i_1,i_2}(x,y)>0$. The following cases might occur:
\begin{itemize}
\item If $t_{i_2,i_1}(x,y)>0$, then we find a 2-cycle $(i_1,i_2)$ for $(x,y)$.
\item If $t_{i_2,i_1}(x,y)=0$, since $t_{i_1,i_2}(x,y)>0$ and $\sum_{j\in[k]}t_{j,i_2}(x,y)=n_{i_2}$, we obtain that $t_{i_2,i_2}(x,y)<n_{i_2}$. Moreover, since $\sum_{j\in[k]}t_{i_2,j}(x,y)=n_{i_2}$, there exists $i_3\in\left([k]\setminus\{i_1,i_2\}\right)$, such that $t_{i_2,i_3}(x,y)>0$.
\item If $t_{i_3,i_1}(x,y)>0$, then we find a 3-cycle $(i_1,i_2,i_3)$ for $(x,y)$.
\end{itemize}
In general, let $s\geq 2$. Assume we find distinct $i_1,i_2,\ldots,i_{s+1}\in[k]$, such that 
\begin{itemize}
\item for each $1\leq r\leq s$, we have $t_{i_s,i_{s+1}}(x,y)>0$; and 
\item $t_{i_{s+1},i_1}(x,y)=0$. 
\end{itemize}
Since $\sum_{j\in[k]}t_{j,i_{s+1}}(x,y)=n_{i_{s+1}}$, we have $t_{i_{s+1},i_{s+1}}(x,y)<n_{i_{s+1}}$. Since $\sum_{j\in[k]}t_{i_{s+1},j}(x,y)=n_{i_{s+1}}$, there exists $i_{s+2}\in \left([k]\setminus \{i_{s+1},i_1\}\right)$, such that $t_{i_{s+1},i_{s+2}}(x,y)>0$. The following cases might occur
\begin{enumerate}
\item if $i_{s+2}=i_{r}$, for some $2\leq r\leq s$. Then we find a $(s+2-r)$-cycle $(i_r, i_{r+1}\ldots,i_{s+1})$ for $(x,y)$.
\item if $i_{s+2}\neq i_{r}$, for all $1\leq r\leq s+1$, but there exists $1\leq g\leq s+1$, such that
\begin{eqnarray}
t_{i_{s+2},t_{i_g}}>0\label{fcn}
\end{eqnarray}
Then let $g_{M}\in[s+1]$ be the maximal $g\in[s+1]$ such that (\ref{fcn}) holds. Then we find a $(s+3-g_{M})$-cycle $(i_{g_{M}},i_{g_{M}+1},\ldots,i_{s+2})$ cycle for $(x,y)$.
\item if neither (1) or (2) occurs, increase $s$ by $1$ and repeat the above process.
\end{enumerate}
Since there are $k$ distinct colors in total, we can always find an $l$-cycle for $(x,y)$ with $2\leq l\leq k$.
\end{proof}

To measure the difference between two color assignment functions $x,y\in \Omega_{n_1,\ldots,n_k}$, we introduce the following definition.

\begin{definition}\label{dfdo}Define the distance function on $\Omega$ $D_{\Omega}:\Omega\times \Omega\rightarrow \NN$ as follows
\begin{eqnarray*}
D_{\Omega}(x,y)=n-\sum_{i\in[k]}t_{i,i}(x,y)
\end{eqnarray*}
\end{definition}

\begin{proposition}\label{pp35} Assume $\delta\in\left(0,\frac{1}{2}\right)$, such that
\begin{eqnarray}
\log k-\frac{\delta\log n}{1-\delta}\leq -B_0<0\label{lbn}
\end{eqnarray}
and
\begin{eqnarray}
\sigma^2<\frac{(1-\delta)C_0n}{4\log n}\label{ss21}
\end{eqnarray}
where $C_0>0$ is a constant given by
\begin{eqnarray}
C_0=\min_{1\leq i<j\leq k}(c_i-c_j)^2\label{c0}
\end{eqnarray}
Then
\begin{eqnarray}
p(\hat{y};\sigma)\geq 2-e^{\frac{e^{-2B_0}}{(1-e^{-B_0})(1-e^{-2B_0})}};\label{plb}
\end{eqnarray}
where $p(\hat{y};\sigma)$ is given by (\ref{phy}).
\end{proposition}
\begin{proof}Let 
\begin{eqnarray}
I:=\sum_{x\in \Omega_{n_1,\ldots,n_k}\setminus \{y\}}e^{-\frac{M(x,y)}{4\sigma^2}}.\label{di}
\end{eqnarray}
By (\ref{wmp}), it suffices to show that $\lim_{n\rightarrow\infty}I=0$.

We shall find an upper bound for $I$. 

 Recall that $y\in \Omega_{n_1,\ldots,n_k}$ is the true color assignment mapping. Since $x\neq y$, by Lemma \ref{lm33},
 there exists an $l$-cycle $(i_1,\ldots,i_{l})$ for $(x,y)$ with $2\leq l\leq k$. Then for each $2\leq s\leq (l+1)$, choose an arbitrary vertex $u_s$ in $S_{i_{s-1},i_s}(x,y)$, and let $y_1(u_s)=c_{i_{s-1}}$, where $i_{l+1}:=i_1$. For any vertex $z\in[n]\setminus\{u_{2},\ldots,u_{l+1}\}$, let $y_1(z)=y(z)$.

 Note that $y_1\in \Omega_{n_1,\ldots,n_k}$. More precisely, for $1\leq s\leq l$, we have
\begin{eqnarray*}
t_{i_s,i_s}(x,y)+1=t_{i_s,i_s}(x,y_1);\\
t_{i_s,i_{s+1}}(x,y)-1=t_{i_s,i_{s+1}}(x,y_1)
\end{eqnarray*}
and
\begin{eqnarray*}
t_{a,b}(x,y)=t_{a,b}(x,y_1),\ \forall (a,b)\notin\{(i_s,i_s),(i_s,i_{s+1})\}_{s=1}^{l}.
\end{eqnarray*}
From (\ref{mxy}) we obtain
\begin{eqnarray*}
M(x,y_1)-M(x,y)&=&-2n\sum_{a,b\in [k]}\left[t_{a,b}(x,y_1)-t_{a,b}(x,y)\right]c_ac_b\\
&=&-2n\left[\sum_{s\in[l]}(c_{i_s}^2-c_{i_s}c_{i_{s+1}})\right]\\
&=&-n\left[\sum_{s\in[l]}(c_{i_s}-c_{i_{s+1}})^2\right]\\
&\leq &-n lC_0
\end{eqnarray*}
where $C_0$ is given in (\ref{c0}). Therefore
\begin{eqnarray}
e^{-\frac{M(x,y)}{4\sigma^2}}\leq e^{-\frac{M(x,y_1)}{4\sigma^2}}e^{-\frac{nl C_0}{4\sigma^2}}.\label{mst}
\end{eqnarray}

If $y_1\neq x$, we find an $l_2$-cycle ($2\leq l_2\leq k$) for $(x,y_1)$, change colors along the $l_2$-cycle  as above, and obtain another cycle assignment mapping $y_2\in \Omega_{n_1,\ldots,n_k}$, and so on. 
Let $y_0:=y$, and note that for each $r\geq 1$, if $y_r$ is obtained from $y_{r-1}$ by changing colors along an $l_r$ cycle, we have
\begin{eqnarray*}
D_{\Omega}(x,y_r)= D_{\Omega}(x,y_{r-1})-l_r
\end{eqnarray*}
Therefore finally we can obtain $x$ from $y$ by changing colors along at most $\left\lfloor \frac{n}{2} \right\rfloor$ cycles. Using similar arguments as those used to derive (\ref{mst}), we obtain that for each $r$
\begin{eqnarray*}
e^{-\frac{M(x,y_{r-1})}{4\sigma^2}}\leq e^{-\frac{M(x,y_r)}{4\sigma^2}}e^{-\frac{nl_r C_0}{4\sigma^2}}.
\end{eqnarray*}
Therefore if $y_s=x$ for some $1\leq s\leq \left\lfloor \frac{n}{2} \right\rfloor$, we have
\begin{eqnarray*}
e^{-\frac{M(x,y)}{4\sigma^2}}\leq e^{-\frac{M(x,y_s)}{4\sigma^2}}e^{-\frac{nC_0\left(\sum_{i\in[s]}l_i\right)}{4\sigma^2}}.
\end{eqnarray*}
By (\ref{mxy}), we have $M(x,y_s)=M(x,x)=0$, hence
\begin{eqnarray*}
e^{-\frac{M(x,y)}{4\sigma^2}}\leq \prod_{i\in[s]}e^{-\frac{nC_0l_i}{4\sigma^2}}.
\end{eqnarray*}


Note also that for any $r_1\neq r_2$, in the process of obtaining $y_{r_1}$ from $y_{r_1-1}$ and the process of obtaining $y_{r_2}$ from $y_{r_2-1}$, we change colors on disjoint sets of vertices. Hence the order of these steps of changing colors along cycles does not affect the final color assignment mapping we obtain.  From (\ref{di}) we obtain
\begin{eqnarray}
I\leq \prod_{l=2}^k \left(\sum_{m_l=0}^{\infty} (nk)^{m_l\ell} e^{-\frac{nm_ll C_0}{4\sigma^2}}\right)-1\label{iup}
\end{eqnarray}
On the right hand side of (\ref{iup}), when expanding the product, each summand has the form
\begin{eqnarray*}
\left[(nk)^{2m_2}e^{-\frac{2m_2 n C_0}{4\sigma^2}}\right]\cdot\left[ (nk)^{3m_3}e^{-\frac{3m_3 n C_0}{4\sigma^2}}\right]\cdot\ldots\cdot\left[(nk)^{km_k} e^{-\frac{km_k n C_0}{4\sigma^2}}\right]
\end{eqnarray*}
where the factor $\left[(nk)^{2m_2}e^{-\frac{2m_2 n C_0}{4\sigma^2}}\right]$ represents that we changed along 2-cycles $m_2$ times, the factor $\left[ (nk)^{3m_3}e^{-\frac{3m_3 n C_0}{4\sigma^2}}\right]$ represents that we changed along 3-cycles $m_3$ times, and so on. Moreover, each time we changed along an $l$-cycle, we need to first determine the $l$ different colors involved in the $l$-cycle, and there are at most $k^l$ different $l$-cycles;  we then need to chose $l$ vertices to change colors, and there are at most $n^{l}$ choices.
It is straightforward to check that if $\sigma$ satisfies (\ref{ss21}), then
\begin{eqnarray*}
nk e^{-\frac{n C_0}{4\sigma^2}}\leq e^{\log k-\frac{\delta \log n}{1-\delta}}
\end{eqnarray*}
as $n\rightarrow\infty$. Therefore we have
\begin{eqnarray*}
\sum_{m_l=0}^{\infty} (nk)^{m_l\ell} e^{-\frac{nm_ll C_0}{4\sigma^2}}\leq \frac{1}{1-e^{l\log k-\frac{\delta l\log n}{1-\delta}}}
\end{eqnarray*}
Let 
\begin{eqnarray*}
U:= \prod_{l=2}^k \left(\sum_{m_l=0}^{\infty} (nk)^{m_l\ell} e^{-\frac{nm_ll C_0}{4\sigma^2}}\right).
\end{eqnarray*}
Since $\log(1+x)\leq x$ for $x\geq 0$, we have
\begin{eqnarray*}
0\leq \log U&=&\sum_{l=2}^{k}\log \left(1+\sum_{m_l=1}^{\infty} (nk)^{m_l\ell} e^{-\frac{nm_ll C_0}{4\sigma^2}}\right)\\
&\leq &\sum_{l=2}^{k}\sum_{m_l=1}^{\infty} (nk)^{m_l\ell} e^{-\frac{nm_ll C_0}{4\sigma^2}}\\
&\leq &\sum_{l=2}^{k}\frac{\left(e^{\log k-\frac{\delta \log n}{1-\delta}}\right)^{l}}{1-\left(e^{\log k-\frac{\delta \log n}{1-\delta}}\right)^{l}}\\
&\leq &\frac{e^{2\log k-\frac{2\delta \log n}{1-\delta}}}{\left(1-e^{2\log k-\frac{2\delta \log n}{1-\delta}}\right)\left(1-e^{\log k-\frac{\delta \log n}{1-\delta}}\right)}
\end{eqnarray*}
Then the proposition follows from (\ref{lbn}).
\end{proof}

\begin{corollary}Assume $\delta\in\left(0,\frac{1}{2}\right)$, such that
\begin{eqnarray*}
\lim_{n\rightarrow\infty}\log k-\frac{\delta\log n}{1-\delta}=-\infty.
\end{eqnarray*}
Assume (\ref{ss21}) holds. Then
\begin{eqnarray*}
\lim_{n\rightarrow\infty}p(\hat{y};\sigma)=1.
\end{eqnarray*}
\end{corollary}

\begin{proof}The corollary follows from (\ref{plb}) by letting $B_0\rightarrow \infty$.
\end{proof}

\begin{proposition}\label{pp37}Let $C_0>0$ be defined by (\ref{c0}), and $y\in \Omega_{n_1,\ldots,n_k}$ be the true color assignment mapping. Assume there exists $u,v\in [k]$, such that 
\begin{eqnarray*}
(c_u-c_v)^2=C_0,
\end{eqnarray*}
and
\begin{eqnarray}
\frac{\log\min\{n_u,n_v\}}{\log n}\geq \beta>0.\label{ac1}
\end{eqnarray}
Assume 
\begin{eqnarray}
\delta=\left(1+\frac{4}{\beta}\right)\frac{\log\log n}{\log n},\label{ac2}
\end{eqnarray}
and
\begin{eqnarray}
\sigma^2>\frac{(1+\delta)C_0n}{4\beta\log n}\label{sb2}
\end{eqnarray}
then
\begin{eqnarray*}
\lim_{n\rightarrow\infty}p(\hat{y};\sigma)=0,
\end{eqnarray*}
where $\beta$ is a constant independent of $n$ as $n\rightarrow\infty$, while the total number of colors $k$, the set of colors $R_k$, and $C_0$ may depend on $n$.
\end{proposition}

\begin{proof}

For $y\in \Omega_{n_1,\ldots,n_k}$, $a,b\in[n]$ such that $c_u=y(a)\neq y(b)=c_v$. Let $y^{(ab)}\in \Omega_{n_1,\ldots,n_k}$ be the coloring of vertices defined by
\begin{eqnarray}
y^{(ab)}(i)=\begin{cases}y(i)& \mathrm{if}\ i\in[n]\setminus\{a,b\}\\ c_v&\mathrm{if}\ i=a\\ c_u& \mathrm{if}\ i=b \end{cases}\label{yab}
\end{eqnarray}
Then 
\begin{eqnarray*}
t_{u,v}(y^{(ab)},y)-1=t_{u,v}(y,y)=0\\
t_{u,u}(y^{(ab)},y)+1=t_{u,u}(y,y)=n_u\\
t_{v,u}(y^{(ab)},y)-1=t_{v,u}(y,y)=0\\
t_{v,v}(y^{(ab)},y)+1=t_{v,v}(y,y)=n_v.
\end{eqnarray*}
and
\begin{eqnarray*}
t_{i,j}(y^{(ab)},y)=t_{i,j}(y),\ \forall\ (i,j)\in \left([k]^2\setminus\{(u,u),(u,v),(v,u),(v,v)\}\right)
\end{eqnarray*}
Note that 
\begin{eqnarray*}
1-p(\hat{y};\sigma)\geq \mathrm{Pr}\left(\cup_{a,b\in[n],c_u=y(a)\neq y(b)=c_v}[f(y^{(ab)})-f(y)>0]\right),
\end{eqnarray*}
since any of the event $[f(y^{(ab)})-f(y)>0]$ implies $\hat{y}\neq y$. By (\ref{mxy}) we obtain
\begin{eqnarray*}
f(y^{(ab)})-f(y)&=&\langle \mathbf{G}(y),\mathbf{G}(y^{(ab)})-\mathbf{G}(y)\rangle+\sigma\langle\mathbf{W},\mathbf{G}(y^{(ab)})-\mathbf{G}(y) \rangle\\
&=&-2n(c_u-c_v)^2+\sigma\langle\mathbf{W},\mathbf{G}(y^{(ab)})-\mathbf{G}(y) \rangle.
\end{eqnarray*}
So $1-p(\hat{y};\sigma)$ is at least
\begin{eqnarray*}
&&\mathrm{Pr}\left(\cup_{a,b\in[n],c_u=y(a)\neq y(b)=c_v}[f(y^{(ab)})-f(y)>0]\right)\\
&\geq &\mathrm{Pr}\left(\mathrm{max}_{a,b\in[n],c_u=y(a)\neq y(b)=c_v}\sigma\langle\mathbf{W},\mathbf{G}(y^{(ab)})-\mathbf{G}(y) \rangle>2n C_0\right)
\end{eqnarray*}
For $i\in\{u,v\}$, let $H_i\subset y^{-1}(c_i)$ such that 
\begin{eqnarray}
|H_i|=\frac{\min\{n_u,n_v\}}{\log^2n}=h. \label{ddh}
\end{eqnarray}
Then
\begin{eqnarray*}
1-p(\hat{y};\sigma)\geq \mathrm{Pr}\left(\mathrm{max}_{a\in H_u,b\in H_v}\sigma\langle\mathbf{W},\mathbf{G}(y^{(ab)})-\mathbf{G}(y) \rangle>2n C_0\right)
\end{eqnarray*}
Let $(\mathcal{X},\mathcal{Y},\mathcal{Z})$ be a partition of $[n]^2$ defined by
\begin{eqnarray*}
&&\mathcal{X}=\{\alpha=(\alpha_1,\alpha_2)\in [n]^2, \{\alpha_1,\alpha_2\}\cap [H_u\cup H_v]=\emptyset\}\\
&&\mathcal{Y}=\{\alpha=(\alpha_1,\alpha_2)\in [n]^2, |\{\alpha_1,\alpha_2\}\cap [H_u\cup H_v]|=1\}\\
&&\mathcal{Z}=\{\alpha=(\alpha_1,\alpha_2)\in [n]^2, |\{\alpha_1,\alpha_2\}\cap [H_u\cup H_v]|=2\}
\end{eqnarray*}
For $\eta\in\{\mathcal{X},\mathcal{Y},\mathcal{Z}\}$, define the $n\times n$ matrix $\mathbf{W}_{\eta}$ from the entries of $\mathbf{W}$ as follows
\begin{eqnarray*}
\mathbf{W}_{\eta}(i,j)=\begin{cases}0&\mathrm{if}\ (i,j)\notin \eta\\ \mathbf{W}(i,j),&\mathrm{if}\ (i,j)\in \eta\end{cases}
\end{eqnarray*}
For each $a\in H_u$ and $b\in H_v$, let
\begin{eqnarray*}
\mathcal{X}_{ab}=\langle\mathbf{W}_{\mathcal{X}},\mathbf{G}(y^{(ab)})-\mathbf{G}(y) \rangle\\
\mathcal{Y}_{ab}=\langle\mathbf{W}_{\mathcal{Y}},\mathbf{G}(y^{(ab)})-\mathbf{G}(y) \rangle\\
\mathcal{Z}_{ab}=\langle\mathbf{W}_{\mathcal{Z}},\mathbf{G}(y^{(ab)})-\mathbf{G}(y) \rangle
\end{eqnarray*}

\begin{lemma}\label{l38}The followings are true:
\begin{enumerate}
\item $\mathcal{X}_{ab}=0$ for $a\in H_u$ and $b\in H_v$.
\item For each $a\in H_u$ and $b\in H_v$, the variables $\mathcal{Y}_{ab}$ and $\mathcal{Z}_{ab}$ are independent.
\item Each $\mathcal{Y}_{ab}$ can be decomposed into $Y_a+Y_b$ where $\{Y_a\}_{a\in H_u}\cup \{Y_b\}_{b\in H_v}$ is a collection of i.i.d.~Gaussian random variables.
\end{enumerate}
\end{lemma}

\begin{proof}
Note that for $i,j\in[n]$,
\begin{eqnarray}
\mathbf{G}_{i,j}(y^{(ab)})-\mathbf{G}_{i,j}(y)=\begin{cases}c_v-c_u &\mathrm{if}\ i=a, j\notin\{a,b\}\\ c_u-c_v&\mathrm{if}\ i\notin\{a,b\}, j=a\\ c_u-c_v&\mathrm{if}\ i=b, j\notin\{a,b\}\\ c_v-c_u & \mathrm{if}\ i\notin\{a,b\}, j=b\\2(c_v-c_u) &\mathrm{if}\ (i,j)=(a,b)\\ 2(c_u-c_v)&\mathrm{if}\ (i,j)=(b,a)\\ 0&\mathrm{otherwise}.\end{cases}\label{gab}
\end{eqnarray}

It is straightforward to check (1). (2) holds because $\mathcal{Y}\cap\mathcal{Z}=\emptyset$.

For $s\in H_u\cup H_v$, let $\mathcal{Y}_s\subseteq \mathcal{Y}$ be defined by
\begin{eqnarray*}
\mathcal{Y}_s=\{\alpha=(\alpha_1,\alpha_2)\in \mathcal{Y}:\alpha_1=s,\ \mathrm{or}\ \alpha_2=s\}.
\end{eqnarray*}
Note that for $s_1,s_2\in H_u\cup H_v$ and $s_1\neq s_2$, $\mathcal{Y}_{s_1}\cap \mathcal{Y}_{s_2}=\emptyset$. Moreover, $\mathcal{Y}=\cup_{s\in H_u\cup H_v}\mathcal{Y}_s$. Therefore
\begin{eqnarray*}
\mathcal{Y}_{ab}=\sum_{s\in H_u\cup H_v}\langle\mathbf{W}_{\mathcal{Y}_s},\mathbf{G}(y^{(ab)})-\mathbf{G}(y) \rangle
\end{eqnarray*}
Note also that $\langle\mathbf{W}_{\mathcal{Y}_s},\mathbf{G}(y^{(ab)})-\mathbf{G}(y) \rangle=0$, if $s\notin \{a,b\}$. Hence
\begin{eqnarray*}
\mathcal{Y}_{ab}=\sum_{\alpha\in\mathcal{Y}_a\cup \mathcal{Y}_b}[\mathbf{W}(\alpha)]\cdot\{[\mathbf{G}(y^{(ab)})-\mathbf{G}(y)](\alpha)\}
\end{eqnarray*}
From (\ref{gab}) we obtain that for $\alpha=(\alpha_1,\alpha_2)\in \mathcal{Y}_a$ and $\alpha_1\neq \alpha_2$,
\begin{eqnarray*}
[\mathbf{G}(y^{(ab)})-\mathbf{G}(y)](\alpha)=\begin{cases}c_v-c_u &\mathrm{if}\ \alpha_1=a\\c_u-c_v &\mathrm{if}\  \alpha_2=a.\end{cases}
\end{eqnarray*}
So, we can define
\begin{eqnarray*}
Y_a:=&&\sum_{\alpha\in \mathcal{Y}_a}[\mathbf{W}(\alpha)]\cdot\{[\mathbf{G}(y^{(ab)})-\mathbf{G}(y)](\alpha)\}\\
&=&\left\{\sum_{\alpha\in \mathcal{Y}_a;\alpha_1=a}[\mathbf{W}(\alpha)]-\sum_{\alpha\in \mathcal{Y}_a;\alpha_2=a}[\mathbf{W}(\alpha)]\right\}(c_v-c_u)\\
\end{eqnarray*}
Similarly, define
\begin{eqnarray*}
Y_b:=\left\{\sum_{\alpha\in \mathcal{Y}_b;\alpha_2=b}[\mathbf{W}(\alpha)]
-\sum_{\alpha\in \mathcal{Y}_b;\alpha_1=b}[\mathbf{W}(\alpha)]\right\}(c_v-c_u)
\end{eqnarray*}
Then $\mathcal{Y}_{ab}=Y_a+Y_b$ and $\{Y_s\}_{s\in H_u\cup H_v}$ is a collection of independent Gaussian random variables. Moreover, the variance of $Y_s$ is equal to $(2n-4h)C_0$ independent of the choice of $s$.
\end{proof}

By the Lemma \ref{l38}, we obtain
\begin{eqnarray*}
\langle\mathbf{W},\mathbf{G}(y^{(ab)})-\mathbf{G}(y) \rangle=Y_a+Y_b+\mathcal{Z}_{ab}
\end{eqnarray*}
Moreover,
\begin{eqnarray*}
\max_{a\in H_u,b\in H_v}Y_a+Y_b+\mathcal{Z}_{ab}&\geq& \max_{a\in H_u,b\in H_v}(Y_a+Y_b)-\max_{a\in H_u,b\in H_v}(-\mathcal{Z}_{ab})\\
&=&\max_{a\in H_u} Y_a+\max_{b\in H_v}Y_b-\max_{a\in H_u,b\in H_v}(-\mathcal{Z}_{ab})
\end{eqnarray*}

Recall the following tail bound result on the maximum of Gaussian random variables:

\begin{lemma}\label{mg}Let $G_1,\ldots, G_N$ be Gaussian random variables with variance $1$. Let $\epsilon\in (0,1)$. Then 
\begin{eqnarray*}
\mathrm{Pr}\left(\max_{i=1,\ldots,N}G_i>(1+\epsilon)\sqrt{2\log N}\right)\leq N^{-\epsilon}
\end{eqnarray*}
and moreover, if $G_i$'s are independent, and $\epsilon, N$ satisfy
\begin{eqnarray}
\frac{N^{\epsilon-\epsilon^2}(1-\epsilon)\sqrt{2\log N}}{\sqrt{2\pi}(1+2(1-\epsilon)^2\log N)}>1\label{epn}
\end{eqnarray}
Then
\begin{eqnarray*}
\mathrm{Pr}\left(\max_{i=1,\ldots,N}G_i<(1-\epsilon)\sqrt{2\log N}\right)\leq \exp(-N^{\epsilon})
\end{eqnarray*}
\end{lemma}

The proof of Lemma \ref{mg} is in the appendix.

By Lemma \ref{mg} we obtain that  when $\epsilon, h$ satisfy (\ref{epn}) with $N$ replaced by $h$, each one of the following two events
\begin{eqnarray*}
E_1:=\left\{\max_{a\in H_u}Y_a\geq (1-\epsilon)\sqrt{2\log h\cdot 2C_0\left(n-2h\right)}\right\}\\
E_2:=\left\{\max_{b\in H_v}Y_b\geq (1-\epsilon)\sqrt{2\log h\cdot 2C_0\left(n-2h\right)}\right\}
\end{eqnarray*}
has probability at least $1-e^{-h^{\epsilon}}$. Moreover, the event 
\begin{eqnarray*}
E_3:=\left\{\max_{a\in H_u,b\in H_v}\mathcal{Z}_{ab}\leq (1+\epsilon)\sqrt{4\log h\cdot \max \mathrm{Var}(Z_{ab})}\right\}
\end{eqnarray*}
occurs with probability at least $1-h^{-2\epsilon}$. Then by (\ref{gab}) we have
\begin{eqnarray*}
\mathrm{Var} \mathcal{Z}_{ab}&=&\|\mathbf{G}(y^{(ab)})-\mathbf{G}(y)\|^2_{F}-\mathrm{Var}(Y_a)-\mathrm{Var}(Y_b)\\
&=&4nC_0-4C_0\left(n-2h\right)\\
&=& 8hC_0
\end{eqnarray*}
Hence the probability of the event
\begin{eqnarray*}
E:=\left\{\max_{a\in H_u,b\in H_v}\langle\mathbf{W},\mathbf{G}(y^{(ab)})-\mathbf{G}(y) \rangle\geq 4(1-\epsilon)\sqrt{\log h C_0(n-2h)}-4(1+\epsilon)\sqrt{2C_0h\log h}\right\}
\end{eqnarray*}
is at least 
\begin{eqnarray*}
\mathrm{Pr}(E_1\cap E_2\cap E_3)&=&1-\mathrm{Pr}(E_1^c\cup E_2^c\cup E_3^c)\\
&\geq &1- \mathrm{Pr}(E_1^c)-\mathrm{Pr}(E_2^c)-\mathrm{Pr}(E_3^c)\\
&\geq &1-2e^{-h^{\epsilon}}-h^{-2\epsilon}.
\end{eqnarray*}
Moreover, from (\ref{ddh}) we obtain
\begin{eqnarray*}
&&4(1-\epsilon)\sqrt{C_0(n-2h)\log h}-4(1+\epsilon)\sqrt{2C_0h\log h}\\
&=& 4\sqrt{C_0(n-2h)\log h}\left[1-\epsilon-(1+\epsilon)\sqrt{\frac{2h}{n-2h}}\right]\\
&\geq &4\sqrt{C_0(n-2h)\log h}\left[1-\epsilon-\frac{2}{\log n}\right]
\end{eqnarray*}
By (\ref{ac1}) we have
\begin{eqnarray*}
&&4\sqrt{C_0(n-2h)\log h}\left[1-\epsilon-\frac{2}{\log n}\right]\\
&\geq &4\sqrt{C_0n\beta\log n\left(1-\frac{2}{\log^2n}\right)\left(1-\frac{2\log\log n}{\beta\log n}\right)}\left[1-\epsilon-\frac{2}{\log n}\right]
\end{eqnarray*}
Let 
\begin{eqnarray}
\epsilon=\frac{\log\log n}{\beta\log n};\label{eph}
\end{eqnarray}
then when $n$ is sufficiently large, (\ref{epn}) holds with $N$ replaced by $h$.
Define an event
\begin{eqnarray*}
\tilde{E}:&=&\left\{\max_{a\in H_u,b\in H_v}\langle\mathbf{W},\mathbf{G}(y^{(ab)})-\mathbf{G}(y) \rangle\right.\\
&&\left.\geq4\sqrt{C_0n\beta\log n\left(1-\frac{2}{\log^2n}\right)\left(1-\frac{2\log\log n}{\beta\log n}\right)}\left[1-\frac{\log\log n}{\beta\log n}-\frac{2}{\log n}\right]\right\}
\end{eqnarray*}
Then $E\subseteq \tilde{E}$

When (\ref{ac2}) and (\ref{sb2}) hold, we have
\begin{eqnarray*}
&&\mathrm{Pr}\left(\mathrm{max}_{a,b\in\{1,2,\ldots,n\},y(a)\neq y(b)}\sigma\langle\mathbf{W},\mathbf{G}(y^{(ab)})-\mathbf{G}(y) \rangle>2C_0n\right)\\
&\geq &\mathrm{Pr}(\tilde{E})\geq \mathrm{Pr}(E)\geq 1-\frac{1}{\log n},
\end{eqnarray*}
as $n$ is sufficiently large. Then the proposition follows.
\end{proof}

\section{Proof of Theorem \ref{m1} when the number of vertices in each color is arbitrary in the sample space}\label{pm12}

Now we consider the MLE in Theorem \ref{m1} whose sample space consists of all the possible mappings from $[n]$ to $R_k$, with no constraints on the number of vertices in each color. Assume that for each $x\in\Omega$, $n_1(x), \ldots,  n_k(x)$ are arbitrary positive integers satisfying (\ref{sc}) and denoting the number of vertices in the colors $c_1,\ldots, c_k$ under the mapping $x$, respectively. For a mapping $x\in \Omega$, let $\mathbf{G}(x)$ be defined as in (\ref{gij}). Let $y$ be the true color assignment mapping, and let $\mathbf{T}$ be defined as in (\ref{tgw}).

Given a sample $\mathbf{T}$, the goal is to determine the color assignment mapping $y$. Let $\check{y}$ be defined by (\ref{cy}).
Then
\begin{eqnarray*}
\check{y}=\mathrm{argmin}_{x\in \Omega}\left(\|\mathbf{G}(x)\|_F^2-2\langle \mathbf{G}(x),\mathbf{T} \rangle\right)
\end{eqnarray*}
Let 
\begin{eqnarray*}
p(\check{y},\sigma)=\mathrm{Pr}(\check{y}=y)
\end{eqnarray*}
For $x\in \Omega$, define
\begin{eqnarray*}
d(x)=\|\mathbf{G}(x)\|_F^2-2\langle \mathbf{G}(x),\mathbf{T} \rangle
\end{eqnarray*}

 Then 
\begin{eqnarray*}
p(\check{y},\sigma)=\mathrm{Pr}\left(d(y)<\min_{x\in \Omega \setminus \{y\}}d(x)\right)
\end{eqnarray*}
Note that
\begin{eqnarray}
&&d(x)-d(y)=\label{dmxy}\\
&&\|\mathbf{G}(x)\|_F^2-\|\mathbf{G}(y)\|_F^2-2\langle\mathbf{G}(y), \mathbf{G}(x)-\mathbf{G}(y) \rangle-2\sigma\langle\mathbf{W}, \mathbf{G}(x)-\mathbf{G}(y) \rangle.\notag
\end{eqnarray}
The expression above shows that $d(x)-d(y)$ is a Gaussian random variable with mean $\|\mathbf{G}(x)\|_F^2-\|\mathbf{G}(y)\|_F^2-2\langle\mathbf{G}(y), \mathbf{G}(x)-\mathbf{G}(y) \rangle$ and variance $4\sigma^2\|\mathbf{G}(x)-\mathbf{G}(y)\|_F^2$.

For $i,j\in [k]$, let $S_{i,j}(x,y)$ be defined as in (\ref{1sij}), and  let $t_{i,j}(x,y)=|S_{i,j}(x,y)|$. 
Then
\begin{eqnarray}
&&\sum_{j\in[k]}t_{i,j}(x,y)=n_i(x),\ \ \forall i\in[k]\label{tt1}\\
&&\sum_{i\in[k]}t_{i,j}(x,y)=n_j(y),\ \ \forall j\in[k]\label{tt2}
\end{eqnarray}

Then as in (\ref{pgxy}),
\begin{eqnarray*}
\langle\mathbf{G}(x),\mathbf{G}(y) \rangle
&=&\sum_{i,j,u,v\in [k]}t_{i,j}t_{u,v}(c_i-c_u)(c_j-c_v)\\
&=&2\left[\sum_{i,j\in [k]} t_{i,j}\cdot c_i\cdot c_j\right]\left[\sum_{u,v\in[k]}t_{u,v}\right]-2\left[\sum_{i,j\in [k]} t_{i,j}\cdot c_i \right]\left[\sum_{u,v\in[k]}t_{u,v}\cdot c_v\right]\\
&=&2n\left[\sum_{i,j\in [k]} t_{i,j}\cdot c_i\cdot c_j\right]-2\left[\sum_{i\in [k]} n_i(x)\cdot c_i \right]\left[\sum_{j\in [k]} n_j(y)\cdot c_j \right],
\end{eqnarray*}
where the last identity follows from (\ref{tt1}) and (\ref{tt2}).

In particular
\begin{eqnarray*}
\langle \mathbf{G}(y),\mathbf{G}(y)\rangle=2n \sum_{i\in [k]}\left[n_i(y) \cdot c_i^2\right]-2\left[\sum_{i\in [k]}n_i(y)\cdot c_i\right]^2
\end{eqnarray*}
Let
\begin{eqnarray}
Q(x,y):&=&\mathbb{E}[d(x)-d(y)]=\|\mathbf{G}(x)\|_F^2+\|\mathbf{G}(y)\|_F^2-2\langle \mathbf{G}(x),\mathbf{G}(y) \rangle\label{qxy}\\
&=&2n \sum_{i,j\in [k]}t_{i,j}(x,y)(c_i-c_j)^2-2\left[\sum_{i\in[k]}n_i(x)c_i-\sum_{j\in[k]}n_j(y)c_j\right]^2\notag\\
&=&\sum_{u,v,i,j\in[k]}t_{u,v}(x,y)t_{i,j}(x,y)(c_u-c_v-c_i+c_j)^2\notag
\end{eqnarray}
Then
\begin{eqnarray*}
\mathbf{Var}[d(x)-d(y)]=4\sigma^2 Q(x,y)
\end{eqnarray*}
For $x\in \Omega \setminus \{y\}$
\begin{eqnarray*}
\mathrm{Pr}\left(d(y)-d(x)>0\right)
&=&\mathrm{Pr_{\xi\sim \mathcal{N}(0,1)}}\left(\xi\geq \frac{\sqrt{Q(x,y)}}{2\sigma}\right)
\end{eqnarray*}
Using the standard Gaussian tail bound $\mathrm{Pr}_{\xi\in \mathcal{N}(0,1)}(\xi>x)<e^{-\frac{1}{2}x^2}$, we obtain
\begin{eqnarray*}
\mathrm{Pr}\left(d(y)-d(x)>0\right)\leq e^{-\frac{Q(x,y)}{8\sigma^2}}
\end{eqnarray*}

Then
\begin{eqnarray}
1-p(\check{y};\sigma)&\leq& \sum_{x\in \Omega\setminus \{y\}}\mathrm{Pr}(d(y)-d(x)> 0)\label{1mp}\\
&=&\sum_{x\in \Omega \setminus \{y\}}\mathrm{Pr_{\xi\sim \mathcal{N}(0,1)}}\left(\xi> \frac{\sqrt{Q(x,y)}}{2\sigma}\right)\notag\\
&\leq &\sum_{x\in \Omega \setminus \{y\}} e^{-\frac{Q(x,y)}{8\sigma^2}}\notag
\end{eqnarray}

\begin{lemma}
Let $y\in \Omega$ be the fixed true color assignment mapping, and assume that $x$ changes over $\Omega$. Under the constraint (\ref{tt2}), $Q(x,y)$ achieves its minimum if
\begin{eqnarray*}
&&t_{i,i}=n_i(y)\\
&&t_{i,j}=0,\qquad\mathrm{if}\ i\neq j.
\end{eqnarray*}
and the minimal value of $Q(x,y)$ is 0.
\end{lemma}
\begin{proof}The lemma follows from (\ref{qxy}).
\end{proof}

Let $\mathcal{B}$ be the set given by
\begin{eqnarray*}
\mathcal{B}=\left\{(t_{1,1},t_{2,1}\ldots,t_{k,k})\in \prod_{j\in[k]}\{0,1,\ldots,n_j(y)\}^k:\sum_{i=1}^{k}t_{i,j}(x,y)=n_j(y)\right\}.
\end{eqnarray*}

For a small positive number $\epsilon>0$, let $\mathcal{B}_{\epsilon}$ be the domain given by
\begin{eqnarray*}
\mathcal{B}_{\epsilon}=\left\{(t_{1,1},\ldots,t_{k,k})\in \mathcal{B}: t_{i,i}\geq n_i(y)-n\epsilon,\ \forall\ i\in[k]\right\}
\end{eqnarray*}

 Then
\begin{eqnarray*}
\sum_{x\in \Omega\setminus\{y\}} e^{-\frac{Q(x,y)}{8\sigma^2}}
=  J_1+J_2
\end{eqnarray*}
where 
\begin{eqnarray*}
J_1=\sum_{x\in \Omega\setminus\{y\}:\left(t_{1,1},\ldots, t_{k,k}\right)\in\left[\mathcal{B}\setminus\mathcal{B}_{\epsilon}\right]}e^{-\frac{Q(x,y)}{8\sigma^2}}.
\end{eqnarray*}
and
\begin{eqnarray*}
J_2=\sum_{x\in \Omega\setminus\{y\}:\left(t_{1,1},\ldots, t_{k,k}\right)\in\mathcal{B}_{\epsilon}}e^{-\frac{Q(x,y)}{8\sigma^2}}.
\end{eqnarray*}

Fix a constant $c>0$. Define a region
\begin{eqnarray}
\mathcal{R}_c:=\{(v_1,\ldots,v_k)\in \RR^{k}: \sum_{i\in[k]}v_k=1,\ \mathrm{and}\ \min_{i\in[k]} v_i\geq c\}.\label{drc}
\end{eqnarray}

\begin{lemma}\label{l32}Assume
\begin{eqnarray}
\left(\frac{n_1(y)}{n},\ldots,\frac{n_k(y)}{n}\right)\in \mathcal{R}_c \label{vir} 
\end{eqnarray}
and 
\begin{eqnarray*}
(t_{1,1}(x,y),\ldots, t_{k,k}(x,y))\in \mathcal{B}\setminus \mathcal{B}_{\epsilon}. 
\end{eqnarray*}
Then if $\epsilon>0$ is small enough,
\begin{eqnarray}
Q(x,y)\geq 4C_0\epsilon^2 n^2\label{qlb}
\end{eqnarray}
where $C_0>0$ is given by (\ref{c0}).

\end{lemma}
\begin{proof}When $(t_{1,1}(x,y),\ldots, t_{k,k}(x,y))\in \mathcal{D}\setminus \mathcal{D}_{\epsilon}$, we have there exists $1\leq i\leq k$,
\begin{eqnarray*}
\sum_{j\in[k],j\neq i}t_{j,i}(x,y)\geq n\epsilon 
\end{eqnarray*}

The following cases might occur
\begin{enumerate}
\item there exists $l\in [k]$, such that
\begin{eqnarray*}
t_{l,l}(x,y)>n_{l}(y)-n\epsilon
\end{eqnarray*}
Then by (\ref{qxy})
\begin{eqnarray*}
Q(x,y)&=&\sum_{u,v,i,j\in[k]}t_{u,v}(x,y)t_{i,j}(x,y)(c_u-c_v-c_i+c_j)^2\\
&\geq& t_{l,l}(x,y)\sum_{j\in [k],j\neq i}t_{j,i}(x,y)(c_i-c_j)^2\\
&\geq & (n_{l}(y)-n\epsilon)\epsilon C_0 n\\
&\geq &(c-\epsilon)\epsilon C_0 n^2,
\end{eqnarray*}
where the last inequality follows from (\ref{vir}). Then (\ref{qlb}) holds when $\epsilon\leq \frac{c}{5}$.
\item For each $b\in [k]$, we have
\begin{eqnarray*}
\sum_{a\in [k],a\neq b}u_{a,b}\geq \epsilon
\end{eqnarray*}
Without loss of generality, assume that $c_1>c_2>\ldots>c_k$. Again by (\ref{qxy})
\begin{eqnarray*}
Q(x,y)&=&\sum_{u,v,i,j\in[k]}t_{u,v}(x,y)t_{i,j}(x,y)(c_u-c_v-c_i+c_j)^2\\
&\geq &\sum_{i\in[k],i\neq 1}\sum_{j\in [k],j\neq k}t_{i,1}(x,y)t_{j,k}(x,y)(c_i-c_1-c_j+c_k)^2\\
&\geq & 4 C_0\epsilon^2 n^2
\end{eqnarray*}
\end{enumerate}
\end{proof}

\begin{proposition}\label{pp43}Let $y$ be the true color assignment mapping. Assume all the following conditions hold:
\begin{itemize}
\item (\ref{vir}) holds; 
\item The total number of colors $k$ and the total number of vertices $n$ satisfy
\begin{eqnarray}
\lim_{n\rightarrow\infty}\frac{\log k}{\log n}=0.\label{lkln}
\end{eqnarray}
\item The quotient of the maximal difference of two colors and the minimal difference of two colors is uniformly bounded for all $n$; that is,
\begin{eqnarray}
\sup_{n}\frac{\max_{u,v\in [k]}|c_u-c_v|}{\min_{u,v\in [k],u\neq v}|c_u-c_v|}<\infty\label{qqb}
\end{eqnarray}
\item there exists $\delta>0$ (independent of $n$), such that
\begin{eqnarray}
\sigma^2<\frac{(1-\delta)C_0n}{4\log n}\label{ss2}
\end{eqnarray}
where $C_0>0$ is defined by (\ref{c0}).
\end{itemize}
 Then
\begin{eqnarray*}
\lim_{n\rightarrow\infty}p(\check{y};\sigma)=1.
\end{eqnarray*}
\end{proposition}
\begin{proof}First of all, for any fixed $\epsilon>0$, if (\ref{ss2}) holds, then $\sigma\sim o(\sqrt{n})$, by (\ref{vir}) and Lemma \ref{l32}
we obtain that 
\begin{eqnarray*}
J_1\leq  \sum_{x\in\Omega\setminus\{y\}}e^{-\frac{C_0\epsilon^2 n^2}{2\sigma^2}}
\end{eqnarray*}
Note that
\begin{eqnarray*}
|\Omega|\leq k^n
\end{eqnarray*}
Hence
\begin{eqnarray*}
J_1\leq e^{n\left(\log k-\frac{2\epsilon^2\log n}{1-\delta}\right)}\rightarrow 0
\end{eqnarray*}
as $n\rightarrow\infty$, for all $\epsilon >0$ by (\ref{lkln}).

Now let us consider $J_2$.
Let $x\in \Omega$ be an arbitrary color assignment mapping. Then $x$ can be obtained from $y$ as follows:
\begin{enumerate}
\item If for all $(i,j)\in [k]^2$, $i\neq j$, $t_{i,j}(x,y)=0$, then $x=y$.
\item If $x\neq y$, find the least $(a,b)\in[k]^2$ in lexicographic order such that $a\neq b$ and $t_{a,b}(x,y)>0$. Arbitrarily choose a vertex $u$ in $y^{-1}(c_b)$ and define $y_1\in \Omega$ by
\begin{eqnarray*}
y_1(z):=\begin{cases}c_a&\mathrm{if}\ z=u\\ y(z)&\mathrm{if}\ z\in[n]\setminus u \end{cases}
\end{eqnarray*}
\end{enumerate}
Then we obtain 
\begin{eqnarray*}
&&t_{a,b}(x,y_1)+1=t_{a,b}(x,y)\\
&&t_{a,a}(x,y_1)-1=t_{a,a}(x,y)\\
&&t_{i,j}(x,y_1)=t_{i,j}(x,y),\ \forall\ (i,j)\in [k]^2\setminus \{(a,b),(a,a)\}
\end{eqnarray*}

Recall from (\ref{qxy}) that 
\begin{eqnarray*}
Q(x,y)=2n\sum_{i,j\in [k]}t_{i,j}(x,y)(c_i-c_j)^2-2\left[\sum_{i,j\in[k]}t_{i,j}(x,y)(c_i-c_j)\right]^2;
\end{eqnarray*}
hence we have
\begin{eqnarray*}
&&Q(x,y_1)-Q(x,y)\\
&=&-2n(c_a-c_b)^2+2\left[\sum_{i,j\in[k]}t_{i,j}(x,y)(c_u-c_v)\right]^2-2\left[\sum_{i,j\in[k]}t_{i,j}(x,y)(c_u-c_v)-(c_a-c_b)\right]^2\\
&=&-(2n+2)(c_a-c_b)^2+4(c_a-c_b)\left[\sum_{i,j\in[k]}t_{i,j}(x,y)(c_i-c_j)\right]\\
&\leq &-(2n+2)C_0+C_1\epsilon n,
\end{eqnarray*}
where $C_1>0$ is a constant given by
\begin{eqnarray*}
C_1:=4 k\sqrt{C_0}\max_{p,q\in [k]}|c_p-c_q|
\end{eqnarray*}
Therefore
\begin{eqnarray*}
e^{-\frac{Q(x,y)}{8\sigma^2}}\leq e^{-\frac{Q(x,y_1)}{8\sigma^2}}e^{-\frac{n}{8\sigma^2}\left[2C_0+\frac{2}{n}C_0-C_1\epsilon \right]}
\end{eqnarray*}

Recall the distance in $\Omega$ was defined by (\ref{dfdo}). Note that
\begin{eqnarray*}
d_{\Omega}(x,y_1)=d_{\Omega}(x,y)-1
\end{eqnarray*}

In general for $r\geq 1$, if we obtained $y_r$, we can obtain $y_{r+1}$ as follows:
\begin{enumerate}
\item If for all $(i,j)\in [k]^2$, $i\neq j$, $t_{i,j}(x,y)=0$, then $x=y_r$.
\item If $x\neq y_r$, find the least $(a,b)\in[k]^2$ in lexicographic order such that $a\neq b$ and $t_{a,b}(x,y_r)>0$. Arbitrarily choose a vertex $u$ in $y_r^{-1}(c_b)$ and define $y_{r+1}\in \Omega$ by
\begin{eqnarray*}
y_{r+1}(z):=\begin{cases}c_a&\mathrm{if}\ z=u\\ y_r(z)&\mathrm{if}\ z\in[n]\setminus u \end{cases}
\end{eqnarray*}
\end{enumerate}

Then by the same arguments as before, we have
\begin{eqnarray*}
d_{\Omega}(x,y_{r+1})=d_{\Omega}(x,y_r)-1
\end{eqnarray*}
and
\begin{eqnarray}
e^{-\frac{Q(x,y_r)}{8\sigma^2}}\leq e^{-\frac{Q(x,y_{r+1})}{8\sigma^2}}e^{-\frac{n}{8\sigma^2}\left[2C_0+\frac{2}{n}C_0-C_1\epsilon \right]}\label{idc}
\end{eqnarray}

By (\ref{dfdo}), we have for any $x,y\in \Omega$, $D_{\Omega}(x,y)\leq n$. Therefore if $x\neq y$, there exists $1\leq l\leq n$, such that $y_s=x$. By (\ref{idc}) we have
\begin{eqnarray*}
e^{-\frac{Q(x,y)}{8\sigma^2}}&\leq& e^{-\frac{Q(x,y_{l})}{8\sigma^2}}e^{-\frac{nl}{8\sigma^2}\left[2C_0+\frac{2}{n}C_0-C_1\epsilon \right]}\\
&=&e^{-\frac{nl}{8\sigma^2}\left[2C_0+\frac{2}{n}C_0-C_1\epsilon \right]};
\end{eqnarray*}
where the last identity follows from (\ref{qxy}) and 
\begin{eqnarray*}
e^{-\frac{Q(x,y_{l})}{8\sigma^2}}=e^{-\frac{Q(x,x)}{8\sigma^2}}=1.
\end{eqnarray*}

Therefore
\begin{eqnarray}
J_2\leq \sum_{l=1}^{\infty} (nk)^{\ell} e^{-\frac{[(2n+2)C_0-C_1\epsilon n]l}{8\sigma^2}}\label{j2u}
\end{eqnarray}
In the sum, $l$ represents $D_{\Omega}(x,y)$. If $D_{\Omega}(x,y)=l$, then $x$ can be obtained from $y$ by changing colors at $l$ vertices, each time there are at most $n$ choices of vertices, and the color of each chosen vertex is changed to one of the $k$ choices of colors. That is where the factor $(nk)^l$ on the right hand side comes from.

When $\sigma$ satisfies (\ref{ss2}), we have
\begin{eqnarray*}
(nk) e^{-\frac{[(2n+2)C_0-C_1\epsilon n]}{8\sigma^2}}<e^{\log k-\log n\left(\frac{\delta}{1-\delta}+\frac{1}{n(1-\delta)}-\frac{C_1\epsilon}{2(1-\delta) C_0}\right)}
\end{eqnarray*}
Then for (\ref{j2u}) we obtain
\begin{eqnarray*}
J_2\leq \frac{e^{\log k-\log n\left(\frac{\delta}{1-\delta}+\frac{1}{n(1-\delta)}-\frac{C_1\epsilon}{2(1-\delta) C_0}\right)}}{1-e^{\log k-\log n\left(\frac{\delta}{1-\delta}+\frac{1}{n(1-\delta)}-\frac{C_1\epsilon}{2(1-\delta) C_0}\right)}}
\end{eqnarray*}

By (\ref{lkln}) and (\ref{qqb}), we can choose $0<\epsilon<\frac{2C_0 \delta}{C_1}$ independent of $n$, then $\lim_{n\rightarrow\infty}J_2=0$.
 Then the proposition follows from (\ref{1mp}) and the fact that $1-p(\check{y};\sigma)\leq J_1+J_2$.
\end{proof}

Let $C_0$ be defined as in (\ref{c0}). Let
\begin{eqnarray*}
A_m:=\{ c_u\in [k],\ \exists c_v\in [k],\, \mathrm{s.t.} (c_u-c_v)^2=C_0\}.
\end{eqnarray*}

\begin{proposition}\label{pp44}Assume for all $n$,
\begin{eqnarray}
\frac{\log |y^{-1}(A_m)|}{\log n}\geq \beta >0\label{yia}
\end{eqnarray}
where $\beta$ is a constant independent of $n$.

Assume 
\begin{eqnarray}
\delta=\left(1+\frac{4}{\beta}\right)\frac{\log\log n}{\log n},\label{ac22}
\end{eqnarray}
and
\begin{eqnarray}
\sigma^2>\frac{(1+\delta)C_0n}{4\beta\log n}\label{sb22}
\end{eqnarray}
then
\begin{eqnarray*}
\lim_{n\rightarrow\infty}p(\hat{y};\sigma)=0,
\end{eqnarray*}

\end{proposition}

\begin{proof}
For $y\in \Omega$, $a\in[n]$ such that $y(a)=c_u\in A_m$. Assume $c_v\in[k]$ such that  $(c_u-c_v)^2=C_0$. Define $y^{(a)}\in \Omega$ as follows:
\begin{eqnarray*}
y^{(a)}(i)=\begin{cases}y(i)& \mathrm{if}\ i\in [n]\setminus\{a\}\\ c_v&\mathrm{if}\ i=a \end{cases}
\end{eqnarray*}
Then
\begin{eqnarray*}
t_{u,u}(y^{(a)},y)&=&n_u(y)-1;\\
t_{v,u}(y^{(a)},y)&=&1;\\
t_{i,i}(y^{(a)},y)&=&n_i(y),\ \mathrm{if}\ i\in [k]\setminus \{u\};\\
t_{i,j}(y^{(a)},y)&=&0,\ \mathrm{if}\ (i,j)\in [k]^2\setminus \{(v,u)\}\ \mathrm{and}\ i\neq j;
\end{eqnarray*}
and for $z,w\in[n]$
\begin{eqnarray}
G_{zw}(y^{(a)})-G_{zw}(y)=\begin{cases}c_v-c_u&\mathrm{if}\ z=a,\ w\neq a\\c_u-c_v&\mathrm{if}\ z\neq a,\ w=a\\ 0&\mathrm{Otherwise}\end{cases}\label{gag}
\end{eqnarray}
Note also that
\begin{eqnarray*}
1-p(\check{y};\sigma)\geq \mathrm{Pr}\left(\cup_{a\in[n],y(a)\in A_m} d(y^{(a)})-d(y)<0]\right),
\end{eqnarray*}
since any of the event $[d(y^{(a)})-d(y)<0]$ implies $\check{y}\neq y$. 

By (\ref{dmxy}) we have
\begin{eqnarray*}
d(y^{(a)})-d(y)&=&\|\mathbf{G}(y^{(a)})-\mathbf{G}(y)\|_F^2-2\sigma\langle\mathbf{W},\mathbf{G}(y^{(a)})-\mathbf{G}(y) \rangle\\
&=&(2n-2)(c_u-c_v)^2-2\sigma\langle\mathbf{W},\mathbf{G}(y^{(a)})-\mathbf{G}(y) \rangle.\\
&=&(2n-2)C_0-2\sigma\langle\mathbf{W},\mathbf{G}(y^{(a)})-\mathbf{G}(y) \rangle
\end{eqnarray*}
So $1-p(\check{y};\sigma)$ is at least
\begin{eqnarray*}
&&\mathrm{Pr}\left(\cup_{a,\in[n],y(a)\in A_m}[d(y^{(a)})-d(y)<0]\right)\\
&\geq &\mathrm{Pr}\left(\mathrm{max}_{a\in[n], y(a)\in A_m}\sigma\langle\mathbf{W},\mathbf{G}(y^{(a)})-\mathbf{G}(y) \rangle>(n-1) C_0\right)
\end{eqnarray*}
 Let $H\subset y^{-1}(A_m)$ such that 
 \begin{eqnarray}
 |H|=\frac{|y^{-1}(A_m)|}{\log^2n}=h.\label{hhd}
 \end{eqnarray}
  Then
\begin{eqnarray*}
1-p(\check{y};\sigma)\geq \mathrm{Pr}\left(\mathrm{max}_{a\in H}\sigma\langle\mathbf{W},\mathbf{G}(y^{(a)})-\mathbf{G}(y) \rangle>(n-1) C_0\right)
\end{eqnarray*}
Let $(\mathcal{X},\mathcal{Y},\mathcal{Z})$ be a partition of $[n]^2$ defined by
\begin{eqnarray*}
&&\mathcal{X}=\{\alpha=(\alpha_1,\alpha_2)\in [n]^2, \{\alpha_1,\alpha_2\}\cap H|=\emptyset\}\\
&&\mathcal{Y}=\{\alpha=(\alpha_1,\alpha_2)\in [n]^2, |\{\alpha_1,\alpha_2\}\cap H|=1\}\\
&&\mathcal{Z}=\{\alpha=(\alpha_1,\alpha_2)\in [n]^2, |\{\alpha_1,\alpha_2\}\cap H|=2\}
\end{eqnarray*}
For $\eta\in\{\mathcal{X},\mathcal{Y},\mathcal{Z}\}$, define the $n\times n$ matrix $\mathbf{W}_{\eta}$ from the entries of $\mathbf{W}$ as follows
\begin{eqnarray*}
\mathbf{W}_{\eta}(i,j)=\begin{cases}0&\mathrm{if}\ (i,j)\notin \eta\\ \mathbf{W}(i,j),&\mathrm{if}\ (i,j)\in \eta\end{cases}
\end{eqnarray*}
For each $a\in H$, let
\begin{eqnarray*}
\mathcal{X}^{(a)}=\langle\mathbf{W}_{\mathcal{X}},\mathbf{G}(y^{(a)})-\mathbf{G}(y) \rangle\\
\mathcal{Y}^{(a)}=\langle\mathbf{W}_{\mathcal{Y}},\mathbf{G}(y^{(a)})-\mathbf{G}(y) \rangle\\
\mathcal{Z}^{(a)}=\langle\mathbf{W}_{\mathcal{Z}},\mathbf{G}(y^{(a)})-\mathbf{G}(y) \rangle
\end{eqnarray*}
\begin{claim}The followings are true:
\begin{enumerate}
\item $\mathcal{X}^{(a)}=0$ for $a\in H$.
\item For each $a\in H$, the variables $\mathcal{Y}^{(a)}$ and $\mathcal{Z}^{(a)}$ are independent.
\item $\{\mathcal{Y}^{(a)}\}_{a\in H}$ is a collection of i.i.d. Gaussian random variables.
\end{enumerate}
\end{claim}

\begin{proof}It is straightforward to check (1). (2) holds because $\mathcal{Y}\cap\mathcal{Z}=\emptyset$.

For $s\in H$, let $\mathcal{Y}_s\subseteq \mathcal{Y}^{(a)}$ be defined by
\begin{eqnarray*}
\mathcal{Y}_s=\{\alpha=(\alpha_1,\alpha_2)\in \mathcal{Y}:\alpha_1=s,\ \mathrm{or}\ \alpha_2=s\}.
\end{eqnarray*}
Note that for $s_1,s_2\in H$ and $s_1\neq s_2$, $\mathcal{Y}_{s_1}\cap \mathcal{Y}_{s_2}=\emptyset$. Moreover, $\mathcal{Y}=\cup_{s\in H}\mathcal{Y}_s$. Therefore
\begin{eqnarray*}
\mathcal{Y}^{(a)}=\sum_{s\in H}\langle\mathbf{W}_{\mathcal{Y}_s},\mathbf{G}(y^{(a)})-\mathbf{G}(y) \rangle
\end{eqnarray*}
Note also that $\langle\mathbf{W}_{\mathcal{Y}_s},\mathbf{G}(y^{(a)})-\mathbf{G}(y) \rangle=0$, if $s\neq a$. Hence
\begin{eqnarray*}
\mathcal{Y}^{(a)}=\sum_{\alpha\in\mathcal{Y}_a}[\mathbf{W}(\alpha)]\cdot\{[\mathbf{G}(y^{(a)})-\mathbf{G}(y)](\alpha)\}
\end{eqnarray*}
From (\ref{gag}) we obtain that $\alpha=(\alpha_1,\alpha_2)\in \mathcal{Y}_a$ with $\alpha_1\neq \alpha_2$,
\begin{eqnarray*}
[\mathbf{G}(y^{(a)})-\mathbf{G}(y)](\alpha)=\begin{cases}c_v-c_u &\mathrm{if}\ \alpha_1=a\\c_u-c_v &\mathrm{if}\  \alpha_2=a.\end{cases}
\end{eqnarray*}
So, 
\begin{eqnarray*}
\mathcal{Y}^{(a)}&=&\sum_{\alpha\in \mathcal{Y}_a}[\mathbf{W}(\alpha)]\cdot\{[\mathbf{G}(y^{(a)})-\mathbf{G}(y)](\alpha)\}\\
&=&\left\{\sum_{\alpha\in \mathcal{Y}_a;\alpha_1=a}[\mathbf{W}(\alpha)]-\sum_{\alpha\in \mathcal{Y}_a;\alpha_2=a}[\mathbf{W}(\alpha)]\right\}(c_v-c_u)\\
\end{eqnarray*}
 and $\{\mathcal{Y}^{(a)}\}_{a\in H}$ is a collection of independent Gaussian random variables. Moreover, the variance of $\mathcal{Y}^{(a)}$ is equal to $(2n-2h)C_0$ independent of the choice of $a$.
\end{proof}
By the claim, we obtain
\begin{eqnarray*}
\langle\mathbf{W},\mathbf{G}(y^{(a)})-\mathbf{G}(y) \rangle=\mathcal{Y}^{(a)}+\mathcal{Z}^{(a)}
\end{eqnarray*}
Moreover,
\begin{eqnarray*}
\max_{a\in H}\mathcal{Y}^{(a)}+\mathcal{Z}^{(a)}&\geq& \max_{a\in H_u}\left[\mathcal{Y}^{(a)}\right]-\max_{a\in H_u}\left[\mathcal{Z}^{(a)}\right]\\
\end{eqnarray*}

By the Lemma \ref{mg} we obtain when $\epsilon, h$ satisfy (\ref{epn}) with $N$ replaced by $h$, the event
\begin{eqnarray*}
F_1:=\left\{\max_{a\in H}\mathcal{Y}^{(a)}\geq (1-\epsilon)\sqrt{2\log h\cdot 2C_0\left(n-h\right)}\right\}
\end{eqnarray*}
has probability at least $1-e^{-h^{\epsilon}}$; and the event
\begin{eqnarray*}
F_2:=\left\{\max_{a\in H}\mathcal{Z}^{(a)}\leq (1+\epsilon)\sqrt{2\log h\cdot \max \mathrm{Var}(Z^{(a)})}\right\}
\end{eqnarray*}
with probability at least $1-h^{-\epsilon}$. Moreover,
\begin{eqnarray*}
\mathrm{Var} \mathcal{Z}^{(a)}&=&\|\mathbf{G}(y^{(a)})-\mathbf{G}(y)]\|^2_{F}-\mathrm{Var}(\mathcal{Y}^{(a)})\\
&=&(2n-2)C_0-2C_0\left(n-h\right)\\
&=& C_0(2h-2)
\end{eqnarray*}
Hence
\begin{eqnarray*}
[F_1\cap F_1]\subseteq F
\end{eqnarray*}
where the event 
\begin{eqnarray*}
F:=\left\{\max_{a\in H}\langle\mathbf{W},\mathbf{G}(y^{(a)})-\mathbf{G}(y) \rangle\geq \left(1-\epsilon-(1+\epsilon)\sqrt{\frac{h-1}{n-h}}\right)\sqrt{2\log h\cdot 2C_0(n-h)}\right\}
\end{eqnarray*}
Hence
\begin{eqnarray*}
\mathrm{Pr}(F)&\geq& \mathrm{Pr}(F_1\cap F_2)=1-\mathrm{Pr}(F_1^c\cup F_2^c)\geq 1-\mathrm{Pr}(F_1^c)-\mathrm{Pr}(F_2^c)\\
&=&\mathrm{Pr}(F_1)+\mathrm{Pr}(F_2)-1\geq 1-h^{-\epsilon}-e^{-h^{\epsilon}}
\end{eqnarray*}
Moreover, by (\ref{yia}) and (\ref{hhd}) we have
\begin{eqnarray*}
&& \left(1-\epsilon-(1+\epsilon)\sqrt{\frac{h-1}{n-h}}\right)\sqrt{2\log h\cdot 2C_0(n-h)}\\
&\geq &2\sqrt{C_0\beta n\log n}\sqrt{1-\frac{2\log\log n}{\beta\log n}}\sqrt{1-\frac{1}{\log^2 n}}\left(1-\epsilon-\frac{2}{\log n}\right)
\end{eqnarray*}
Let 
\begin{eqnarray*}
&&\tilde{F}:=\\
&&\left\{\max_{a\in H}\langle\mathbf{W},\mathbf{G}(y^{(a)})-\mathbf{G}(y) \rangle\geq 2\sqrt{C_0\beta n\log n}\sqrt{1-\frac{2\log\log n}{\beta\log n}}\sqrt{1-\frac{1}{\log^2 n}}\left(1-\epsilon-\frac{2}{\log n}\right)\right\}
\end{eqnarray*}
Then $F\subseteq \tilde{F}$.

Let $\epsilon$ be given as in (\ref{eph}), then when $n$ is sufficiently large (\ref{epn}) holds with $N$ replaced by $h$.
When (\ref{ac22}) and (\ref{sb22}) hold,
\begin{eqnarray*}
&&\mathrm{Pr}\left(\mathrm{max}_{a\in y^{-1}(A_m)}\sigma\langle\mathbf{W},\mathbf{G}(y^{(ab)})-\mathbf{G}(y) \rangle>2C_0n\right)\\
&\geq &\mathrm{Pr}(\tilde{F})\geq \mathrm{Pr}(F)\geq 1-\frac{1}{\log n},
\end{eqnarray*}
as $n$ is sufficiently large. Then the proposition follows.
\end{proof}

\section{Proof of Theorem \ref{m2} when the number of vertices in each color is arbitrary in the sample space}\label{pm21}

Now we prove theorem \ref{m2} when the sample space for the MLE is $\Omega$. Let $y\in \Omega$ be the true color assignment mapping. Assume that for each $x\in\Omega$, 
\begin{eqnarray*}
n_i(x)=|x^{-1}(c_i)|,\ \mathrm{for}\ i\in[k]
\end{eqnarray*}
are arbitrary positive integers satisfying (\ref{sc}) Let $\mathbf{K}(x)$ be defined as in (\ref{kij}), and $\mathbf{R}$ be defined as in (\ref{rgw}).

Given a sample $\mathbf{R}$, we want to determine the the groups of vertices such that vertices within each group have the same color under the mapping $y$. Let $\tilde{y}$ be defined as in (\ref{ty}).

Again for $i,j\in [k]$, let $S_{i,j}(x,y)$ be defined as in (\ref{sij}), and $t_{i,j}(x,y)=|S_{i,j}(x,y)|$. Then
\begin{eqnarray}
\sum_{i\in [k]}t_{i,j}(x,y)=n_j(y);\qquad
\sum_{j\in [k]}t_{i,j}(x,y)=n_i(x);\qquad 
\sum_{i,j\in [k]}t_{i,j}(x,y)=n\label{tij1};
\end{eqnarray}
and
\begin{eqnarray}
\langle\mathbf{K}(x),\mathbf{K}(y) \rangle&=&\sum_{a,b\in[n]}\mathbf{1}_{x(a)\neq x(b)}\mathbf{1}_{y(a)\neq y(b)}\notag\\
&=&\sum_{i,j\in [k]}t_{i,j}(x,y)\left[n-n_i(x)-n_j(y)+t_{i,j}(x,y)\right]\notag\\
&=&n^2-\sum_{i\in[k]}[n_i(x)]^2-\sum_{j\in[k]}[n_j(y)]^2+\sum_{i,j\in[k]}[t_{i,j}(x,y)]^2;\label{lak}
\end{eqnarray}
where 
\begin{eqnarray*}
n-n_i(x)-n_j(y)+t_{i,j}(x,y)=\left\{u\in[n]: x(u)\neq i,\ \mathrm{and}\ y(u)\neq j\right\};
\end{eqnarray*}
and (\ref{lak}) follows from (\ref{tij1}).
In particular,
\begin{eqnarray*}
\langle\mathbf{K}(y),\mathbf{K}(y) \rangle=n^2-\sum_{j\in[k]}[n_j(y)]^2
\end{eqnarray*}
Hence from (\ref{ty}) we obtain
\begin{eqnarray*}
\tilde{y}=\mathrm{argmax}_{x\in \Omega}\left[\sum_{j=1}^{k}[n_j(x)]^2+2\langle\mathbf{K}(x),\mathbf{R} \rangle\right]
\end{eqnarray*}

Define
\begin{eqnarray}
g(x):=\sum_{j=1}^{k}[n_j(x)]^2+2\langle\mathbf{K}(x),\mathbf{R} \rangle;\label{dgx}
\end{eqnarray}
then
\begin{eqnarray}
&&g(x)-g(y)\label{gxmy}\\
&=&2\sigma\langle\mathbf{K}(x)-\mathbf{K}(y),\mathbf{W} \rangle+2\sum_{i,j\in[k]}[t_{i,j}(x,y)]^2-\sum_{i=1}^{k}[n_i(x)]^2-\sum_{j=1}^{k}[n_j(y)]^2.\notag
\end{eqnarray} 
Note that $g(x)-g(y)$ is a Gaussian random variable with mean $2\sum_{i,j\in[k]}[t_{i,j}(x,y)]^2-\sum_{i=1}^{k}[n_i(x)]^2-\sum_{j=1}^{k}[n_j(y)]^2$ and variance $4\sigma^2\|\mathbf{K}(x)-\mathbf{K}(y)\|_F^2$. 

Let
\begin{eqnarray}
L(x,y)=\sum_{i=1}^{k}[n_i(x)]^2+\sum_{j=1}^{k}[n_j(y)]^2-2\sum_{i,j\in[k]}[t_{i,j}(x,y)]^2\label{lxy}
\end{eqnarray}
Then it is straightforward to check that
\begin{eqnarray}
\|\mathbf{K}(x)-\mathbf{K}(y)\|^2_F=L(x,y)\label{kle}
\end{eqnarray}
Therefore
\begin{eqnarray*}
\mathrm{Pr}(g(x)-g(y)>0)&=&\mathrm{Pr}_{\xi\sim\mathcal{N}(0,1)}\left(\xi>\frac{\sqrt{L(x,y)}}{2\sigma}\right)
\leq  e^{-\frac{L(x,y)}{8\sigma^2}},
\end{eqnarray*}
where the last inequality follows from the fact that if $\xi\sim\mathcal{N}(0,1)$, then for $x>0$, $\mathrm{Pr}(\xi>x)\leq e^{-\frac{x^2}{2}}$.

\begin{definition}\label{dfeq}
For $y\in \Omega$, let $C(y)$ consist of all the $x\in \Omega$ such that $x$ can be obtained from $y$ by a permutation of colors.  More precisely, $x\in C(y)\subset \Omega$ if and only if the following condition holds
\begin{itemize}
\item for $i,j\in[n]$, $y(i)=y(j)$ if and only if $x(i)=x(j)$.
\end{itemize}

We define an equivalence relation on $\Omega$ as follows: we say $x,z\in \Omega$ are equivalent if and only if $x\in C(z)$. Let $\ol{\Omega}$ be the set of all the equivalence classes in $\Omega$.
\end{definition}

 We have the following elementary lemma:
\begin{lemma}If $x,z\in \Omega$ are equivalent, then
\begin{eqnarray*}
\sum_{i,j\in[k]}[t_{i,j}(x,y)]^2=\sum_{i,j\in[k]}[t_{i,j}(z,y)]^2
\end{eqnarray*}
\end{lemma}
\begin{proof}By definition if $x,z\in \Omega$ are equivalent, then there exists a permutation $\omega$ of $[k]$, such that for all $l\in[n]$, we have
\begin{eqnarray}
z(l)=\omega(x(l)).\label{zxr}
\end{eqnarray}
Therefore for all $i \in [k]$, the following two sets are equal:
\begin{eqnarray}
x^{-1}(i)=z^{-1}(\omega(i)).\label{xzo}
\end{eqnarray}
In other words, for any $u\in [n]$ and $i\in[k]$, $x(u)=i$ if and only if $z(u)=\omega(i)$. Then we have
\begin{eqnarray*}
t_{i,j}(x,y)=t_{\omega(i),j}(z,y)
\end{eqnarray*}
by summing over all the $i,j$'s in $[k]$, and using the fact that $\omega$ is a bijection from $[k]$ to $[k]$, we obtain the lemma.
\end{proof}

\begin{lemma}\label{l12}Let $y\in \Omega$ be the true color assignment mapping. If $x$ and $z$ are equivalent elements in $\Omega$, then for any chosen sample $\mathbf{W}$, we have
\begin{eqnarray}
g(x)=g(z).\label{gxez};
\end{eqnarray}
and
\begin{eqnarray}
L(x,y)=L(z,y)\label{lxez}.
\end{eqnarray}
Moreover, if $y^*\in C(y)$, then
\begin{eqnarray}
L(x,y)=L(x,y^*)\label{lyee}
\end{eqnarray}
\end{lemma}
\begin{proof}From the definition (\ref{dgx}) of $g(x)$ we obtain
\begin{eqnarray*}
g(x)=2\langle\mathbf{K}(x),\mathbf{K}(y) \rangle+2\sigma\langle \mathbf{K}(x),\mathbf{W}\rangle+\sum_{j=1}^{k}[n_j(x)]^2
\end{eqnarray*}
Recall that $\mathbf{K}_{i,j}(x)=1$ if and only if $x(i)\neq x(j)$; if $x(i)=x(j)$, $\mathbf{K}_{i,j}(x)=0$. Since $x$ and $z$ are equivalent $x(i)\neq x(j)$ if and only if $z(i)\neq z(j)$, therefore 
\begin{eqnarray}
\mathbf{K}(x)=\mathbf{K}(z).\label{kxez}
\end{eqnarray}
Moreover, if $x$ and $z$ are equivalent and (\ref{zxr}) holds, by (\ref{xzo}) we obtain $x\in \Omega_{n_{\omega(1)}(z),\ldots,n_{\omega(k)}(z)}$; in particular this implies
\begin{eqnarray*}
\sum_{j=1}^{k}[n_j(x)]^2=\sum_{j=1}^{k}[n_{\omega(j)}(z)]^2=\sum_{j=1}^{k}[n_{j}(z)]^2
\end{eqnarray*}
Then we obtain (\ref{gxez}). The expression (\ref{lxez}) follows from (\ref{kxez}) and (\ref{kle}). The expression (\ref{lyee}) follows from (\ref{lxez}) by observing that the expression (\ref{lxy}) of $L(x,y)$ is symmetric in $x$ and $y$.
\end{proof}

Let
\begin{eqnarray*}
p(\tilde{y},\sigma)=\mathrm{Pr}(\tilde{y}\in C(y))
\end{eqnarray*}
Then
\begin{eqnarray*}
p(\tilde{y},\sigma)=\mathrm{Pr}\left(g(y)>\max_{x\in \Omega\setminus C(y)}g(x)\right);
\end{eqnarray*}
hence
\begin{eqnarray*}
1-p(\tilde{y},\sigma)&\leq& \sum_{C(x)\in \overline{\Omega}\setminus \{C(y)\}}\mathrm{Pr}(g(x)-g(y)\geq 0)
\leq \sum_{C(x)\in \overline{\Omega}\setminus\{C(y)\}}e^{-\frac{L(x,y)}{8\sigma^2}}.
\end{eqnarray*}

\begin{lemma}\label{lgz}For any $x,y\in \Omega$, $L(x,y)\geq 0$, where the equality holds if and only if $x\in C(y)$.
\end{lemma}

\begin{proof}By (\ref{tij1}) and (\ref{lxy}) we have
\begin{eqnarray}
&&L(x,y)\label{lxy1}\\
&=&2\left[\sum_{i\in [k]}\sum_{1\leq j_1<j_2\leq k}t_{i,j_1}(x,y)t_{i,j_2}(x,y)+\sum_{j\in [k]}\sum_{1\leq i_1<i_2\leq k}t_{i_1,j}(x,y)t_{i_2,j}(x,y)\right]\geq 0\notag
\end{eqnarray}

We first note that $L(x,y)=0$ if $x\in C(y)$. Indeed, if $x\in C(y)$, then there exists a permutation $\omega:[k]\rightarrow [k]$, such that $x=\omega\circ y$. Then for all $u\in [k]$ satisfying $x(u)=i$, we have $y(u)=\omega^{-1}(i)$. Then
\begin{eqnarray*}
&&t_{i,\omega^{-1}(i)}=n_i(x);\\
&&t_{i,j}=0,\ \forall\ j\in[k]\setminus\{\omega^{-1}(i)\}.
\end{eqnarray*}
Therefore for any $j_1,j_2\in [k]$ and $j_1<j_2$, at least one of $t_{i,j_1}(x,y)$ and $t_{i,j_2}(x,y)$ is 0. Similarly for any $i_1,i_2\in[k]$ and $i_1<i_2$, at least one of $t_{i_1,j}(x,y)$ and $t_{i_2,j}(x,y)$ is 0. Then $L(x,y)=0$ by (\ref{lxy1}).

It remains to show that if $L(x,y)=0$, then $x\in C(y)$. Note that if $L(x,y)=0$, then
\begin{eqnarray*}
&&t_{i,j_1}(x,y)t_{i,j_2}(x,y)=0,\ \forall i\in [k],\ 1\leq j_1<j_2\leq k;\ \mathrm{and}\\
&&t_{i_1,j}(x,y)t_{i_2,j}(x,y)=0,\ \forall j\in [k],\ 1\leq i_1<i_2\leq k
\end{eqnarray*}
Then for any fixed $i\in [k]$, there exists exactly one $j\in [k]$, such that $t_{i,j}\neq 0$; and for each fixed $j\in [k]$, there exists exactly one $i\in [k]$, such that $t_{i,j}\neq 0$. Then the lemma follows.
\end{proof}

Let $\tilde{\mathcal{B}}$ be the set given by
\begin{eqnarray}
\tilde{\mathcal{B}}=\left\{(t_{1,1},t_{1,2},\ldots,t_{k,k})\in \left\{0,1,\ldots,n\right\}^{k^2}:\sum_{i=1}^{k}t_{i,j}=n_j(y)\right\}.\label{db}
\end{eqnarray}

For a small positive number $\epsilon>0$, let $\tilde{\mathcal{B}}_{\epsilon}$ be the set given by
\begin{eqnarray}
\tilde{\mathcal{B}}_{\epsilon}=\left\{(t_{1,1},\ldots,t_{k,k})\in \tilde{\mathcal{B}}: \ \forall\ i\in[k],\exists j\in[k],\ \mathrm{s.t.}\ t_{j,i}\geq n_i(y)-n\epsilon\right\}\label{dbe}
\end{eqnarray}

\begin{lemma}\label{l55}Let $\mathcal{R}_c$ be defined in (\ref{drc}).
When 
\begin{eqnarray}
(t_{1,1}(x,y),\ldots,t_{k,k}(x,y))\in \tilde{\mathcal{B}}\setminus \tilde{\mathcal{B}}_{\epsilon},\label{tme}
\end{eqnarray}
and (\ref{vir}) hold, we have
\begin{eqnarray*}
L(x,y)\geq \frac{c\epsilon n^2}{k}.
\end{eqnarray*}
\end{lemma}

\begin{proof}When (\ref{tme}) holds, by (\ref{dbe}) we have there exists $i_0\in[k]$, such that
\begin{eqnarray*}
n_{i_0}(y)-\max_{j\in[k]}t_{j,i_0}(x,y)\geq \epsilon n.
\end{eqnarray*}

By (\ref{lxy1}), we obtain
\begin{eqnarray*}
L(x,y)&\geq& \max_{j\in[k]} t_{j,i_0}(x,y)\left[n_{i_0}(y)-\max_{j\in[k]} t_{j,i_0}(x,y)\right]\geq \frac{n_{i_0}(y)}{k}\epsilon n
\end{eqnarray*}
By (\ref{vir}), we have
\begin{eqnarray*}
\frac{n_{i_0}(y)}{k}\epsilon n\geq \frac{c\epsilon n^2}{k}
\end{eqnarray*}
\end{proof}

\begin{lemma}\label{ll56}Let $\mathcal{R}_c$ be defined in (\ref{drc}). Assume (\ref{vir}) holds. Let $x\in\Omega$ satisfy
\begin{eqnarray}
\left(\frac{n_1(x)}{n},\frac{n_2(x)}{n},\ldots,\frac{n_k(x)}{n}\right)\in\mathcal{R}_{\frac{2c}{3}}\label{xrc}
\end{eqnarray}
For $i\in [k]$,  let
\begin{eqnarray}
t_{w(i),i}(x,y)=\max_{j\in [k]}t_{j,i}(x,y),\label{twi}
\end{eqnarray}
where $w(i)\in [k]$. When $\epsilon\in\left(0,\frac{2c}{3k}\right)$ and $(t_{1,1}(x,y),\ldots, t_{k,k}(x,y))\in \tilde{\mathcal{B}}_{\epsilon}$, $w$ is a bijection from $[k]$ to $[k]$.
\end{lemma}

\begin{proof}When $(t_{1,1}(x,y),\ldots, t_{k,k}(x,y))\in \tilde{\mathcal{B}}_{\epsilon}$, by (\ref{dbe}) we have
\begin{eqnarray*}
\max_{j\in[k]}t_{j,i}(x,y)>n_i(y)-\epsilon n,\ \forall\ i\in[k].
\end{eqnarray*}
If there exist $j_1,j_2\in [k]$, such that $j_1\neq j_2$, and
\begin{eqnarray*}
t_{j_1,i}(x,y)=t_{j_2,i}(x,y)=\max_{j\in [k]}t_{j,i}(x,y);
\end{eqnarray*}
then 
\begin{eqnarray*}
n_i(y)\geq t_{j_1,i}+t_{j_2,i}>2n_i(y)-2\epsilon n\geq \left(2-\frac{2\epsilon }{c}\right)n_i(y);
\end{eqnarray*}
where the last inequality follows from (\ref{vir}). But this is impossible when $\epsilon < \frac{c}{2}$. Therefore $w(i)$ satisfying (\ref{twi}) is unique for each $i\in[k]$, and $w$ is a mapping from $[k]$ to $[k]$.

If there exist $i,j\in[k]$ such that $i\neq j$ and $w(i)=w(j)=l\in[k]$; 
then there exists $q\in[k]$, such that $q\neq w(s)$ for all $s\in[k]$. More precisely,
\begin{eqnarray*}
t_{q,s}(x,y)\neq \max_{r\in[k]}t_{r,s}(x,y)>n_s(y)-\epsilon n,\ \forall\ s\in[k].
\end{eqnarray*}
By (\ref{tij1}) we have
\begin{eqnarray}
t_{q,s}(x,y)\leq\epsilon n,\ \forall s\in[k].\label{tqs}
\end{eqnarray}
Therefore
\begin{eqnarray*}
\frac{2cn}{3}\leq \sum_{s\in[k]}t_{q,s}(x,y)=n_q(x)\leq k\epsilon n,
\end{eqnarray*}
where the first inequality follows from (\ref{xrc}), and the last inequality follows from (\ref{tqs}). But this is impossible when $\epsilon<\frac{2c}{3k}$. The contradiction implies that $w$ must be a bijection from $[k]$, and the lemma follows.
\end{proof}

We have the following proposition
\begin{proposition}\label{l2a}Let $\mathcal{R}_c$ be defined as in (\ref{drc}). Assume (\ref{vir}) and (\ref{xrc}) holds. Recall that $k$ is the total number of colors. Here $c,k$ may depend on $n$. Assume the following conditions hold
\begin{eqnarray}
\lim_{n\rightarrow\infty}\frac{\log k}{c^3 \log n}=0;\label{asc1}
\end{eqnarray}

If there exists $\delta>0$ independent of $n$, such that 
\begin{eqnarray}
\sigma^2<\frac{(1-\delta)[n_k(y)+n_{k-1}(y)]}{4\log n}\label{spp}
\end{eqnarray}
then 
\begin{eqnarray*}
\lim_{n\rightarrow\infty}p(\tilde{y};\sigma)=1
\end{eqnarray*}

\end{proposition}
\begin{proof} 
Note that
\begin{eqnarray*}
\sum_{C(x)\in \overline{\Omega}\setminus \{C(y)\}} e^{-\frac{L(x,y)}{8\sigma^2}}
\leq  I_3+I_4
\end{eqnarray*}
where 
\begin{eqnarray*}
I_3&=&\sum_{C(x)\in\ol{\Omega}:\left(t_{1,1}(x,y),\ldots,t_{k,k}(x,y)\right)\in[\mathcal{B}\setminus\tilde{\mathcal{B}}_{\epsilon}], C(x)\neq C(y)}e^{\frac{-L(x,y)}{8\sigma^2}}
\end{eqnarray*}
and
\begin{eqnarray*}
I_4=\sum_{C(x)\in\ol{\Omega}:\left(t_{1,1}(x,y),\ldots, t_{k,k}(x,y)\right)\in\tilde{\mathcal{B}}_{\epsilon}, C(x)\neq C(y)}e^{\frac{-L(x,y)}{8\sigma^2}}.
\end{eqnarray*}

By Lemma \ref{l55}, when (\ref{spp}) holds, we have
\begin{eqnarray*}
I_3&\leq& k^n e^{-\frac{c\epsilon n^2\log n }{2k(1-\delta)[n_k(y)+n_{k-1}(y)]}}
\end{eqnarray*}
Since
\begin{eqnarray}
n_k(y)+n_{k-1}(y)\leq \frac{2n}{k},\label{snu}
\end{eqnarray}
we have
\begin{eqnarray}
I_3\leq e^{n\left(\log k-\frac{c\epsilon \log n}{4(1-\delta)}\right)}.\label{i3ub}
\end{eqnarray}

Now let us consider $I_4$. Assume $\epsilon\in\left(0,\frac{2c}{3k}\right)$. Let $w$ be the bijection from $[k]$ to $[k]$ as defined in (\ref{twi}). Let $y^*\in \Omega$ be defined by
\begin{eqnarray*}
y^*(z)=w(y(z)),\ \forall z\in [n].
\end{eqnarray*}
Then $y^*\in C(y)$. Moreover, $x$ and $y^*$ satisfies
\begin{eqnarray}
t_{i,i}(x,y^*)\geq n_i(y^*)-n\epsilon,\ \forall i\in[k].\label{dxys}
\end{eqnarray}

We consider the following color changing process to obtain $x$ from $y^*$.
\begin{enumerate}
\item If for all $(j,i)\in [k]^2$, and $j\neq i$, $t_{j,i}(x,y^*)=0$, then $x=y^*$.
\item If (1) does not hold, find the least $(j,i)\in[k]^2$ in lexicographic order such that $j\neq i$ and $t_{j,i}(x,y^*)>0$. Choose an arbitrary vertex $u\in S_{j,i}(x,y^*)$. Define $y_1\in \Omega$ as follows
\begin{eqnarray*}
y_1(z)= \begin{cases}c_j&\mathrm{if}\ z=u\\ y^*(z)&\mathrm{if}\ z\in [n]\setminus \{u\} \end{cases}
\end{eqnarray*}
\end{enumerate}

Then we have
\begin{eqnarray}
&&t_{j,i}(x,y_1)=t_{j,i}(x,y^*)-1\label{tjid}\\
&&t_{j,j}(x,y_1)=t_{j,j}(x,y^*)+1\label{tjjd}\\
&&t_{a,b}(x,y_1)=t_{a,b}(x,y^*)\ \forall (a,b)\in \left([k]^2\setminus\{(j,i),(j,j)\}\right).
\end{eqnarray}
Therefore $x$ and $y^*$ satisfies
\begin{eqnarray*}
t_{i,i}(x,y_1)\geq n_i(y_1)-n\epsilon,\ \forall i\in[k].
\end{eqnarray*}

From (\ref{lyee}) and (\ref{lxy1}) we obtain
\begin{eqnarray*}
\frac{1}{2}\left(L(x,y_1)-L(x,y)\right)&=&\frac{1}{2} \left(L(x,y_1)-L(x,y^*)\right)\\
&=&t_{j,i}(x,y_1)\left(\sum_{l\neq i}t_{j,l}(x,y_1)\right)+t_{j,j}(x,y_1)\left(\sum_{l\neq j}t_{j,l}(x,y_1)\right)-t_{j,i}(x,y_1)t_{j,j}(x,y_1)\\
&&+t_{j,i}(x,y_1)\left(\sum_{l\neq j}t_{l,i}(x,y_1)\right)+t_{j,j}(x,y_1)\left(\sum_{l\neq j}t_{l,j}(x,y_1)\right)\\
&&-\left[t_{j,i}(x,y^*)\left(\sum_{l\neq i}t_{j,l}(x,y^*)\right)+t_{j,j}(x,y^*)\left(\sum_{l\neq j}t_{j,l}(x,y^*)\right)-t_{j,i}(x,y^*)t_{j,j}(x,y^*)\right.\\
&&+\left.t_{j,i}(x,y^*)\left(\sum_{l\neq j}t_{l,i}(x,y^*)\right)+t_{j,j}(x,y^*)\left(\sum_{l\neq j}t_{l,j}(x,y^*)\right)\right]
\end{eqnarray*}
By (\ref{tjid}) we obtain
\begin{eqnarray*}
&&t_{j,i}(x,y_1)\left(\sum_{l\neq i}t_{j,l}(x,y_1)\right)-t_{j,i}(x,y^*)\left(\sum_{l\neq i}t_{j,l}(x,y^*)\right)\\
&=& t_{j,i}(x,y_1)\left[n_j(x)-t_{j,i}(x,y_1)\right]- t_{j,i}(x,y^*)\left[n_j(x)-t_{j,i}(x,y^*)\right]\\
&=&t_{j,i}(x,y^*)+t_{j,i}(x,y_1)-n_j(x)
\end{eqnarray*}
Similarly by (\ref{tjjd}) we obtain
\begin{eqnarray*}
&&t_{j,j}(x,y_1)\left(\sum_{l\neq j}t_{j,l}(x,y_1)\right)-t_{j,j}(x,y^*)\left(\sum_{l\neq j}t_{j,l}(x,y^*)\right)\\
&=& n_j(x)-t_{j,j}(x,y_1)-t_{j,j}(x,y^*).
\end{eqnarray*}
\begin{eqnarray*}
-t_{j,i}(x,y_1)t_{j,j}(x,y_1)+t_{j,i}(x,y^*)t_{j,j}(x,y^*)=-t_{j,i}(x,y^*)+t_{j,j}(x,y^*)+1
\end{eqnarray*}
\begin{eqnarray*}
t_{j,i}(x,y_1)\left(\sum_{l\neq j}t_{l,i}(x,y_1)\right)-t_{j,i}(x,y^*)\left(\sum_{l\neq j}t_{l,i}(x,y^*)\right)=t_{j,i}(x,y^*)-n_i(y^*)
\end{eqnarray*}
\begin{eqnarray*}
t_{j,j}(x,y_1)\left(\sum_{l\neq j}t_{l,j}(x,y_1)\right)-t_{j,j}(x,y^*)\left(\sum_{l\neq j}t_{l,j}(x,y^*)\right)=n_j(y^*)-t_{j,j}(x,y^*)
\end{eqnarray*}
Therefore we obtain
\begin{eqnarray*}
L(x,y_1)-L(x,y)&=&2\left[n_j(y^*)-n_i(y^*)-2t_{j,j}(x,y^*)+2t_{j,i}(x,y^*)-1\right]
\end{eqnarray*}
Then by (\ref{dxys}),
\begin{eqnarray*}
L(x,y_1)-L(x,y)\leq  -2n_i(y^*)-2n_j(y^*)+8n\epsilon\leq -2(n_{k-1}(y)+n_k(y)-4n\epsilon)
\end{eqnarray*}
Therefore when (\ref{spp}) holds,
\begin{eqnarray}
e^{-\frac{L(x,y)}{8\sigma^2}}\leq e^{-\frac{L(x,y_1)}{8\sigma^2}}e^{-\frac{\log n}{1-\delta}\left(1-\frac{4n\epsilon}{n_{k-1}(y)+n_k(y)}\right)}\label{idi}
\end{eqnarray}

In general, if we have constructed $y_r$ ($r\geq 1$), we shall construct $y_{r+1}$ as follows.
\begin{enumerate}[label=(\alph*)]
\item If for all $(j,i)\in [k]^2$, and $j\neq i$, $t_{j,i}(x,y_r)=0$, then $x=y_r$.
\item If (a) does not hold, find the least $(j,i)\in[k]^2$ in lexicographic order such that $j\neq i$ and $t_{j,i}(x,y_r)>0$. Choose an arbitrary vertex $u\in S_{j,i}(x,y_r)$. Define $y_{r+1}\in \Omega$ as follows
\begin{eqnarray*}
y_{r+1}(z)= \begin{cases}c_j&\mathrm{if}\ z=u\\ y_r(z)&\mathrm{if}\ z\in [n]\setminus \{u\} \end{cases}
\end{eqnarray*}
\end{enumerate}
Then if (\ref{dxys}) holds with $y^*$ replaced by $y_r$, then (\ref{dxys}) holds with $y^*$ replaced by $y_{r+1}$. By similar computations as above we obtain
\begin{eqnarray}
e^{-\frac{L(x,y_r)}{8\sigma^2}}\leq e^{-\frac{L(x,y_{r+1})}{8\sigma^2}}e^{-\frac{\log n}{1-\delta}\left(1-\frac{4n\epsilon}{n_{k-1}(y)+n_k(y)}\right)}.\label{idr}
\end{eqnarray}
Recall that the distance $D_{\Omega}$ in $\Omega$ is defined in (\ref{dfdo}). From the constructions of $y_{r+1}$ we have
\begin{eqnarray*}
D_{\Omega}(x,y_{r+1})=D_{\Omega}(x,y)-1.
\end{eqnarray*}
Therefore there exists $s\in [n]$, such that $y_s=x$. By (\ref{idi}) and (\ref{idr}) and the fact that $L(x,x)=0$ we obtain
\begin{eqnarray*}
e^{-\frac{L(x,y)}{8\sigma^2}}\leq e^{-\frac{s\log n}{1-\delta}\left(1-\frac{4n\epsilon}{n_{k-1}(y)+n_k(y)}\right)}
\end{eqnarray*}

Since any $x$ in $\tilde{\mathcal{B}}_{\epsilon}$ can be obtained by $y$ be the color changing process described above, we have
\begin{eqnarray}
I_4\leq \sum_{l=1}^{\infty} (nk)^{l} e^{-\frac{l\log n}{1-\delta}\left(1-\frac{4n\epsilon}{n_{k-1}(y)+n_k(y)}\right)}\label{ifb}
\end{eqnarray}
The right hand side of (\ref{ifb}) is the sum of geometric series with both initial term and common ratio equal to 
\begin{eqnarray}
V:=e^{\log k-\log n\left(\frac{\delta}{1-\delta}-\frac{4n\epsilon}{(1-\delta)(n_{k-1}(y)+n_k(y))}\right)}\label{ves}
\end{eqnarray}

By (\ref{vir}) we have
\begin{eqnarray*}
n_{k-1}(y)+n_k(y)\geq 2cn
\end{eqnarray*}
\begin{eqnarray}
V\leq e^{\log k-\log n\left(\frac{\delta}{1-\delta}-\frac{2\epsilon}{c(1-\delta)}\right)}\label{vup}
\end{eqnarray}
Let
\begin{eqnarray}
\epsilon=\min\left\{\frac{c}{2k},\frac{\delta c}{4}\right\}<\frac{c}{k};\label{evep}
\end{eqnarray}
then from (\ref{i3ub}) and using the fact that $c\leq \frac{1}{k}$, we obtain
\begin{eqnarray*}
0\leq I_3\leq e^{n\left(\log k-\log n\min\left(\frac{c^3}{8},\frac{\delta c^2}{16}\right)\right)}
\end{eqnarray*}
Then from (\ref{asc1}), when $n$ is sufficiently large
\begin{eqnarray*}
\log k\leq \frac{c^3\log n}{16};
\end{eqnarray*}
Since $k\geq 2$ we have
\begin{eqnarray*}
n\left(\log k-\frac{c^3\log n}{8}\right)\leq -n\log k\rightarrow -\infty,
\end{eqnarray*}
as $n\rightarrow\infty$.
Moreover, since $c\leq 1$ from (\ref{asc1}) we obtain
\begin{eqnarray*}
\lim_{n\rightarrow\infty}\frac{\log k}{c^2\log n}=0
\end{eqnarray*}
Hence when $n$ is sufficiently large, we obtain
\begin{eqnarray*}
n\left(\log k-\frac{\delta c^2\log n}{16}\right)\leq -n\log k\rightarrow -\infty,
\end{eqnarray*}
as $n\rightarrow\infty$.

Therefore we obtain
\begin{eqnarray}
\lim_{n\rightarrow\infty}I_3=0.\label{li3}
\end{eqnarray}
Moreover, when $\epsilon$, $V$ are given by (\ref{evep}) and (\ref{ves}), from (\ref{vup}) we obtain
\begin{eqnarray*}
0\leq V\leq e^{\log k-\frac{\delta \log n}{2(1-\delta)}}
\end{eqnarray*}
By (\ref{asc1}), for any constant $\delta>0$ independent of $n$, $\lim_{n\rightarrow \infty}V=0$, therefore
\begin{eqnarray}
\lim_{n_\rightarrow\infty}I_4=\lim_{n\rightarrow\infty}\frac{V}{1-V}=0.\label{li4}
\end{eqnarray}
Then the proposition follows from (\ref{li3}), (\ref{li4}) and the fact that $1-p(\tilde{y};\sigma)\leq I_3+I_4$.
\end{proof}

\begin{proposition}\label{l2b}Let $\mathcal{R}_c$ be defined as in (\ref{drc}). Assume (\ref{vir}) and (\ref{xrc}) holds. Here $k$ and $c$ may depend on $n$. Assume $n_1(y)\geq n_2(y)\geq\ldots\geq n_k(y)$ and
\begin{eqnarray}
\frac{\log n_k(y)}{\log n}\geq 1+\frac{\log c}{\log n}\geq \beta>0,\ \forall\ n,\label{bll}
\end{eqnarray}
where $\beta$ is a constant independent of $n$.

Assume
\begin{eqnarray}
\delta=\left(1+\frac{4}{\beta}\right)\frac{\log\log n}{\log n},\label{acc}
\end{eqnarray}
and
\begin{eqnarray}
\sigma^2>\frac{(1+\delta)[n_k(y)+n_{k-1}(y)]}{4\beta \log n},\label{sbb}
\end{eqnarray}
then 
\begin{eqnarray*}
\lim_{n\rightarrow\infty}p(\tilde{y};\sigma)=0.
\end{eqnarray*}
\end{proposition}

\begin{proof}
For $y\in \Omega$, $a\in[n]$ such that $c_{k-1}=y(a)$. Let $y^{(a)}\in\Omega$ be defined by
\begin{eqnarray*}
y^{(a)}(i)=\begin{cases} y(i)&\mathrm{if}\ i\in[n],\ \mathrm{and}\ i\neq a\\ c_{k}&\mathrm{if}\ i=a.\end{cases}
\end{eqnarray*}
Then
\begin{eqnarray*}
&&t_{k,k-1}(y^{(a)},y)=1;\\
&&t_{k-1,k-1}(y^{(a)},y)=n_{k-1}(y)-1;\\
&&t_{i,i}(y^{(a)},y)=n_i(y);\ \forall\ i\in [k]\setminus \{k-1\};\\
&&t_{i,j}(y^{(a)},y)=0;\ \forall (i,j)\in [k]^2\setminus\{(k,k-1)\},\ \mathrm{and}\ i\neq j.
\end{eqnarray*}
and
\begin{eqnarray*}
&&n_k(y^{(a)})=n_k(y)+1;\\
&&n_{k-1}(y^{(a)})=n_{k-1}(y)-1;\\
&&n_i(y^{(a)})=n_i(y);\ \forall\ i\in[k]\setminus \{k,k-1\}.
\end{eqnarray*}
Moreover,
\begin{eqnarray*}
1-p(\tilde{y};\sigma)\geq \mathrm{Pr}\left(\cup_{a\in[n]\cap y^{-1}(c_{k-1})}[g(y^{(a)})-g(y)>0]\right)
\end{eqnarray*}
Since any of the event $[g(y^{(a)})-g(y)>0]$ implies $\tilde{y}\neq y$. 

From (\ref{lak}) we obtain
\begin{eqnarray}
\langle \mathbf{K}(y),\mathbf{K}(y^{(a)})-\mathbf{K}(y)\rangle=\sum_{i,j\in[k]}[t_{i,j}(y^{(a)},y)]^2-\sum_{i\in[k]}[n_i(y^{(a)})]^2=-2n_k(y)\label{kn1}
\end{eqnarray}
Then from (\ref{gxmy}) we have
\begin{eqnarray*}
g(y^{(a)})-g(y)&=&\sum_{i\in[k]}[n_i(y^{(a)})]^2-\sum_{i\in[k]}[n_i(y)]^2+2\langle \mathbf{K}(y),\mathbf{K}(y^{(a)})-\mathbf{K}(y)\rangle+2\sigma\langle\mathbf{W},\mathbf{K}(y^{(a)})-\mathbf{K}(y) \rangle\\
&=&-2n_{k-1}(y)-2n_k(y)+2+2\sigma\langle\mathbf{W},\mathbf{K}(y^{(a)})-\mathbf{K}(y) \rangle.
\end{eqnarray*}
Let $H_{k-1}\subset y^{-1}(c_{k-1})$, such that 
\begin{eqnarray}
|H_{k-1}|=\frac{n_{k-1}(y)}{\log^2n}=h.\label{hkm1}
\end{eqnarray}
Then $1-p(\tilde{y};\sigma)$ is at least
\begin{eqnarray*}
&&\mathrm{Pr}\left(\cup_{a\in[n]\cap y^{-1}(c_{k-1})}[g(y^{(a)})-g(y)>0]\right)\\
&\geq &\mathrm{Pr}\left(\mathrm{max}_{a\in[n]\cap y^{-1}(c_{k-1})}\sigma\langle\mathbf{W},\mathbf{K}(y^{(a)})-\mathbf{K}(y) \rangle>n_k(y)+n_{k-1}(y)-1\right)\\
&\geq &\mathrm{Pr}\left(\mathrm{max}_{a\in H_{k-1}}\sigma\langle\mathbf{W},\mathbf{K}(y^{(a)})-\mathbf{K}(y) \rangle>n_k(y)+n_{k-1}(y)-1\right)
\end{eqnarray*}

Let $(\mathcal{X},\mathcal{Y},\mathcal{Z})$ be a partition of $[n]^2$ defined by
\begin{eqnarray*}
&&\mathcal{X}=\{\alpha=(\alpha_1,\alpha_2)\in [n]^2, \{\alpha_1,\alpha_2\}\cap [H_{k-1}]=\emptyset\}\\
&&\mathcal{Y}=\{\alpha=(\alpha_1,\alpha_2)\in [n]^2, |\{\alpha_1,\alpha_2\}\cap [H_{k-1}]|=1\}\\
&&\mathcal{Z}=\{\alpha=(\alpha_1,\alpha_2)\in [n]^2, |\{\alpha_1,\alpha_2\}\cap [H_{k-1}]|=2\}
\end{eqnarray*}
For $\eta\in\{\mathcal{X},\mathcal{Y},\mathcal{Z}\}$, define the $n\times n$ matrix $\mathbf{W}_{\eta}$ from the entries of $\mathbf{W}$ as follows
\begin{eqnarray*}
\mathbf{W}_{\eta}(i,j)=\begin{cases}0&\mathrm{if}\ (i,j)\notin \eta\\ \mathbf{W}(i,j),&\mathrm{if}\ (i,j)\in \eta\end{cases}
\end{eqnarray*}
 For each $a\in H_{k-1}$, let
\begin{eqnarray*}
\mathcal{X}_{a}=\langle\mathbf{W}_{\mathcal{X}},\mathbf{K}(y^{(a)})-\mathbf{K}(y) \rangle\\
\mathcal{Y}_{a}=\langle\mathbf{W}_{\mathcal{Y}},\mathbf{K}(y^{(a)})-\mathbf{K}(y) \rangle\\
\mathcal{Z}_{a}=\langle\mathbf{W}_{\mathcal{Z}},\mathbf{K}(y^{(a)})-\mathbf{K}(y) \rangle
\end{eqnarray*}
Explicit computations show that
\begin{eqnarray*}
\mathbf{K}_{i,j}(y^{(a)})-\mathbf{K}_{i,j}(y)=\begin{cases}-1&\mathrm{if}\ i=a,\ \mathrm{and}\ j\in y^{-1}(c_k)\\1&\mathrm{if}\ i=a,\ \mathrm{and}\ j\in y^{-1}(c_{k-1}), j\neq a\\-1&\mathrm{if}\ j=a,\ \mathrm{and}\ i\in y^{-1}(c_k)\\1&\mathrm{if}\ j=a,\ \mathrm{and}\ i\in y^{-1}(c_{k-1}),\ i\neq a\\0 &\mathrm{otherwise}. \end{cases}
\end{eqnarray*}
\begin{claim}The followings are true:
\begin{enumerate}
\item $\mathcal{X}_{a}=0$ for $a\in H_{k-1}$.
\item For each $a\in H_{k-1}$, the variables $\mathcal{Y}_{a}$ and $\mathcal{Z}_{a}$ are independent.
\end{enumerate}
\end{claim}

\begin{proof}It is straightforward to check (1). (2) holds because $\mathcal{Y}\cap\mathcal{Z}=\emptyset$.
\end{proof}

For $s\in H_{k-1}$, let $\mathcal{Y}^{s}\subseteq \mathcal{Y}$ be defined by
\begin{eqnarray*}
\mathcal{Y}^s=\{\alpha=(\alpha_1,\alpha_2)\in \mathcal{Y}:\alpha_1=s,\ \mathrm{or}\ \alpha_2=s\}.
\end{eqnarray*}
Note that for $s_1,s_2\in H_{k-1}$ and $s_1\neq s_2$, $\mathcal{Y}^{s_1}\cap \mathcal{Y}^{s_2}=\emptyset$. Moreover, $\mathcal{Y}=\cup_{s\in H_{k-1}}\mathcal{Y}^s$. Therefore
\begin{eqnarray*}
\mathcal{Y}_{a}=\sum_{s\in H_{k-1}}\langle\mathbf{W}_{\mathcal{Y}^s},\mathbf{K}(y^{(a)})-\mathbf{K}(y) \rangle
\end{eqnarray*}
Note also that $\langle\mathbf{W}_{\mathcal{Y}^s},\mathbf{K}(y^{(a)})-\mathbf{K}(y) \rangle=0$, if $s\neq a$. Hence
\begin{eqnarray*}
\mathcal{Y}_{a}=\sum_{\alpha\in\mathcal{Y}^a}[\mathbf{W}(\alpha)]\cdot\{[\mathbf{K}(y^{(a)})-\mathbf{K}(y)](\alpha)\}
\end{eqnarray*}
Note that for $\alpha\in \mathcal{Y}^a$,
\begin{eqnarray*}
[\mathbf{K}(y^{(a)})-\mathbf{K}(y)](\alpha)=\begin{cases}-1 &\mathrm{if}\ \{\alpha_1,\alpha_2\}\cap [y^{-1}(c_k)]=1 \\1 &\mathrm{if}\ \{\alpha_1,\alpha_2\}\cap [y^{-1}(c_{k-1})]=2\\0 &\mathrm{else}.\end{cases}
\end{eqnarray*}
So, 
\begin{eqnarray*}
\sum_{\alpha\in \mathcal{Y}^a}[\mathbf{W}(\alpha)]\cdot\{[\mathbf{K}(y^{(a)})-\mathbf{K}(y)](\alpha)\}
=\sum_{\alpha\in \mathcal{Y}^a;\{\alpha_1,\alpha_2\}\cap [y^{-1}(c_k)]=2}[\mathbf{W}(\alpha)]-\sum_{\alpha\in \mathcal{Y}^a;\{\alpha_1,\alpha_2\}\cap [y^{-1}(c_k)]=1}[\mathbf{W}(\alpha)]
\end{eqnarray*}
$\{\mathcal{Y}_s\}_{s\in H_{k-1}}$ is a collection of independent centered Gaussian random variables. Moreover, the variance of $\mathcal{Y}_s$ is equal to $2(n_k(y)+n_{k-1}(y)-h)$ independent of the choice of $s$.

By the claim, we obtain
\begin{eqnarray*}
\langle\mathbf{W},\mathbf{K}(y^{(a)})-\mathbf{K}(y) \rangle=\mathcal{Y}_a+\mathcal{Z}_{a}
\end{eqnarray*}
Moreover, 
\begin{eqnarray*}
\max_{a\in H_{k-1}}\mathcal{Y}_a+\mathcal{Z}_{a}&\geq& \max_{a\in H_{k-1}} \mathcal{Y}_a-\max_{a\in H_{k-1}}(-\mathcal{Z}_{a})
\end{eqnarray*}

By  Lemma \ref{mg} about the tail bound result of the maximum of Gaussian random variables, if (\ref{epn}) holds with $N$ replaced by $h$, the event
\begin{eqnarray*}
E^*_1:=\left\{\max_{a\in H_{k-1}}\mathcal{Y}_a\geq (1-\epsilon)\sqrt{2\log h\cdot 2\left(n_k(y)+n_{k-1}(y)-h\right)}\right\}
\end{eqnarray*}
has probability at least $1-e^{-h^{\epsilon}}$; and the event
\begin{eqnarray*}
E^*_2:=\left\{\max_{a\in H_{k-1}}\mathcal{Z}_{a}\leq (1+\epsilon)\sqrt{2\log h\cdot \max_{a\in H_{k-1}} \mathrm{Var}(\mathcal{Z}_{a})}\right\}
\end{eqnarray*}
has probability $1-h^{-\epsilon}$. 

Note that
\begin{eqnarray}
\langle\mathbf{K}(y^{(a)}),\mathbf{K}(y^{(a)})-\mathbf{K}(y) \rangle&=&-\sum_{i,j\in [k]}[t_{i,j}(y,y^{(a)})]^2+\sum_{i\in [k]}[n_i(y)]^2\label{kn2}\\
&=&2n_{k-1}(y)-2\notag
\end{eqnarray}
 Moreover, by (\ref{kn1}) and (\ref{kn2})
\begin{eqnarray*}
\mathrm{Var} \mathcal{Z}_{a}&=&\|\mathbf{K}(y^{(a)})-\mathbf{K}(y)\|^2_{F}-\mathrm{Var}(\mathcal{Y}_a)\\
&=&2\left[n_{k-1}(y)+n_k(y)\right]-2-2\left(n_{k-1}(y)+n_k(y)-h\right)\\
&=& 2h-2
\end{eqnarray*}
Let 
\begin{eqnarray*}
E^*:=\left\{\max_{a\in H_{k-1}}\mathcal{Y}_a+\mathcal{Z}_{a}\geq \left(1-\epsilon-(1+\epsilon)\sqrt{\frac{h-1}{n_k(y)+n_{k-1}(y)-h}}\right)\sqrt{4\log h(n_k(y)+n_{k-1}(y)-h)}\right\}
\end{eqnarray*}
Then $E_1^*\cap E_2^*\subseteq E^*$. Moreover, by (\ref{hkm1}) and (\ref{bll}) we have
\begin{eqnarray*}
&&\left(1-\epsilon-(1+\epsilon)\sqrt{\frac{h-1}{n_k(y)+n_{k-1}(y)-h}}\right)\sqrt{4\log h(n_k(y)+n_{k-1}(y)-h)}\\
&\geq &2\left(1-\epsilon-\frac{1+\epsilon}{\log n}\right)\sqrt{\beta\log n \left(n_k(y)+n_{k-1}(y)\right)}\sqrt{\left(1-\frac{1}{\log^2 n}\right)\left(1-\frac{2\log\log n}{\beta\log n}\right)}
\end{eqnarray*}

Let $\tilde{E}^*$ be the event defined by
\begin{eqnarray*}
\tilde{E}^*:=\left\{\max_{a\in H_{k-1}}\mathcal{Y}_a+\mathcal{Z}_{a}\geq
2\left(1-\epsilon-\frac{1+\epsilon}{\log n}\right)\sqrt{\beta\log n \left(n_k(y)+n_{k-1}(y)\right)}\sqrt{\left(1-\frac{1}{\log^2 n}\right)\left(1-\frac{2\log\log n}{\beta\log n}\right)}\right\}.
\end{eqnarray*}
Then $E^*\subseteq \tilde{E}^*$.

Let $\epsilon$ be given as in (\ref{eph}), then when $n$ is sufficiently large (\ref{epn}) holds with $N$ replaced by $h$.
When (\ref{acc}) and (\ref{sbb}) holds,
\begin{eqnarray*}
&&\mathrm{Pr}\left(\mathrm{max}_{a\in H_{k-1} }\sigma\langle\mathbf{W},\mathbf{K}(y^{(a)})-\mathbf{K}(y) \rangle>n_k(y)+n_{k-1}(y)-1\right)\geq \mathrm{Pr}(\tilde{E}^*)\geq \mathrm{Pr}(E^*),
\end{eqnarray*}
as $n$ is sufficiently large. Moreover, we have
\begin{eqnarray*}
\mathrm{Pr}(E^*)&\geq& \mathrm{Pr}(E_1^*\cap E_2^*)\\
&\geq &1-\mathrm{Pr}(E_1^c)-\mathrm{Pr}(E_2^c)\\
&\geq &1-\log n,
\end{eqnarray*}
when $n$ is sufficiently large. Then the proposition follows.

\end{proof}

\section{Proof of Theorem \ref{m2} when the number of vertices in each color is fixed in the sample space}\label{pm22}

We now discuss the MLE in Theorem \ref{m2} when the number of vertices of each color in the sample space is fixed to be the same as the true value. Let $y\in \Omega_{n_1,\ldots,n_k}$ be the unknown true color assignment mapping. Let $\mathbf{K}(y)$ be defined as in (\ref{kij}), and let $\mathbf{R}$ be defined as in (\ref{rgw}). For given sample $\mathbf{R}$, let $\ol{y}$ be defined as in (\ref{by}).

Given a sample $\mathbf{R}$, the goal is to find the exact division of $[n]$ into groups of vertices such that all the vertices in the same group have the same color under the true color assignment mapping $y$. Note that for all $x\in \Omega_{n_1,\ldots,n_k}$, 
\begin{eqnarray*}
\|\mathbf{K}(x)\|_F^2=2\sum_{1\leq i<j\leq k} n_in_j,
\end{eqnarray*}
which depends only on $n_1,\ldots,n_k$, but is independent of $x$.
Then we have
\begin{eqnarray*}
\ol{y}=\mathrm{argmax}_{x\in \Omega_{n_1,\ldots,n_k}}\langle \mathbf{K}(x),\mathbf{R} \rangle
\end{eqnarray*}

We define an equivalence relation on $\Omega_{n_1,\ldots,n_k}$ as follows:

\begin{definition}\label{dfeqs}
For $y\in \Omega_{n_1,\ldots,n_k}$, let $C^*(y)$ consist of all the $x\in \Omega_{n_1,\ldots,n_k}$ such that $x$ can be obtained from $y$ by a permutation of colors.  More precisely, $x\in C^*(y)\subset \Omega_{n_1,\ldots,n_k}$ if and only if the following condition holds
\begin{itemize}
\item for $i,j\in[n]$, $y(i)=y(j)$ if and only if $x(i)=x(j)$.
\end{itemize}

We define an equivalence relation on $\Omega_{n_1,\ldots,n_k}$ as follows: we say $x,z\in \Omega_{n_1,\ldots,n_k}$ are equivalent if and only if $x\in C^*(z)$. Let $\ol{\Omega}_{n_1,\ldots,n_k}$ be the set of all the equivalence classes in $\Omega_{n_1,\ldots,n_k}$.
\end{definition}

It is not hard to check that for each $x\in\Omega_{n_1,\ldots,n_k}$, $C^*(x)=C(x)\cap \Omega_{n_1,\ldots,n_k}$, where $C(x)$ is defined as in Definition \ref{dfeq}. Moreover, for $x,z\in \Omega_{n_1,\ldots,n_k}$, $x\in C^*(z)$ if and only if there exists a permutation $\omega$ of $[k]$, such that
\begin{eqnarray*}
x=\omega(z)
\end{eqnarray*}
and for any $i\in[k]$,
\begin{eqnarray*}
|z^{-1}(i)|=|z^{-1}(\omega(i))|.
\end{eqnarray*}

Let 
\begin{eqnarray}
p(\ol{y},\sigma)=\mathrm{Pr}(\ol{y}\in C^*(y)),\label{phy}
\end{eqnarray}
where $y\in \Omega_{n_1,\ldots,n_k}$ is the true color assignment function.

For each $x\in \Omega_{n_1,\ldots,n_k}$, define
\begin{eqnarray}
h(x)=\langle \mathbf{K}(x),\mathbf{R} \rangle\label{dhx}
\end{eqnarray}

\begin{lemma}Let $y\in \Omega_{n_1,\ldots,n_k}$ be the true color assignment mapping. If $x,z\in \Omega_{n_1,\ldots,n_k}$ is such that $x\in C^*(z)$, then
\begin{eqnarray}
\mathbf{K}(x)=\mathbf{K}(z);\label{621}
\end{eqnarray}
and
\begin{eqnarray}
h(x)=h(z).\label{622}
\end{eqnarray}
\end{lemma}

\begin{proof}From (\ref{dhx}) we see that (\ref{622}) follows from (\ref{621}). Recall that $\mathbf{K}_{i,j}(x) = 1$ if and only if $x(i)= x(j)$; if $x(i)\neq x(j)$, $\mathbf{K}_{i,j}(x)=0$. Since x and
z are equivalent $x(i)=x(j)$ if and only if $z(i)=z(j)$, therefore $\mathbf{K}(x)=\mathbf{K}(z)$.
\end{proof}

 Then 
\begin{eqnarray*}
p(\ol{y},\sigma)=\mathrm{Pr}\left(h(y)>\max_{x\in \Omega_{n_1,\ldots,n_k}\setminus C^*(y)}h(x)\right)
\end{eqnarray*}
Note that
\begin{eqnarray*}
h(x)-h(y)=\langle\mathbf{K}(y), \mathbf{K}(x)-\mathbf{K}(y) \rangle+\sigma\langle\mathbf{W}, \mathbf{K}(x)-\mathbf{K}(y) \rangle.
\end{eqnarray*}
The expression above shows that $h(x)-h(y)$ is a Gaussian random variable with mean $\langle\mathbf{K}(y), \mathbf{K}(x)-\mathbf{K}(y) \rangle$ and variance $\sigma^2\|\mathbf{K}(x)-\mathbf{K}(y)\|_F^2$. 

For $i,j\in[k]$, $x,y\in \Omega_{n_1,\ldots,n_k}$, let $S_{i,j}(x,y)$ be defined as in (\ref{1sij}). Let $t_{i,j}(x,y)=|S_{i,j}(x,y)|$. Let 
\begin{eqnarray*}
U(x,y):=-\mathbb{E}[h(x)-h(y)]=\langle \mathbf{K}(y),\mathbf{K}(y)-\mathbf{K}(x)\rangle;
\end{eqnarray*}
Then by (\ref{lak}) we have
\begin{eqnarray}
U(x,y)=\sum_{i\in[k]}[n_i]^2-\sum_{i,j\in[k]}[t_{i,j}(x,y)]^2\label{uxy}
\end{eqnarray}
and 
\begin{eqnarray*}
\|\mathbf{K}(x)-\mathbf{K}(y)\|_F^2=2U(x,y).
\end{eqnarray*}

For $x\in \Omega_{n_1,\ldots,n_k} \setminus C^*(y)$
\begin{eqnarray*}
\mathrm{Pr}\left(h(x)-h(y)>0\right)
&=&\mathrm{Pr_{\xi\sim \mathcal{N}(0,1)}}\left(\xi\geq \frac{\sqrt{U(x,y)}}{\sqrt{2}\sigma}\right)
\end{eqnarray*}
Using the standard Gaussian tail bound $\mathrm{Pr}_{\xi\in \mathcal{N}(0,1)}(\xi>x)<e^{-\frac{1}{2}x^2}$, we obtain
\begin{eqnarray*}
\mathrm{Pr}\left(h(x)-h(y)>0\right)\leq e^{-\frac{U(x,y)}{4\sigma^2}}
\end{eqnarray*}

Then
\begin{eqnarray}
1-p(\ol{y};\sigma)&\leq& \sum_{x\in \Omega\setminus C^*(y)}\mathrm{Pr}(h(x)-h(y)> 0)\label{2mp}\\
&\leq &\sum_{x\in \Omega \setminus C^*(y)} e^{-\frac{U(x,y)}{4\sigma^2}}\notag
\end{eqnarray}

\begin{lemma}\label{lm63}For any $x,y\in\Omega_{n_1,\ldots,n_k}$, $U(x,y)\geq 0$, where the equality holds if and only if $x\in C^*(y)$.
\end{lemma}

\begin{proof}Note that
\begin{eqnarray}
U(x,y)=\left.\frac{1}{2}L(x,y)\right|_{x,y\in \Omega_{n_1,\ldots,n_k}}.\label{ulr}
\end{eqnarray}
Then the lemma follows from Lemma \ref{lgz}.
\end{proof}

\begin{lemma}\label{lm64}Assume $x,z_1,z_2\in \Omega_{n_1,\ldots,n_k}$. Assume $z_1\in C^*(z_2)$. Then
\begin{eqnarray*}
U(x,z_1)=U(x,z_2).
\end{eqnarray*}
\end{lemma}

\begin{proof}The lemma follows from (\ref{lyee}) and (\ref{ulr}).
\end{proof}

\begin{lemma}\label{ll64}Let $\tilde{\mathcal{B}}$ and $\tilde{\mathcal{B}}_{\epsilon}$ be defined as in (\ref{db}) and (\ref{dbe}), respectively. Let $\mathcal{R}_c$ be defined as in (\ref{drc}). Let $x,y\in \Omega_{n_1,\ldots,n_k}$. Assume (\ref{vir}) and (\ref{tme}) hold. Then
\begin{eqnarray*}
U(x,y)\geq \frac{c\epsilon n^2}{2k}
\end{eqnarray*}
\end{lemma}

\begin{proof}The lemma follows from (\ref{ulr}) and Lemma \ref{l55}.
\end{proof}

\begin{lemma}\label{ll65}Let $y\in \Omega_{n_1,\ldots,n_k}$ be the true color assignment mapping. Let $\mathcal{R}_c$ be defined in (\ref{drc}). Assume (\ref{vir}) holds. Let $x\in\Omega_{n_1,\ldots,n_k}$ 
For $i\in [k]$,  let $w(i)\in[k]$ be defined as in (\ref{twi}). Then
\begin{enumerate}
\item when $\epsilon\in\left(0,\frac{c}{k}\right)$ and $(t_{1,1}(x,y),\ldots, t_{k,k}(x,y))\in \tilde{\mathcal{B}}_{\epsilon}$, $w$ is a bijection from $[k]$ to $[k]$.
\item Assume there exist $i,j\in[k]$, such that $n_i\neq n_j$. If
\begin{eqnarray}
\epsilon<\min_{i,j\in[k]:n_i\neq n_j}\left|\frac{n_i-n_j}{n}\right|\label{eu1}
\end{eqnarray}
Then for any $i\in[k]$,
\begin{eqnarray}
n_i=|y^{-1}(i)|=|y^{-1}(w(i))|=n_{w(i)}.\label{esc}
\end{eqnarray}
\end{enumerate}
\end{lemma}

\begin{proof}Part (1) of the lemma is a special case of Lemma \ref{ll56} when $x\in\Omega_{n_1,\ldots,n_k}$. Now we prove Part (2) of the lemma. We shall prove by contradiction. Assume (\ref{esc}) does not hold, then there exists $i_0\in [k]$, such that
\begin{eqnarray}
n_i>n_{w(i)}.\label{ngn}
\end{eqnarray}
By (\ref{twi}), we obtain
\begin{eqnarray*}
n_{w_i}\geq t_{w(i),i}(x,y)>n_i-\epsilon n
\end{eqnarray*}
Hence we have
\begin{eqnarray*}
\epsilon>\frac{n_i-n_{w}(i)}{n}
\end{eqnarray*}
But this is a contradiction to (\ref{eu1}) and (\ref{ngn}). The contradiction implies Part (2) of the lemma.
\end{proof}

We have the following proposition
\begin{proposition}\label{p67}Let $\mathcal{R}_c$ be defined as in (\ref{drc}). Let $y\in\Omega_{n_1,\ldots,n_k}$ be the true color assignment mapping satisfying (\ref{vir}), where $n_1\geq n_2\geq \ldots\geq n_k$. Here $c,k$ may depend on $n$. Assume
\begin{eqnarray}
\lim_{n\rightarrow\infty}\frac{c}{k(n_k+n_{k-1})}=0;\label{ckl}
\end{eqnarray}
\begin{eqnarray}
\lim_{k\rightarrow\infty}\frac{\log k}{\log n}=0;\label{knl}
\end{eqnarray}
and
\begin{eqnarray}
\lim_{n\rightarrow\infty}\frac{cn\min\left\{\frac{c}{k},\min_{i,j\in[k],n_i\neq n_j}\left|\frac{n_i-n_j}{n}\right|\right\}}{(n_{k-1}+n_k)}\geq \alpha>0,\label{ck3}
\end{eqnarray}
where $\alpha>0$ is a constant independent of $n$; and $\min_{i,j\in[k],n_i\neq n_j}\left|\frac{n_i-n_j}{n}\right|=\infty$ if $n_i=n_j$ for all $i,j\in[k]$.
If there exists a constant $\delta>0$ independent of $n$, such that 
\begin{eqnarray}
\sigma^2<\frac{(1-\delta)[n_{k-1}+n_k]}{4\log n}\label{cdcd1}
\end{eqnarray}
then 
\begin{eqnarray*}
\lim_{n\rightarrow\infty}p(\ol{y};\sigma)=1.
\end{eqnarray*}
\end{proposition}

\begin{proof}Let 
\begin{eqnarray*}
\Gamma:=\sum_{x\in \ol{\Omega}_{n_1,\ldots,n_k}\setminus C^*(y)}e^{-\frac{U(x,y)}{4\sigma^2}}.
\end{eqnarray*}
By (\ref{2mp}), it suffices to show that $\lim_{n\rightarrow\infty}\Gamma=0$.

Let $0<\epsilon<\min\left\{\frac{c}{k},\min_{i,j\in[k],n_i\neq n_j}\left|\frac{n_i-n_j}{n}\right|\right\}$. Then
\begin{eqnarray*}
\Gamma\leq \Gamma_1+\Gamma_2;
\end{eqnarray*}
where 
\begin{eqnarray*}
\Gamma_1&=&\sum_{C^*(x)\in\ol{\Omega}_{n_1,\ldots,n_k}:\left(t_{1,1}(x,y),\ldots,t_{k,k}(x,y)\right)\in[\tilde{\mathcal{B}}\setminus\tilde{\mathcal{B}}_{\epsilon}], C(x)\neq C(y)}e^{\frac{-U(x,y)}{4\sigma^2}}
\end{eqnarray*}
and
\begin{eqnarray}
\Gamma_2=\sum_{C^*(x)\in\ol{\Omega}_{n_1,\ldots,n_k}:\left(t_{1,1}(x,y),\ldots, t_{k,k}(x,y)\right)\in\tilde{\mathcal{B}}_{\epsilon}, C(x)\neq C(y)}e^{\frac{-U(x,y)}{4\sigma^2}}.\label{dg2}
\end{eqnarray} 

By Lemma \ref{ll64}, when (\ref{cdcd1}) holds, we have
\begin{eqnarray*}
\Gamma_1&\leq& k^n e^{-\frac{c\epsilon n^2\log n}{2(1-\delta)k(n_{k-1}+n_k)}}
\end{eqnarray*}
When
\begin{eqnarray*}
\epsilon=\frac{1}{2}\min\left\{\frac{c}{k},\min_{i,j\in[k],n_i\neq n_j}\left|\frac{n_i-n_j}{n}\right|\right\},
\end{eqnarray*}
we have
\begin{eqnarray*}
\Gamma_1&\leq &e^{n\left(\log k-\frac{cn\log n\min\left\{\frac{c}{k},\min_{i,j\in[k],n_i\neq n_j}\left|\frac{n_i-n_j}{n}\right|\right\}}{4(1-\delta)(n_{k-1}+n_k)}\right)}
\end{eqnarray*}
Then by (\ref{knl}) and (\ref{ck3}) we obtain that
\begin{eqnarray}
0\leq\lim_{n\rightarrow\infty}\Gamma_1\leq \lim_{n\rightarrow\infty}e^{-n\log k}=0.\label{g1z}
\end{eqnarray}
since $k\geq 2$.

Now let us consider $\Gamma_2$. Recall that $y\in \Omega_{n_1,\ldots,n_k}$ is the true color assignment mapping.  Let $w$ be the bijection from $[k]$ to $[k]$ as defined in (\ref{twi}). Let $y^*\in \Omega$ be defined by
\begin{eqnarray*}
y^*(z)=w(y(z)),\ \forall z\in [n].
\end{eqnarray*}
Then $y^*\in C(y)$. By Part (2) of Lemma \ref{ll65}, we obtain that for $i\in[k]$
\begin{eqnarray*}
\left|[y^*]^{-1}(i)\right|=\left|y^{-1}(w^{-1}(i))\right|=\left|y^{-1}(i)\right|;
\end{eqnarray*}
therefore $y^*\in \Omega_{n_1,\ldots,n_k}$.
 Moreover, $x$ and $y^*$ satisfies
\begin{eqnarray}
t_{i,i}(x,y^*)\geq n_i(y^*)-n\epsilon,\ \forall i\in[k].\label{dxys}
\end{eqnarray}

If $x\neq y^*$, by Lemma \ref{lm33},
 there exists an $l$-cycle $(i_1,\ldots,i_{l})$ for $(x,y^*)$ with $2\leq l\leq k$. Then for each $2\leq s\leq (l+1)$, choose an arbitrary vertex $u_s$ in $S_{i_{s-1},i_s}(x,y)$, and let $y_1(u_s)=c_{i_{s-1}}$, where $i_{l+1}:=i_1$. For any vertex $z\in[n]\setminus\{u_{2},\ldots,u_{l+1}\}$, let $y_1(z)=y^*(z)$.

 Note that $y_1\in \Omega_{n_1,\ldots,n_k}$. More precisely, for $1\leq s\leq l$, we have
\begin{eqnarray*}
t_{i_s,i_s}(x,y^*)+1=t_{i_s,i_s}(x,y_1);\\
t_{i_s,i_{s+1}}(x,y^*)-1=t_{i_s,i_{s+1}}(x,y_1)
\end{eqnarray*}
and
\begin{eqnarray*}
t_{a,b}(x,y^*)=t_{a,b}(x,y_1),\ \forall (a,b)\notin\{(i_s,i_s),(i_s,i_{s+1})\}_{s=1}^{l}.
\end{eqnarray*}
From (\ref{uxy}) and Lemma \ref{lm64} we obtain
\begin{eqnarray*}
U(x,y_1)-U(x,y)&=&U(x,y_1)-U(x,y^*)\\
&=&-\sum_{i,j\in[k]}[t_{i,j}(x,y_1)]^2+\sum_{i,j\in[k]}[t_{i,j}(x,y^*)]^2\\
&=&2\sum_{s=1}^{l}[t_{i_s,i_{s+1}}(x,y^*)-t_{i_s,i_s}(x,y^*)].
\end{eqnarray*}
When $\left(t_{1,1}(x,y),\ldots, t_{k,k}(x,y)\right)\in\tilde{\mathcal{B}}_{\epsilon}$, we obtain
\begin{eqnarray*}
U(x,y_1)-U(x,y)&\leq &2\sum_{s=1}^{l}\left[2n\epsilon-n_{i_s}(y^*)\right]\leq-l\left(n_k+n_{k-1}-4n\epsilon\right)
\end{eqnarray*}

 Therefore
\begin{eqnarray}
e^{-\frac{U(x,y)}{4\sigma^2}}\leq e^{-\frac{U(x,y_1)}{4\sigma^2}}e^{-\frac{l(n_k+n_{k-1}-4n\epsilon)}{4\sigma^2}}.\label{mst}
\end{eqnarray}

If $y_1\neq x$, we find an $l_2$-cycle ($2\leq l_2\leq k$) for $(x,y_1)$, change colors along the $l_2$-cycle  as above, and obtain another cycle assignment mapping $y_2\in \Omega_{n_1,\ldots,n_k}$, and so on. 
Let $y_0:=y$, and note that for each $r\geq 1$, if $y_r$ is obtained from $y_{r-1}$ by changing colors along an $l_r$ cycle, we have
\begin{eqnarray*}
D_{\Omega}(x,y_r)= D_{\Omega}(x,y_{r-1})-l_r
\end{eqnarray*}
Therefore finally we can obtain $x$ from $y$ by changing colors along at most $\left\lfloor \frac{n}{2} \right\rfloor$ cycles. Using similar arguments as those used to derive (\ref{mst}), we obtain that for each $r$
\begin{eqnarray*}
e^{-\frac{U(x,y_{r-1})}{4\sigma^2}}\leq e^{-\frac{U(x,y_r)}{4\sigma^2}}e^{-\frac{l_r(n_k+n_{k-1}-4n\epsilon)}{4\sigma^2}}.
\end{eqnarray*}
Therefore if $y_s=x$ for some $1\leq s\leq \left\lfloor \frac{n}{2} \right\rfloor$, we have
\begin{eqnarray*}
e^{-\frac{U(x,y)}{4\sigma^2}}\leq e^{-\frac{U(x,y_s)}{4\sigma^2}}e^{-\frac{(n_k+n_{k-1}-4n\epsilon)\left(\sum_{i\in[s]}l_i\right)}{4\sigma^2}}.
\end{eqnarray*}
By Lemma \ref{lm63}, we have $U(x,y_s)=U(x,x)=0$, hence
\begin{eqnarray*}
e^{-\frac{U(x,y)}{4\sigma^2}}\leq \prod_{i\in[s]}e^{-\frac{(n_k+n_{k-1}-4n\epsilon)l_i}{4\sigma^2}}.
\end{eqnarray*}


Note also that for any $r_1\neq r_2$, in the process of obtaining $y_{r_1}$ from $y_{r_1-1}$ and the process of obtaining $y_{r_2}$ from $y_{r_2-1}$, we change colors on disjoint sets of vertices. Hence the order of these steps of changing colors along cycles does not affect the final color assignment mapping we obtain.  From (\ref{dg2}) we obtain
\begin{eqnarray}
\Gamma_2\leq \prod_{l=2}^k \left(\sum_{m_l=0}^{\infty} (nk)^{m_l\ell} e^{-\frac{m_ll(n_k+n_{k-1}-4n\epsilon)}{4\sigma^2}}\right)-1\label{iup}
\end{eqnarray}
On the right hand side of (\ref{iup}), when expanding the product, each summand has the form
\begin{eqnarray*}
\left[(nk)^{2m_2}e^{-\frac{2m_2(n_k+n_{k-1}-4n\epsilon)}{4\sigma^2}}\right]\cdot\left[ (nk)^{3m_3}e^{-\frac{3m_3(n_k+n_{k-1}-4n\epsilon)}{4\sigma^2}}\right]\cdot\ldots\cdot\left[(nk)^{km_k} e^{-\frac{km_k(n_k+n_{k-1}-4n\epsilon)}{4\sigma^2}}\right]
\end{eqnarray*}
where the factor $\left[(nk)^{2m_2}e^{-\frac{2m_2(n_k+n_{k-1}-4n\epsilon)}{4\sigma^2}}\right]$ represents that we changed along 2-cycles $m_2$ times, the factor $\left[ (nk)^{3m_3}e^{-\frac{3m_3(n_k+n_{k-1}-4n\epsilon)}{4\sigma^2}}\right]$ represents that we changed along 3-cycles $m_3$ times, and so on. Moreover, each time we changed along an $l$-cycle, we need to first determine the $l$ different colors involved in the $l$-cycle, and there are at most $k^l$ different $l$-cycles;  we then need to chose $l$ vertices to change colors, and there are at most $n^{l}$ choices.
It is straightforward to check that if $\sigma$ satisfies (\ref{cdcd1}), then
\begin{eqnarray*}
nk e^{-\frac{n_k+n_{k-1}-4\epsilon}{4\sigma^2}}\leq e^{\log k-\frac{\log n}{1-\delta}\left(\delta-\frac{4\epsilon}{n_k+n_{k-1}}\right)}.
\end{eqnarray*}
When $\epsilon<\frac{c}{k}$ and (\ref{ckl}) holds, we obtain
\begin{eqnarray*}
e^{\log k-\frac{\log n}{1-\delta}\left(\delta-\frac{4\epsilon}{n_k+n_{k-1}}\right)}&\leq& e^{\log k-\frac{\log n}{1-\delta}\left(\delta-\frac{4c}{k(n_k+n_{k-1})}\right)}\\
&\leq & e^{\log k-\frac{\delta \log n}{2(1-\delta)}}
\end{eqnarray*}
When (\ref{knl}) holds, we obtain
\begin{eqnarray*}
e^{\log k-\frac{\delta \log n}{2(1-\delta)}}\leq e^{-\frac{\delta \log n}{4(1-\delta)}}\rightarrow 0
\end{eqnarray*}
as $n\rightarrow \infty$.

 Therefore we have
\begin{eqnarray*}
\sum_{m_l=0}^{\infty} (nk)^{m_l\ell} e^{-\frac{m_ll(n_k+n_{k-1}-4n\epsilon)}{4\sigma^2}}\leq \frac{1}{1-e^{l\log k-\frac{l\log n}{1-\delta}\left(\delta-\frac{4\epsilon}{n_k+n_{k-1}}\right)}}\leq \frac{1}{1-e^{-\frac{\delta l n}{4(1-\delta)}}}
\end{eqnarray*}
Let 
\begin{eqnarray*}
\Sigma:= \prod_{l=2}^k \left(\sum_{m_l=0}^{\infty} (nk)^{m_l\ell} e^{-\frac{m_ll (n_k+n_{k-1}-4\epsilon)}{4\sigma^2}}\right).
\end{eqnarray*}
Since $\log(1+x)\leq x$ for $x\geq 0$, we have
\begin{eqnarray*}
0\leq \log \Sigma&=&\sum_{l=2}^{k}\log \left(1+\sum_{m_l=1}^{\infty} (nk)^{m_l\ell} e^{-\frac{m_ll (n_k+n_{k-1}-4n\epsilon)}{4\sigma^2}}\right)\\
&\leq &\sum_{l=2}^{k}\sum_{m_l=1}^{\infty} (nk)^{m_l\ell} e^{-\frac{m_ll(n_k+n_{k-1}-4n\epsilon)}{4\sigma^2}}\\
&\leq &\sum_{l=2}^{k}\frac{\left(e^{-\frac{\delta\log n}{4(1-\delta)}}\right)^{l}}{1-\left(e^{-\frac{\delta\log n}{4(1-\delta)}}\right)^l}\\
&\leq &\frac{\left(e^{-\frac{\delta\log n}{4(1-\delta)}}\right)^{2}}{\left[1-\left(e^{-\frac{\delta\log n}{4(1-\delta)}}\right)^2\right]\left(1-e^{-\frac{\delta\log n}{4(1-\delta)}}\right)}\rightarrow 0,
\end{eqnarray*}
as $n\rightarrow\infty$. Then
\begin{eqnarray}
0\leq \lim_{n\rightarrow\infty}\Gamma_2\leq \lim_{n\rightarrow\infty}e^{\log \Sigma}-1=0.\label{g2z}
\end{eqnarray}

Then the proposition follows from (\ref{g1z}) and (\ref{g2z}).
\end{proof}

\begin{proposition}\label{p68}
Assume $n_1\geq n_2\geq n_k$, and $y\in \Omega_{n_1,\ldots,n_k}$ is the true color assignment mapping. Suppose that
\begin{eqnarray}
\frac{\log n_k}{\log n}\geq \beta>0,\ \forall\ n,\label{nkb}
\end{eqnarray}
where $\beta$ is a constant independent of $n$.

Assume
\begin{eqnarray}
\delta=\left(1+\frac{4}{\beta}\right)\frac{\log\log n}{\log n},\label{dtb}
\end{eqnarray}
and
\begin{eqnarray}
\sigma^2>\frac{(1+\delta)[n_k(y)+n_{k-1}(y)]}{4\beta \log n},\label{sgb}
\end{eqnarray}
then 
\begin{eqnarray*}
\lim_{n\rightarrow\infty}p(\ol{y};\sigma)=0.
\end{eqnarray*}
\end{proposition}

\begin{proof}For $y\in \Omega_{n_1,\ldots,n_k}$, $a,b\in[n]$ such that $c_k=y(a)\neq y(b)=c_{k-1}$. Let $y^{(ab)}\in \Omega_{n_1,\ldots,n_k}$ be the coloring of vertices defined by
\begin{eqnarray}
y^{(ab)}(i)=\begin{cases}y(i)& \mathrm{if}\ i\in[n]\setminus\{a,b\}\\ c_{k-1}&\mathrm{if}\ i=a\\ c_k& \mathrm{if}\ i=b \end{cases}\label{yab1}
\end{eqnarray}
Then 
\begin{eqnarray*}
t_{k-1,k}(y^{(ab)},y)-1=t_{k-1,k}(y,y)=0\\
t_{k-1,k-1}(y^{(ab)},y)+1=t_{k-1,k-1}(y,y)=n_{k-1}\\
t_{k,k-1}(y^{(ab)},y)-1=t_{k,k-1}(y,y)=0\\
t_{k,k}(y^{(ab)},y)+1=t_{k,k}(y,y)=n_k.
\end{eqnarray*}
and
\begin{eqnarray*}
t_{i,j}(y^{(ab)},y)=t_{i,j}(y),\ \forall\ (i,j)\in \left([k]^2\setminus\{(k,k),(k,k-1),(k-1,k),(k-1,k-1)\}\right)
\end{eqnarray*}
Note that 
\begin{eqnarray*}
1-p(\ol{y};\sigma)\geq \mathrm{Pr}\left(\cup_{a,b\in[n],c_k=y(a)\neq y(b)=c_{k-1}}[h(y^{(ab)})-h(y)>0]\right),
\end{eqnarray*}
since any of the event $[h(y^{(ab)})-h(y)>0]$ implies $\ol{y}\neq y$. By (\ref{lak}) we obtain
\begin{eqnarray*}
h(y^{(ab)})-h(y)&=&\langle \mathbf{K}(y),\mathbf{K}(y^{(ab)})-\mathbf{K}(y)\rangle+\sigma\langle\mathbf{W},\mathbf{K}(y^{(ab)})-\mathbf{K}(y) \rangle\\
&=&\sum_{i,j\in[k]}[t_{i,j}(y^{(ab)},y)]^2-\sum_{i\in[k]}[n_i]^2+\sigma\langle\mathbf{W},\mathbf{K}(y^{(ab)})-\mathbf{K}(y) \rangle\\
&=&4-2n_k-2n_{k-1}+\sigma\langle\mathbf{W},\mathbf{K}(y^{(ab)})-\mathbf{K}(y) \rangle.
\end{eqnarray*}
So $1-p(\ol{y};\sigma)$ is at least
\begin{eqnarray*}
&&\mathrm{Pr}\left(\cup_{a,b\in[n],c_k=y(a)\neq y(b)=c_{k-1}}[h(y^{(ab)})-h(y)>0]\right)\\
&\geq &\mathrm{Pr}\left(\mathrm{max}_{a,b\in[n],c_k=y(a)\neq y(b)=c_{k-1}}\sigma\langle\mathbf{W},\mathbf{K}(y^{(ab)})-\mathbf{K}(y) \rangle>2(n_k+n_{k-1}-2)\right)
\end{eqnarray*}
For $i\in\{k-1,k\}$, let $H_i\subset y^{-1}(c_i)$ such that 
\begin{eqnarray}
|H_i|=\frac{n_k}{\log^2n}=h. \label{ddh}
\end{eqnarray}
Then
\begin{eqnarray*}
1-p(\ol{y};\sigma)\geq \mathrm{Pr}\left(\mathrm{max}_{a\in H_{k-1},b\in H_k}\sigma\langle\mathbf{W},\mathbf{K}(y^{(ab)})-\mathbf{K}(y) \rangle>2(n_k+n_{k-1}-2)\right)
\end{eqnarray*}
Let $(\mathcal{X},\mathcal{Y},\mathcal{Z})$ be a partition of $[n]^2$ defined by
\begin{eqnarray*}
&&\mathcal{X}=\{\alpha=(\alpha_1,\alpha_2)\in [n]^2, \{\alpha_1,\alpha_2\}\cap [H_{k-1}\cup H_k]=\emptyset\}\\
&&\mathcal{Y}=\{\alpha=(\alpha_1,\alpha_2)\in [n]^2, |\{\alpha_1,\alpha_2\}\cap [H_{k-1}\cup H_k]|=1\}\\
&&\mathcal{Z}=\{\alpha=(\alpha_1,\alpha_2)\in [n]^2, |\{\alpha_1,\alpha_2\}\cap [H_{k-1}\cup H_k]|=2\}
\end{eqnarray*}
For $\eta\in\{\mathcal{X},\mathcal{Y},\mathcal{Z}\}$, define the $n\times n$ matrix $\mathbf{W}_{\eta}$ from the entries of $\mathbf{W}$ as follows
\begin{eqnarray*}
\mathbf{W}_{\eta}(i,j)=\begin{cases}0&\mathrm{if}\ (i,j)\notin \eta\\ \mathbf{W}(i,j),&\mathrm{if}\ (i,j)\in \eta\end{cases}
\end{eqnarray*}
For each $a\in H_{k}$ and $b\in H_{k-1}$, let
\begin{eqnarray*}
\mathcal{X}_{ab}=\langle\mathbf{W}_{\mathcal{X}},\mathbf{K}(y^{(ab)})-\mathbf{K}(y) \rangle\\
\mathcal{Y}_{ab}=\langle\mathbf{W}_{\mathcal{Y}},\mathbf{K}(y^{(ab)})-\mathbf{K}(y) \rangle\\
\mathcal{Z}_{ab}=\langle\mathbf{W}_{\mathcal{Z}},\mathbf{K}(y^{(ab)})-\mathbf{K}(y) \rangle
\end{eqnarray*}

\begin{lemma}\label{l69}The followings are true:
\begin{enumerate}
\item $\mathcal{X}_{ab}=0$ for $a\in H_{k}$ and $b\in H_{k-1}$.
\item For each $a\in H_{k}$ and $b\in H_{k-1}$, the variables $\mathcal{Y}_{ab}$ and $\mathcal{Z}_{ab}$ are independent.
\item Each $\mathcal{Y}_{ab}$ can be decomposed into $Y_a+Y_b$ where $\{Y_a\}_{a\in H_k}\cup \{Y_b\}_{b\in H_{k-1}}$ is a collection of i.i.d.~Gaussian random variables.
\end{enumerate}
\end{lemma}

\begin{proof}
Note that for $i,j\in[n]$,
\begin{eqnarray}
\mathbf{K}_{i,j}(y^{(ab)})-\mathbf{K}_{i,j}(y)=\begin{cases}1 &\mathrm{if}\ i=a, j\in y^{-1}(c_k),\ \mathrm{and}\ j\neq a. \\ -1&\mathrm{if}\ i=a, j\in y^{-1}(c_{k-1})\\ -1&\mathrm{if}\ i=b, j\in y^{-1}(c_k)\\ 1 & \mathrm{if}\ i=b, j\in y^{-1}(c_{k-1})\ \mathrm{and}\ j\neq b\\1 &\mathrm{if}\ j=a, i\in y^{-1}(c_k),\ \mathrm{and}\ i\neq a. \\ -1&\mathrm{if}\ j=a, i\in y^{-1}(c_{k-1})\\ -1&\mathrm{if}\ j=b, i\in y^{-1}(c_k)\\ 1 & \mathrm{if}\ j=b, i\in y^{-1}(c_{k-1})\ \mathrm{and}\ i\neq b\\ 0&\mathrm{otherwise}.\end{cases}\label{kab1}
\end{eqnarray}

It is straightforward to check (1). (2) holds because $\mathcal{Y}\cap\mathcal{Z}=\emptyset$.

For $s\in H_{k-1}\cup H_{k}$, let $\mathcal{Y}_s\subseteq \mathcal{Y}$ be defined by
\begin{eqnarray*}
\mathcal{Y}_s=\{\alpha=(\alpha_1,\alpha_2)\in \mathcal{Y}:\alpha_1=s,\ \mathrm{or}\ \alpha_2=s\}.
\end{eqnarray*}
Note that for $s_1,s_2\in H_{k-1}\cup H_k$ and $s_1\neq s_2$, $\mathcal{Y}_{s_1}\cap \mathcal{Y}_{s_2}=\emptyset$. Moreover, $\mathcal{Y}=\cup_{s\in H_{k-1}\cup H_k}\mathcal{Y}_s$. Therefore
\begin{eqnarray*}
\mathcal{Y}_{ab}=\sum_{s\in H_{k-1}\cup H_k}\langle\mathbf{W}_{\mathcal{Y}_s},\mathbf{K}(y^{(ab)})-\mathbf{K}(y) \rangle
\end{eqnarray*}
Note also that $\langle\mathbf{W}_{\mathcal{Y}_s},\mathbf{K}(y^{(ab)})-\mathbf{K}(y) \rangle=0$, if $s\notin \{a,b\}$. Hence
\begin{eqnarray*}
\mathcal{Y}_{ab}=\sum_{\alpha\in\mathcal{Y}_a\cup \mathcal{Y}_b}[\mathbf{W}(\alpha)]\cdot\{[\mathbf{K}(y^{(ab)})-\mathbf{K}(y)](\alpha)\}
\end{eqnarray*}
From (\ref{kab1}) we obtain that for $\alpha=(\alpha_1,\alpha_2)\in \mathcal{Y}_a$ and $\alpha_1\neq \alpha_2$,
\begin{eqnarray*}
[\mathbf{K}(y^{(ab)})-\mathbf{K}(y)](\alpha)=\begin{cases}1 &\mathrm{if}\ |\{\alpha_1,\alpha_2\}\cap y^{-1}(c_{k})|=2.\\-1 &\mathrm{if}\ |\{\alpha_1,\alpha_2\}\cap y^{-1}(c_{k-1})|=1.\end{cases}
\end{eqnarray*}
So, we can define
\begin{eqnarray*}
Y_a:=&&\sum_{\alpha\in \mathcal{Y}_a}[\mathbf{W}(\alpha)]\cdot\{[\mathbf{K}(y^{(ab)})-\mathbf{K}(y)](\alpha)\}\\
&=&\sum_{\alpha\in \mathcal{Y}_a;|\{\alpha_1,\alpha_2\}\cap y^{-1}(c_k)|=2}[\mathbf{W}(\alpha)]-\sum_{\alpha\in \mathcal{Y}_a;|\{\alpha_1,\alpha_2\}\cap y^{-1}(c_{k-1})|=1}[\mathbf{W}(\alpha)]\\
\end{eqnarray*}
Similarly, define
\begin{eqnarray*}
Y_b:=\sum_{\alpha\in \mathcal{Y}_b;|\{\alpha_1,\alpha_2\}\cap y^{-1}(c_{k-1})|=2}[\mathbf{W}(\alpha)]
-\sum_{\alpha\in \mathcal{Y}_b;|\{\alpha_1,\alpha_2\}\cap y^{-1}(c_{k})|=1}[\mathbf{W}(\alpha)]
\end{eqnarray*}
Then $\mathcal{Y}_{ab}=Y_a+Y_b$ and $\{Y_s\}_{s\in H_{k-1}\cup H_k}$ is a collection of independent Gaussian random variables. Moreover, the variance of $Y_s$ is  $2(n_k+n_{k-1}-2h)$ independent of the choice of $s$.
\end{proof}

By the Lemma \ref{l69}, we obtain
\begin{eqnarray*}
\langle\mathbf{W},\mathbf{K}(y^{(ab)})-\mathbf{K}(y) \rangle=Y_a+Y_b+\mathcal{Z}_{ab}
\end{eqnarray*}
Moreover,
\begin{eqnarray*}
\max_{a\in H_{k},b\in H_{k-1}}Y_a+Y_b+\mathcal{Z}_{ab}&\geq& \max_{a\in H_k,b\in H_{k-1}}(Y_a+Y_b)-\max_{a\in H_k,b\in H_{k-1}}(-\mathcal{Z}_{ab})\\
&=&\max_{a\in H_k} Y_a+\max_{b\in H_{k-1}}Y_b-\max_{a\in H_k,b\in H_{k-1}}(-\mathcal{Z}_{ab})
\end{eqnarray*}

By Lemma \ref{mg} we obtain that  when $\epsilon, h$ satisfy (\ref{epn}) with $N$ replaced by $h$, each one of the following two events
\begin{eqnarray*}
F_1^*:=\left\{\max_{a\in H_k}Y_a\geq (1-\epsilon)\sqrt{2\log h\cdot 2(n_k+n_{k-1}-2h)}\right\}\\
F_2^*:=\left\{\max_{b\in H_{k-1}}Y_b\geq (1-\epsilon)\sqrt{2\log h\cdot 2(n_k+n_{k-1}-2h)}\right\}
\end{eqnarray*}
has probability at least $1-e^{-h^{\epsilon}}$. Moreover, the event 
\begin{eqnarray*}
F_3^*:=\left\{\max_{a\in H_k,b\in H_{k-1}}\mathcal{Z}_{ab}\leq (1+\epsilon)\sqrt{4\log h\cdot \max \mathrm{Var}(Z_{ab})}\right\}
\end{eqnarray*}
occurs with probability at least $1-h^{-2\epsilon}$. Then by (\ref{kab1}) we have
\begin{eqnarray*}
\mathrm{Var} \mathcal{Z}_{ab}&=&\|\mathbf{K}(y^{(ab)})-\mathbf{K}(y)\|^2_{F}-\mathrm{Var}(Y_a)-\mathrm{Var}(Y_b)\\
&=&4(n_k+n_{k-1})-6-4\left(n_k+n_{k-1}-2h\right)\\
&=& 8h-6
\end{eqnarray*}
Hence the probability of the event
\begin{eqnarray*}
&&F^*:=\\
&&\left\{\max_{a\in H_k,b\in H_{k-1}}\langle\mathbf{W},\mathbf{K}(y^{(ab)})-\mathbf{K}(y) \rangle\geq 4(1-\epsilon)\sqrt{\log h(n_k+n_{k-1}-2h)}-4(1+\epsilon)\sqrt{\left(2h-\frac{3}{2}\right)\log h}\right\}
\end{eqnarray*}
is at least 
\begin{eqnarray*}
\mathrm{Pr}(F_1^*\cap F_2^*\cap F_3^*)&=&1-\mathrm{Pr}([F_1^*]^c\cup [F_2^*]^c\cup [F_3^*]^c)\\
&\geq &1- \mathrm{Pr}([F_1^*]^c)-\mathrm{Pr}([F_2^*]^c)-\mathrm{Pr}([F_3^*]^c)\\
&\geq &1-2e^{-h^{\epsilon}}-h^{-2\epsilon}.
\end{eqnarray*}
Moreover, from (\ref{ddh}) we obtain
\begin{eqnarray*}
&&4(1-\epsilon)\sqrt{\log h(n_k+n_{k-1}-2h)}-4(1+\epsilon)\sqrt{\left(2h-\frac{3}{2}\right)\log h}\\
&=& 4\sqrt{(n_k+n_{k-1}-2h)\log h}\left[1-\epsilon-(1+\epsilon)\sqrt{\frac{2h-\frac{3}{2}}{n_k+n_{k-1}-2h}}\right]\\
&\geq &4\sqrt{(n_k+n_{k-1}-2h)\log h}\left[1-\epsilon-\frac{2}{\log n}\right]
\end{eqnarray*}
By (\ref{nkb}) we have
\begin{eqnarray*}
&&4\sqrt{(n_k+n_{k-1}-2h)\log h}\left[1-\epsilon-\frac{2}{\log n}\right]\\
&\geq &4\sqrt{\beta\log n(n_k+n_{k-1}-2h)\left(1-\frac{2\log\log n}{\beta\log n}\right)}\left[1-\epsilon-\frac{2}{\log n}\right]
\end{eqnarray*}
Let 
\begin{eqnarray}
\epsilon=\frac{\log\log n}{\beta\log n};\label{eph}
\end{eqnarray}
then when $n$ is sufficiently large, (\ref{epn}) holds with $N$ replaced by $h$.
Define an event
\begin{eqnarray*}
\tilde{F}^*:&=&\left\{\max_{a\in H_k,b\in H_{k-1}}\langle\mathbf{W},\mathbf{K}(y^{(ab)})-\mathbf{K}(y) \rangle\right.\\
&&\left.\geq4\sqrt{\beta\log n(n_k+n_{k-1}-2h)\left(1-\frac{2\log\log n}{\beta\log n}\right)}\left[1-\frac{\log\log n}{\beta\log n}-\frac{2}{\log n}\right]\right\}
\end{eqnarray*}
Then $F^*\subseteq \tilde{F}^*$

When (\ref{dtb}) and (\ref{sgb}) hold, we have
\begin{eqnarray*}
&&\mathrm{Pr}\left(\mathrm{max}_{a,b\in[n],y(a)\neq y(b)}\sigma\langle\mathbf{W},\mathbf{G}(y^{(ab)})-\mathbf{G}(y) \rangle>2(n_k+n_{k-1}-2)\right)\\
&\geq &\mathrm{Pr}(\tilde{F}^*)\geq \mathrm{Pr}(F)\geq 1-\frac{1}{\log n},
\end{eqnarray*}
as $n$ is sufficiently large. Then the proposition follows.

\end{proof}

\section{Complex Unitary Matrix with Gaussian Perturbation}

Now we consider the community detection problem when the observation is a complex unitary matrix plus a multiple of a GUE or GOE matrix. In the former case, we prove a threshold with respect the intensity $\sigma$ of the GUE perturbation for the exact recovery of the MLE. In the latter case, we develop a ``complex version'' of SDP algorithm for efficient recovery, and explicitly prove the region of the intensity of the GOE perturbation for the exact recovery of the SDP.

\subsection{GUE perturbation}\label{pm3}

In this section, we prove Theorem \ref{m3}. 
Recall that $y\in \Theta_A$ is the true color assignment function satisfying Assumption \ref{ap1}. For $x\in \Theta_A$, define
\begin{eqnarray*}
r(x)=\Re\langle \mathbf{U},\mathbf{P}(x) \rangle.
\end{eqnarray*}

Then
\begin{eqnarray}
r(x)-r(y)=\Re\left[\langle \mathbf{P}(y),\mathbf{P}(x)-\mathbf{P}(y) \rangle+\sigma\langle\mathbf{W}_c,\mathbf{P}(x)-\mathbf{P}(y)\rangle\right]\label{rxy}
\end{eqnarray}
which is a real Gaussian random variable with mean $\Re[\langle \mathbf{P}(y),\mathbf{P}(x)-\mathbf{P}(y) \rangle]$, and variance 
\begin{eqnarray*}
4\sigma^2\sum_{i<j}\{1-\Re[\ol{x(i)}x(j)y(i)\ol{y(j)}]\}.
\end{eqnarray*}
 Moreover,
\begin{eqnarray*}
\Re\langle \mathbf{P}(x),\mathbf{P}(y)\rangle=n+2\sum_{i<j}\Re[\ol{x(i)}x(j)y(i)\ol{y(j)}]
\end{eqnarray*}
Hence
\begin{eqnarray*}
\Re[\langle \mathbf{P}(y),\mathbf{P}(x)-\mathbf{P}(y) \rangle]=-2\sum_{i<j}[1-\Re[\ol{x(i)}x(j)y(i)\ol{y(j)}]]
\end{eqnarray*}
Let 
\begin{eqnarray}
J(x,y)=-\mathbb{E}[r(x)-r(y)]=2\sum_{i<j}[1-\Re[\ol{x(i)}x(j)y(i)\ol{y(j)}]]\label{jxyd}
\end{eqnarray}
Therefore for $x\in \Theta_A$
\begin{eqnarray*}
\mathrm{Pr}(r(x)-r(y)>0)=\mathrm{Pr}_{\xi\in\mathcal{N}(0,1)}\left(\xi>\frac{\sqrt{J(x,y)}}{\sqrt{2}\sigma}\right)\leq e^{-\frac{J(x,y)}{4\sigma^2}}
\end{eqnarray*}

\begin{lemma}\label{l63}For any $x,y\in\Theta_A$, $J(x,y)\geq 0$; $J(x,y)=0$ if and only if there exists a fixed angle $\alpha$, such that $e^{\mathbf{i}\alpha}\mathbf{x}=\mathbf{y}$.
\end{lemma}
\begin{proof}First of all, $J(x,y)\geq 0$ follows from the fact that $\left|\ol{x(i)}x(j)y(i)\ol{y(j)}\right|=1$. Moreover, $J(x,y)=0$ if and only if for any $1\leq i<j\leq n$, $\ol{x(i)}x(j)y(i)\ol{y(j)}=1$. Then the lemma follows.
\end{proof}

 Then

\begin{lemma}\label{l64}If $x,z\in \Theta_A$ such that $x\in \theta(z)$, then 
\begin{enumerate}
\item $\mathbf{P}(x)=\mathbf{P}(z)$.
\item $r(x)=r(z)$.
\end{enumerate}
\end{lemma}
\begin{proof}Note that if $x\in \theta(z)$, then $\mathbf{P}(x)=\mathbf{P}(z)$. Then $r(x)=r(z)$ follows from (\ref{rxy}).
\end{proof}

\begin{lemma}\label{ll7}If $x,y,y'\in \Theta_{A}$ such that $y\in \theta(y')$, then 
\begin{eqnarray*}
J(x,y)=J(x,y').
\end{eqnarray*}
\end{lemma}

\begin{proof}The lemma follows from (\ref{jxyd}) and Part (2) of Lemma \ref{l64}.
\end{proof}

By Lemmas \ref{l63} and \ref{l64}, we obtain
\begin{eqnarray*}
p(y^A;\sigma)=\mathrm{Pr}\left[r(y^A)>\max_{x\in \Theta_A,x\notin \theta(y)}r(x)\right]
\end{eqnarray*}

Hence
\begin{eqnarray}
1-p(y^A;\sigma)\leq \sum_{\theta(x)\subseteq[\Theta_{A}\setminus \theta(y)]} e^{-\frac{J(x,y)}{4\sigma^2}}\label{p1}
\end{eqnarray}

For $i,j\in [k]$, let 
\begin{eqnarray}
S_{i,j}(x,y)=\{1\leq l\leq n:x(l)=e^{\mathbf{i}d_i}, y(l)=e^{\mathbf{i}d_j}\};\label{ssij}
\end{eqnarray}
i.e., $S_{i,j}(x,y)$ consists of all the vertices which have color $e^{\mathbf{i}d_i}$ in $x$ and color $e^{\mathbf{i}d_j}$ in $y$.

 Let $t_{i,j}(x,y)=|S_{i,j}(x,y)|$. Again $t_{i,j}(x,y)$ satisfies (\ref{tt1}) and (\ref{tt2}). We shall now define a distance function on $\Theta$.

\begin{definition}\label{df7}Let $D_{\Theta}: \Theta\times\Theta\rightarrow \mathbb{N}$ be the distance function on $\Theta$ defined as follows: for $x,z\in \Theta$,
\begin{eqnarray*}
D_{\Theta}(x,z)=\sum_{i,j\in [k],i\neq j}t_{i,j}(x,z).
\end{eqnarray*}
\end{definition}

 From (\ref{jxyd}), we obtain
 \begin{eqnarray*}
 J(x,y)&=&\sum_{i,j,p,q\in[k]}t_{i,j}(x,y)t_{p,q}(x,y)[1-\cos(-d_i+d_p+d_j-d_q)]\\
 &=&\sum_{i,j,p,q\in[k]}t_{i,j}(x,y)t_{p,q}(x,y)\left[1-\cos\left(\frac{2\pi(p+j-q-i)}{k}\right)\right]
 \end{eqnarray*}
 When $p,q,i,j\in[k]$, $p+j-q-i\in[-(2k-2),2k-2]$; hence $\cos\left(\frac{2\pi(p+j-q-i)}{k}\right)=1$ if and only if $p+j-q-i\in\{-k,0,k\}$. Therefore,
 \begin{eqnarray}
&& J(x,y)\label{jxy}\\
 &=&\sum_{i,j,p,q\in[k], p+j-q-i\notin\{-k,0,k\}}t_{i,j}(x,y)t_{p,q}(x,y)\left[1-\cos\left(\frac{2\pi(p+j-q-i)}{k}\right)\right]\notag
 \end{eqnarray}

By Assumption \ref{ap1},
\begin{eqnarray*}
n_1(x)=n_2(x)=\ldots =n_k(x)=n_1(y)=n_2(y)=\ldots=n_k(y)=\frac{n}{k}
\end{eqnarray*}

Let $\hat{D}$ be the set defined by
\begin{eqnarray*}
\hat{D}=\left\{(t_{1,1},\ldots,t_{k,k})\in \left\{0,1,\ldots \frac{n}{k}\right\}^{k^2}: \sum_{j\in[k]}t_{i,j}=\frac{n}{k},\sum_{i\in [k]}t_{i,j}=\frac{n}{k}\right\}
\end{eqnarray*}
For  $\epsilon>0$, let
\begin{eqnarray}
\hat{D}_{\epsilon}=\left\{(t_{1,1},\ldots,t_{k,k})\in \hat{D}: \sum_{i,j,p,q\in[k], p+j-q-i\notin\{-k,0,k\}}t_{i,j}t_{p,q}\leq \epsilon n^2\right\}\label{hde}
\end{eqnarray}
 Then we have
\begin{eqnarray}
\sum_{\theta(x)\subseteq[\Theta_{A}\setminus \theta(y)]} e^{-\frac{J(x,y)}{4\sigma^2}}\leq \hat{I}_1+\hat{I}_2\label{p2}
\end{eqnarray}
where
\begin{eqnarray*}
\hat{I}_1=\sum_{\theta(x)\subseteq \Theta_A\setminus \theta(y):\left(t_{1,1},\ldots, t_{k,k}\right)\in D\setminus \hat{D}_{\epsilon}}e^{\frac{-J(x,y)}{4\sigma^2}}.
\end{eqnarray*}
and
\begin{eqnarray}
\hat{I}_2=\sum_{\theta(x)\subseteq \Theta_A\setminus \theta(y):\left(t_{1,1},\ldots,t_{k,k}\right)\in\hat{D}_{\epsilon}}e^{\frac{-J(x,y)}{4\sigma^2}}.\label{hi2}
\end{eqnarray}

\begin{lemma}\label{l72}Assume
\begin{eqnarray*}
\lim_{n\rightarrow\infty}\frac{\epsilon n^2}{k^2}=\infty.
\end{eqnarray*}
Then we have
\begin{eqnarray*}
\lim_{n\rightarrow\infty}\hat{I}_1=0.
\end{eqnarray*}
\end{lemma}
\begin{proof}From the definition of the domains $\hat{D}$ and $\hat{D}_{\epsilon}$, as well as the expression (\ref{jxy}), we obtain that when $(u_{1,1,},\ldots,u_{k,k})\in \hat{\mathcal{D}}\setminus \hat{\mathcal{D}}_{\epsilon}$
\begin{eqnarray*}
J(x,y)\geq \epsilon n^2\left[1-\cos\left(\frac{2\pi}{k}\right)\right].
\end{eqnarray*}
Then the lemma follows.
\end{proof}

From (\ref{jxy}), we obtain
\begin{eqnarray*}
 J(x,y)=\sum_{u=1}^{\lfloor\frac{k}{2} \rfloor}\sum_{\{i,j,p,q\in[k],\ \left((p+j-q-i)\mod k\right)\in\{ u,k-u\}\}}t_{i,j}(x,y)t_{p,q}(x,y)\left[1-\cos\left(\frac{2u\pi}{k}\right)\right]
\end{eqnarray*}

\begin{lemma}\label{l65}For each $x\in \Theta_A$ satisfying $\left(t_{1,1}(x,y),\ldots, t_{k,k}(x,y)\right)\in \hat{D}_{\epsilon}$, there exists $y'\in \theta(y)$ such that 
\begin{eqnarray}
\frac{1}{n}\sum_{i,j\in [k],i\neq j}t_{i,j}(x,y')<(k-1)\sqrt{2\epsilon}\label{xyp}
\end{eqnarray}
In particular, this implies that
\begin{eqnarray}
\frac{1}{n}\sum_{i\in [k]}t_{i,i}(x,y')>1-(k-1)\sqrt{2\epsilon};\label{dlb}
\end{eqnarray}
and for each $i\in[k]$,
\begin{eqnarray}
\frac{1}{n}t_{i,i}(x,y')>\frac{1}{k}-(k-1)\sqrt{2\epsilon}.\label{elb}
\end{eqnarray}
\end{lemma}
\begin{proof}From the definition (\ref{hde}) and the fact that $\sum_{i,j\in[k]}t_{i,j}(x,y)=n$, we obtain
\begin{eqnarray*}
\sum_{i,j,p,q\in[k], p+j-q-i\in\{-k,0,k\}}t_{i,j}(x,y)t_{p,q}(x,y)\geq (1-\epsilon)n^2;
\end{eqnarray*}
which is the same as the following inequality
\begin{eqnarray}
\sum_{l=0}^{k-1}\left[\sum_{[i,j\in[k], (i-j)\mod k=l]}t_{i,j}(x,y)\right]^2\geq (1-\epsilon)n^2\label{goe}
\end{eqnarray}
Let $a=\sqrt{2\epsilon}$. Again using the fact that $\sum_{i,j\in[k]}t_{i,j}(x,y)=n$; if there exists $l_1,l_2\in \{0,\ldots,k-1\}$ and $l_1\neq l_2$ and
\begin{eqnarray*}
\min\left\{\sum_{[i,j\in[k], (i-j)\mod k=l_1]}t_{i,j}(x,y),\sum_{[i,j\in[k], (i-j)\mod k=l_2]}t_{i,j}(x,y)\right\}\geq an;
\end{eqnarray*}
then
\begin{eqnarray*}
&&\sum_{l=0}^{k-1}\left[\sum_{[i,j\in[k], (i-j)\mod k=l]}t_{i,j}(x,y)\right]^2\\
&\leq&\left[\sum_{l=0}^{k-1}\sum_{[i,j\in[k], (i-j)\mod k=l]}t_{i,j}(x,y)\right]^2-\left[\sum_{[i,j\in[k], (i-j)\mod k=l_1]}t_{i,j}(x,y)\right]\cdot \left[\sum_{[i,j\in[k], (i-j)\mod k=l_2]}t_{i,j}(x,y)\right]\\
&\leq & n^2-a^2n^2=(1-2\epsilon)n^2
\end{eqnarray*}
which is a contradiction to (\ref{goe}). Hence there exists at most one $l_0\in\{0,1,\ldots, k-1\}$, such that
\begin{eqnarray*}
\sum_{[i,j\in[k], (i-j)\mod k=l_0]}t_{i,j}(x,y)\geq an.
\end{eqnarray*}
Then for and $l\in\{0,1,\ldots,k-1\}\setminus\{l_0\}$ 
\begin{eqnarray*}
\sum_{[i,j\in[k], (i-j)\mod k= l]}t_{i,j}(x,y)< an.
\end{eqnarray*}
Since $\sum_{i,j\in[k]}t_{i,j}(x,y)=n$, we have
\begin{eqnarray*}
\sum_{[i,j\in[k], (i-j)\mod k=l_0]}t_{i,j}(x,y)&=&\sum_{i,j\in[k]}t_{i,j}(x,y)-\sum_{[l\in(\{0,1,\ldots,k-1\}\setminus\{l_0\})]}\sum_{[i,j\in[k], (i-j)\mod k= l]}t_{i,j}(x,y)\\
&\geq& \left[1-(k-1)a\right]n
\end{eqnarray*}
Let $y'=e^{\frac{2l_0\pi\mathbf{i}}{k}}y$, then (\ref{xyp}) and (\ref{dlb}) follows. To obtain (\ref{elb}), note that
\begin{eqnarray*}
&&\frac{1}{n}t_{i,i}(x,y')=\frac{1}{n}\sum_{j\in[k]}t_{i,j}(x,y')-\frac{1}{n}\sum_{j\in[k],j\neq i}t_{i,j}(x,y')\geq \frac{n}{k}-(k-1)\sqrt{2\epsilon}.
\end{eqnarray*}
Then the lemma follows.
\end{proof}

 \begin{proposition}
Let $y\in\Theta_A$ be the true color assignment function satisfying Assumption \ref{ap1}. Let $k$ be the total number of colors and $n$ be the total number of vertices. Assume
\begin{eqnarray}
\lim_{n\rightarrow\infty}\frac{\log k}{\log n}=0.\label{lknz}
\end{eqnarray} 
If there exists $\delta>0$, such that 
\begin{eqnarray}
\sigma^2<\frac{(1-\delta)\left[n(1-\cos\frac{2\pi}{k})\right]}{2\log n}\label{sc1}
\end{eqnarray}
then 
\begin{eqnarray*}
\lim_{n\rightarrow\infty}\mathrm{Pr}(y^A\in \theta(y))=1
\end{eqnarray*}
\end{proposition}

\begin{proof}For given $x\in \Theta_A$, such that $(t_{1,1}(x,y),\ldots,t_{k,k}(x,y))\in \hat{D}_{\epsilon}$, and $x\notin \theta(y)$, by Lemma \ref{l65}, let $y'\in \theta(y)$ such that (\ref{xyp}) holds. By Lemma \ref{lm33}, there exists an $l$-cycle for $(x,y')$ with $2\leq l\leq k$.

Then for each $2\leq s\leq (l+1)$, choose an arbitrary vertex $u_s$ in $S_{i_{s-1},i_s}(x,y)$, and let $y_1(u_s)=e^{\mathbf{i}d_{i_{s-1}}}$, where $i_{l+1}:=i_1$. For any vertex $z\in[n]\setminus\{u_{2},\ldots,u_{l+1}\}$, let $y_1(z)=y'(z)$.

 Note that $y_1\in \Theta_A$. More precisely, for $1\leq s\leq l$, we have
\begin{eqnarray}
t_{i_s,i_s}(x,y')+1=t_{i_s,i_s}(x,y_1);\label{tpy1}\\
t_{i_s,i_{s+1}}(x,y')-1=t_{i_s,i_{s+1}}(x,y_1)\label{tpy2}
\end{eqnarray}
and
\begin{eqnarray*}
t_{a,b}(x,y')=t_{a,b}(x,y_1),\ \forall (a,b)\notin\{(i_s,i_s),(i_s,i_{s+1})\}_{s=1}^{l}.
\end{eqnarray*}
Let
\begin{eqnarray*}
\Delta:=\{(i_s,i_s),(i_s,i_{s+1})\}_{s=1}^{l}.
\end{eqnarray*}
From (\ref{jxy}) and Lemma \ref{ll7} we obtain
\begin{eqnarray*}
&&J(x,y_1)-J(x,y)=J(x,y_1)-J(x,y')\\
&=&\sum_{i,j,p,q\in[k],p+j-q-i\notin\{-k,0,k\}}\left[t_{i,j}(x,y_1)t_{p,q}(x,y_1)-t_{i,j}(x,y')t_{p,q}(x,y')\right]\left(1-\cos\left(\frac{2\pi(p+j-q-i)}{k}\right)\right)\\
&=& B_1+B_2+B_3+B_4+B_5+B_6+B_7+B_8.
\end{eqnarray*}
Here 
\begin{eqnarray*}
B_1&=&\sum_{s=1}^{l}\sum_{p,q\in[k],(p,q)\notin\Delta}\left[t_{i_s,i_s}(x,y_1)t_{p,q}(x,y_1)-t_{i_s,i_s}(x,y')t_{p,q}(x,y')\right]\left(1-\cos\left(\frac{2\pi(p-q)}{k}\right)\right)\\
B_2&=&\sum_{s=1}^{l}\sum_{p,q\in[k],(p,q)\notin\Delta}\left[t_{i_s,i_{s+1}}(x,y_1)t_{p,q}(x,y_1)-t_{i_s,i_{s+1}}(x,y')t_{p,q}(x,y')\right]\left(1-\cos\left(\frac{2\pi(p+i_{s+1}-q-i_s)}{k}\right)\right)\\
B_3&=&\sum_{i,j\in[k],i,j\notin\Delta}\sum_{s=1}^{l}\left[t_{i,j}(x,y_1)t_{i_s,i_s}(x,y_1)-t_{i,j}(x,y')t_{i_s,i_s}(x,y')\right]\left(1-\cos\left(\frac{2\pi(j-i)}{k}\right)\right)=B_1\\
B_4&=&\sum_{i,j\in[k],i,j\notin\Delta}\sum_{s=1}^{l}\left[t_{i,j}(x,y_1)t_{i_s,i_{s+1}}(x,y_1)-t_{i,j}(x,y')t_{i_s,i_{s+1}}(x,y')\right]\left(1-\cos\left(\frac{2\pi(i_s+j-i_{s+1}-i)}{k}\right)\right)=B_2\\
B_5&=&\sum_{s=1}^{l}\sum_{r=1}^{l}\left[t_{i_s,i_s}(x,y_1)t_{i_r,i_r}(x,y_1)-t_{i_s,i_s}(x,y')t_{i_r,i_r}(x,y')\right]\left(1-\cos0\right)=0\\
B_6&=&\sum_{s=1}^{l}\sum_{r=1}^{l}\left[t_{i_s,i_s}(x,y_1)t_{i_r,i_{r+1}}(x,y_1)-t_{i_s,i_s}(x,y')t_{i_r,i_{r+1}}(x,y')\right]\left(1-\cos\left(\frac{2\pi(i_r-i_{r+1})}{k}\right)\right)\\
B_7&=&\sum_{s=1}^{l}\sum_{r=1}^{l}\left[t_{i_s,i_{s+1}}(x,y_1)t_{i_r,i_r}(x,y_1)-t_{i_s,i_{s+1}}(x,y')t_{i_r,i_r}(x,y')\right]\left(1-\cos\left(\frac{2\pi(i_{s+1}-i_s)}{k}\right)\right)=B_6\\
B_8&=&\sum_{s=1}^{l}\sum_{r=1}^{l}\left[t_{i_s,i_{s+1}}(x,y_1)t_{i_r,i_{r+1}}(x,y_1)-t_{i_s,i_{s+1}}(x,y')t_{i_r,i_{r+1}}(x,y')\right]\left(1-\cos\left(\frac{2\pi(i_r-i_{r+1}+i_{s+1}-i_s)}{k}\right)\right)\\
\end{eqnarray*}
By (\ref{tpy1}) and (\ref{tpy2}) we have
\begin{eqnarray*}
B_1&=&l\sum_{p,q\in[k],(p,q)\notin\Delta}t_{p,q}(x,y')\left(1-\cos\left(\frac{2\pi(p-q)}{k}\right)\right)
\end{eqnarray*}
When $(t_{1,1}(x,y),\ldots,t_{k,k}(x,y))\in \hat{D}_{\epsilon}$, by Lemma \ref{l65} we obtain
\begin{eqnarray*}
0\leq B_1&\leq& 2l(k-1)\sqrt{2\epsilon}n.
\end{eqnarray*}
Similarly,
\begin{eqnarray*}
B_2&=&-\sum_{s=1}^{l}\sum_{p,q\in[k],(p,q)\notin\Delta}t_{p,q}(x,y')\left(1-\cos\left(\frac{2\pi(p+i_{s+1}-q-i_s)}{k}\right)\right)
\end{eqnarray*}
Then
\begin{eqnarray*}
B_2&\leq&-\sum_{s=1}^{l}\sum_{p\in([k]\setminus\{i_1,\ldots,i_l\})}t_{p,p}(x,y')\left(1-\cos\left(\frac{2\pi(i_{s+1}-i_s)}{k}\right)\right)\\
&\leq&-\frac{(k-l)n}{k} \sum_{s=1}^{l}\left(1-\cos\left(\frac{2\pi(i_{s+1}-i_s)}{k}\right)\right)+2l(k-1)\sqrt{2\epsilon}n
\end{eqnarray*}

\begin{eqnarray*}
B_6&=&\sum_{s=1}^{l}\sum_{r=1}^{l}\left[-t_{i_s,i_s}(x,y')+t_{i_r,i_{r+1}}(x,y')-1\right]\left(1-\cos\left(\frac{2\pi(i_r-i_{r+1})}{k}\right)\right)\\
&=&-l\sum_{r=1}^{l}\left(1-\cos\left(\frac{2\pi(i_r-i_{r+1})}{k}\right)\right)+l\sum_{r=1}^{l}t_{i_r,i_{r+1}}(x,y')\left(1-\cos\left(\frac{2\pi(i_r-i_{r+1})}{k}\right)\right)\\
&&-\sum_{s=1}^{l}t_{i_s,i_s}(x,y')\sum_{r=1}^{l}\left(1-\cos\left(\frac{2\pi(i_r-i_{r+1})}{k}\right)\right)
\end{eqnarray*}
By Lemma \ref{l65} we obtain
\begin{eqnarray*}
B_6&\leq& -l\sum_{r=1}^{l}\left(1-\cos\left(\frac{2\pi(i_r-i_{r+1})}{k}\right)\right)+2l(k-1)\sqrt{2\epsilon}n\\
&&-\frac{ln}{k}\sum_{r=1}^{l}\left(1-\cos\left(\frac{2\pi(i_r-i_{r+1})}{k}\right)\right)+2l(k-1)\sqrt{2\epsilon}n
\end{eqnarray*}

\begin{eqnarray*}
B_8&=&\sum_{s=1}^{l}\sum_{r=1}^{l}\left[-t_{i_s,i_{s+1}}(x,y')-t_{i_r,i_{r+1}}(x,y')-1\right]\left(1-\cos\left(\frac{2\pi(i_r-i_{r+1}+i_{s+1}-i_s)}{k}\right)\right)
\end{eqnarray*}

By Lemma \ref{l65} we obtain
\begin{eqnarray*}
0\geq B_8\geq -4l(k-1)n\sqrt{2\epsilon}-2l^2
\end{eqnarray*}

Then we have 
\begin{eqnarray*}
J(x,y_1)-J(x,y)\leq \left[16l(k-1)\sqrt{2\epsilon}-2l\left(1-\cos\frac{2\pi}{k}\right)\right]n.
\end{eqnarray*}

Therefore
\begin{eqnarray}
e^{-\frac{J(x,y)}{4\sigma^2}}\leq e^{-\frac{J(x,y_1)}{4\sigma^2}}e^{-\frac{n}{4\sigma^2}\left[2l\left(1-\cos\frac{2\pi}{k}\right)-16l(k-1)\sqrt{2\epsilon}\right]}.\label{jst}
\end{eqnarray}

If $y_1\neq x$, we find an $l_2$-cycle ($2\leq l_2\leq k$) for $(x,y_1)$, change colors along the $l_2$-cycle  as above, and obtain another cycle assignment mapping $y_2\in \Theta_{A}$, and so on. 
Let $y_0:=y$, and note that for each $r\geq 1$, if $y_r$ is obtained from $y_{r-1}$ by changing colors along an $l_r$ cycle, we have
\begin{eqnarray*}
D_{\Theta}(x,y_r)= D_{\Theta}(x,y_{r-1})-l_r
\end{eqnarray*}
Therefore finally we can obtain $x$ from $y$ by changing colors along at most $\left\lfloor \frac{n}{2} \right\rfloor$ cycles. Using similar arguments as those used to derive (\ref{jst}), we obtain that for each $r$
\begin{eqnarray*}
e^{-\frac{J(x,y_{r-1})}{4\sigma^2}}\leq e^{-\frac{J(x,y_r)}{4\sigma^2}}e^{-\frac{n}{4\sigma^2}\left[2l_r\left(1-\cos\frac{2\pi}{k}\right)-16l_r(k-1)\sqrt{2\epsilon}\right]}.
\end{eqnarray*}
Therefore if $y_s=x$ for some $1\leq s\leq \left\lfloor \frac{n}{2} \right\rfloor$, we have
\begin{eqnarray*}
e^{-\frac{J(x,y)}{4\sigma^2}}\leq e^{-\frac{J(x,y_s)}{4\sigma^2}}e^{-\frac{n}{4\sigma^2}\left[2\left(1-\cos\frac{2\pi}{k}\right)-16(k-1)\sqrt{2\epsilon}\right]\left[\sum_{i=1}^{s}l_i\right]}.
\end{eqnarray*}
By (\ref{jxyd}), we have $J(x,y_s)=J(x,x)=0$, hence
\begin{eqnarray*}
e^{-\frac{J(x,y)}{4\sigma^2}}\leq \prod_{i\in[s]}e^{-\frac{nl_i}{4\sigma^2}\left[2\left(1-\cos\frac{2\pi}{k}\right)-16(k-1)\sqrt{2\epsilon}\right]}.
\end{eqnarray*}

Let
\begin{eqnarray*}
\gamma_{\epsilon}:=2\left(1-\cos\frac{2\pi}{k}\right)-16(k-1)\sqrt{2\epsilon}.
\end{eqnarray*}


Note also that for any $r_1\neq r_2$, in the process of obtaining $y_{r_1}$ from $y_{r_1-1}$ and the process of obtaining $y_{r_2}$ from $y_{r_2-1}$, we change colors on disjoint sets of vertices. Hence the order of these steps of changing colors along cycles does not affect the final color assignment mapping we obtain.  From (\ref{hi2}) we obtain
\begin{eqnarray}
\hat{I}_2\leq \prod_{l=2}^k \left(\sum_{m_l=0}^{\infty} (nk)^{m_l\ell} e^{-\frac{nm_ll \gamma_{\epsilon}}{4\sigma^2}}\right)-1\label{hiup}
\end{eqnarray}
On the right hand side of (\ref{hiup}), when expanding the product, each summand has the form
\begin{eqnarray*}
\left[(nk)^{2m_2}e^{-\frac{2m_2 n \gamma_{\epsilon}}{4\sigma^2}}\right]\cdot\left[ (nk)^{3m_3}e^{-\frac{3m_3 n \gamma_{\epsilon}}{4\sigma^2}}\right]\cdot\ldots\cdot\left[(nk)^{km_k} e^{-\frac{km_k n \gamma_{\epsilon}}{4\sigma^2}}\right]
\end{eqnarray*}
where the factor $\left[(nk)^{2m_2}e^{-\frac{2m_2 n \gamma_{\epsilon}}{4\sigma^2}}\right]$ represents that we changed along 2-cycles $m_2$ times, the factor $\left[ (nk)^{3m_3}e^{-\frac{3m_3 n \gamma_{\epsilon}}{4\sigma^2}}\right]$ represents that we changed along 3-cycles $m_3$ times, and so on. Moreover, each time we changed along an $l$-cycle, we need to first determine the $l$ different colors involved in the $l$-cycle, and there are at most $k^l$ different $l$-cycles;  we then need to chose $l$ vertices to change colors, and there are at most $n^{l}$ choices.
It is straightforward to check that if $\sigma$ satisfies (\ref{sc1}), then
\begin{eqnarray*}
nk e^{-\frac{n \gamma_{\epsilon}}{4\sigma^2}}\leq e^{\log k-\frac{\delta \log n}{1-\delta}+\frac{8(k-1)\sqrt{2\epsilon}\log n}{(1-\delta)\left(1-\cos\frac{2\pi}{k}\right)}}.
\end{eqnarray*}
as $n\rightarrow\infty$. Therefore we have
\begin{eqnarray*}
\sum_{m_l=0}^{\infty} (nk)^{m_l\ell} e^{-\frac{nm_ll \gamma_{\epsilon}}{4\sigma^2}}\leq \frac{1}{1-e^{l\left(\log k-\frac{\delta \log n}{1-\delta}+\frac{8(k-1)\sqrt{2\epsilon}\log n}{(1-\delta)\left(1-\cos\frac{2\pi}{k}\right)}\right)}}
\end{eqnarray*}
Let 
\begin{eqnarray*}
\hat{U}:= \prod_{l=2}^k \left(\sum_{m_l=0}^{\infty} (nk)^{m_l\ell} e^{-\frac{nm_ll \gamma_{\epsilon}}{4\sigma^2}}\right).
\end{eqnarray*}
When (\ref{lknz}) holds, let
\begin{eqnarray*}
\epsilon:=\frac{1}{\sqrt{n}}
\end{eqnarray*}
Then when $\sigma$ satisfies (\ref{sc1})
\begin{eqnarray*}
\frac{\gamma_{\epsilon}}{4\sigma^2}\leq \frac{\delta \log n}{2(1-\delta)}
\end{eqnarray*}
when $n$ is sufficiently large.

Since $\log(1+x)\leq x$ for $x\geq 0$, we have
\begin{eqnarray*}
0\leq \log \hat{U}&=&\sum_{l=2}^{k}\log \left(1+\sum_{m_l=1}^{\infty} (nk)^{m_l\ell} e^{-\frac{nm_ll \gamma_{\epsilon}}{4\sigma^2}}\right)\\
&\leq &\sum_{l=2}^{k}\sum_{m_l=1}^{\infty} (nk)^{m_l\ell} e^{-\frac{nm_ll \gamma_{\epsilon}}{4\sigma^2}}\\
&\leq &\sum_{l=2}^{k}\frac{\left(e^{-\frac{\delta \log n}{2(1-\delta)}}\right)^{l}}{1-\left(e^{-\frac{\delta \log n}{2(1-\delta)}}\right)^{l}}\\
&\leq &\frac{e^{-\frac{\delta \log n}{1-\delta}}}{\left(1-e^{-\frac{\delta \log n}{1-\delta}}\right)\left(1-e^{-\frac{\delta \log n}{2(1-\delta)}}\right)}\rightarrow 0
\end{eqnarray*}
as $n\rightarrow\infty$. Hence we have $\lim_{n\rightarrow\infty}\hat{I}_2=0.$ By Lemma \ref{l72}, $\lim_{n\rightarrow\infty}\hat{I}_1=0.$ Then the proposition follows.

\end{proof}

\bigskip

\noindent\textbf{Proof of Theorem \ref{m3}(b).} For $y\in \Theta_A$, $a,b\in[n]$ such that $1=y(a)\neq y(b)=e^{\frac{2\pi \mathbf{i}}{k}}$. Let $y^{(ab)}$ be the coloring of vertices defined by
\begin{eqnarray}
y^{(ab)}(i)=\begin{cases}y(i)& \mathrm{if}\ i\in\{1,2,\ldots,n\}\setminus\{a,b\}\\ e^{\frac{2\pi\mathbf{i}}{k}}&\mathrm{if}\ i=a\\ 1& \mathrm{if}\ i=b \end{cases}\label{yab}
\end{eqnarray}
Then
\begin{eqnarray*}
1-p(y^A;\sigma)\geq \mathrm{Pr}\left(\cup_{a,b\in[n], 1=y(a)\neq y(b)=e^{\frac{2\pi\mathbf{i}}{k}}}[r(y^{(ab)})-r(y)>0]\right),
\end{eqnarray*}
since any of the event $[r(y^{(ab)})-r(y)>0]$ implies $y^A\neq y$. Recall that
\begin{eqnarray*}
r(y^{(ab)})-r(y)&=&\langle \mathbf{P}(y),\mathbf{P}(y^{(ab)})-\mathbf{P}(y)\rangle+\sigma\langle\mathbf{W}_c,\mathbf{P}(y^{(ab)})-\mathbf{P}(y) \rangle\\
&=&-4(n-2)\left(1-\cos\frac{2\pi}{k}\right)-2\left(1-\cos \frac{4\pi}{k}\right)+\sigma\langle\mathbf{W}_c,\mathbf{P}(y^{(ab)})-\mathbf{P}(y) \rangle.
\end{eqnarray*}
Let
\begin{eqnarray*}
E(n):=4(n-2)\left(1-\cos\frac{2\pi}{k}\right)+2\left(1-\cos \frac{4\pi}{k}\right)
\end{eqnarray*}

For $l\in\{0,1\}$, let $H_l\subset y^{-1}\left(e^{\frac{2l\pi\mathbf{i}}{k}}\right)$ such that $|H_l|=\frac{n}{\log^2n}=h$. Then
\begin{eqnarray*}
1-p(y^A;\sigma)\geq \mathrm{Pr}\left(\mathrm{max}_{a\in H_0,b\in H_1}\sigma\langle\mathbf{W},\mathbf{P}(y^{(ab)})-\mathbf{P}(y) \rangle>E(n)\right)
\end{eqnarray*}
Let $(\mathcal{X},\mathcal{Y},\mathcal{Z})$ be a partition of $[n]^2$ defined by
\begin{eqnarray*}
&&\mathcal{X}=\{\alpha=(\alpha_1,\alpha_2)\in [n]^2, \{\alpha_1,\alpha_2\}\cap [H_0\cup H_1]=\emptyset\}\\
&&\mathcal{Y}=\{\alpha=(\alpha_1,\alpha_2)\in [n]^2, |\{\alpha_1,\alpha_2\}\cap [H_0\cup H_1]|=1\}\\
&&\mathcal{Z}=\{\alpha=(\alpha_1,\alpha_2)\in [n]^2, |\{\alpha_1,\alpha_2\}\cap [H_0\cup H_1]|=2\}
\end{eqnarray*}
For $\eta\in\{\mathcal{X},\mathcal{Y},\mathcal{Z}\}$, define the $n\times n$ matrix $\mathbf{W}_{\eta}$ from the entries of $\mathbf{W}$ as follows
\begin{eqnarray*}
\mathbf{W}_{\eta,c}(i,j)=\begin{cases}0&\mathrm{if}\ (i,j)\notin \eta\\ \mathbf{W}_c(i,j),&\mathrm{if}\ (i,j)\in \eta\end{cases}
\end{eqnarray*}
For each $a\in H_0$ and $b\in H_1$, let
\begin{eqnarray*}
\mathcal{X}_{ab}=\langle\mathbf{W}_{\mathcal{X},c},\mathbf{P}(y^{(ab)})-\mathbf{P}(y) \rangle\\
\mathcal{Y}_{ab}=\langle\mathbf{W}_{\mathcal{Y},c},\mathbf{P}(y^{(ab)})-\mathbf{P}(y) \rangle\\
\mathcal{Z}_{ab}=\langle\mathbf{W}_{\mathcal{Z},c},\mathbf{P}(y^{(ab)})-\mathbf{P}(y) \rangle
\end{eqnarray*}
\begin{claim}The followings are true:
\begin{enumerate}
\item $\mathcal{X}_{ab}=0$ for $a\in H_0$ and $b\in H_1$.
\item For each $a\in H_0$ and $b\in H_1$, the variables $\mathcal{Y}_{ab}$ and $\mathcal{Z}_{ab}$ are independent.
\item Each $\mathcal{Y}_{ab}$ can be decomposed into $Y_a+Y_b$ where $\{Y_a\}_{a\in H_0}\cup \{Y_b\}_{b\in H_1}$ is a collection of i.i.d. Gaussian random variables.
\end{enumerate}
\end{claim}

\begin{proof}It is straightforward to check (1). (2) holds because $\mathcal{Y}\cap\mathcal{Z}=\emptyset$.

For $s\in H_0\cup H_1$, let $\mathcal{Y}_s\subseteq \mathcal{Y}$ be defined by
\begin{eqnarray*}
\mathcal{Y}_s=\{\alpha=(\alpha_1,\alpha_2)\in \mathcal{Y}:\alpha_1=s,\ \mathrm{or}\ \alpha_2=s\}.
\end{eqnarray*}
Note that for $s_1,s_2\in H_0\cup H_1$ and $s_1\neq s_2$, $\mathcal{Y}_{s_1}\cap \mathcal{Y}_{s_2}=\emptyset$. Moreover, $\mathcal{Y}=\cup_{s\in H_0\cup H_1}\mathcal{Y}_s$. Therefore
\begin{eqnarray*}
\mathcal{Y}_{ab}=\sum_{s\in H_0\cup H_1}\langle\mathbf{W}_{\mathcal{Y}_s,c},\mathbf{P}(y^{(ab)})-\mathbf{P}(y) \rangle
\end{eqnarray*}
Note also that $\langle\mathbf{W}_{\mathcal{Y}_s,c},\mathbf{P}(y^{(ab)})-\mathbf{P}(y) \rangle=0$, if $s\notin \{a,b\}$. Hence
\begin{eqnarray*}
\mathcal{Y}_{ab}=\sum_{\alpha\in\mathcal{Y}_a\cup \mathcal{Y}_b}[\mathbf{W}_c(\alpha)]\cdot\{[\mathbf{P}(y^{(ab)})-\mathbf{P}(y)](\alpha)\}
\end{eqnarray*}
Note that for $\alpha\in \mathcal{Y}_a$,
\begin{eqnarray*}
[\mathbf{P}(y^{(ab)})-\mathbf{P}(y)](\alpha)=\begin{cases}(e^{\frac{2\pi\mathbf{i}}{k}}-1)\ol{y(\alpha_2)} &\mathrm{if}\ \alpha_1=a\\ y(\alpha_1)(e^{-\frac{2\pi\mathbf{i}}{k}}-1) &\mathrm{if}\  \alpha_2=a.\end{cases}
\end{eqnarray*}
So, 
\begin{eqnarray*}
Y_a:=&&\sum_{\alpha\in \mathcal{Y}_a}[\mathbf{W}_c(\alpha)]\cdot\{[\mathbf{P}(y^{(ab)})-\mathbf{P}(y)](\alpha)\}\\
&=&\left\{\sum_{\alpha\in \mathcal{Y}_a;\alpha_1=a}[\ol{y(\alpha_2)}\mathbf{W}_c(\alpha)]+\sum_{\alpha\in \mathcal{Y}_a;\alpha_2=a}[y(\alpha_1)\mathbf{W}_c(\alpha)]\right\}\left(\cos \frac{2\pi}{k}-1\right)\\
&&+\mathbf{i}\left\{\sum_{\alpha\in \mathcal{Y}_a;\alpha_1=a}[\ol{y(\alpha_2)}\mathbf{W}_c(\alpha)]-\sum_{\alpha\in \mathcal{Y}_a;\alpha_2=a}[y(\alpha_1)\mathbf{W}_c(\alpha)]\right\}\sin \frac{2\pi}{k}
\end{eqnarray*}
Similarly, define
\begin{eqnarray*}
Y_b:&=&\left\{\sum_{\alpha\in \mathcal{Y}_b;\alpha_2=b}[\ol{y(\alpha_2)}\mathbf{W}_c(\alpha)]
-\sum_{\alpha\in \mathcal{Y}_b;\alpha_1=b}[y(\alpha_1)\mathbf{W}_c(\alpha)]\right\}\left(1-\cos\frac{2\pi}{k}\right)\\
&&-\mathbf{i}\left\{\sum_{\alpha\in \mathcal{Y}_a;\alpha_1=b}[\ol{y(\alpha_2)}\mathbf{W}_c(\alpha)]-\sum_{\alpha\in \mathcal{Y}_a;\alpha_2=b}[y(\alpha_1)\mathbf{W}_c(\alpha)]\right\}\sin \frac{2\pi}{k}
\end{eqnarray*}
Then $\mathcal{Y}_{ab}=Y_a+Y_b$ and $\{Y_s\}_{s\in H_0\cup H_1}$ is a collection of independent Gaussian random variables. Moreover, the variance of $Y_s$ is equal to $8(n-2h)\left(1-\cos\frac{2\pi}{k}\right)$ independent of the choice of $s$.
\end{proof}
By the claim, we obtain
\begin{eqnarray*}
\langle\mathbf{W}_c,\mathbf{P}(y^{(ab)})-\mathbf{P}(y) \rangle=Y_a+Y_b+\mathcal{Z}_{ab}
\end{eqnarray*}
Moreover,
\begin{eqnarray*}
\max_{a\in H_0,b\in H_1}Y_a+Y_b+\mathcal{Z}_{ab}&\geq& \max_{a\in H_0,b\in H_1}(Y_a+Y_b)-\max_{a\in H_0,b\in H_1}(-\mathcal{Z}_{ab})\\
&=&\max_{a\in H_0} Y_a+\max_{b\in H_1}Y_b-\max_{a\in H_0,b\in H_1}(-\mathcal{Z}_{ab})
\end{eqnarray*}

By Lemma \ref{mg} we obtain
\begin{eqnarray*}
\max_{a\in H_0}Y_a\geq (1-0.01\epsilon)\sqrt{2\log h\cdot 8(n-2h)\left(1-\cos\frac{2\pi}{k}\right)}\\
\max_{b\in H_1}Y_b\geq (1-0.01\epsilon)\sqrt{2\log h\cdot 8(n-2h)\left(1-\cos\frac{2\pi}{k}\right)}\\
\max_{a\in H_0,b\in H_1}\mathcal{Z}_{ab}\leq (1+\epsilon)\sqrt{4\log h\cdot \max \mathrm{Var}(Z_{ab})}
\end{eqnarray*}
with probability $1-o_n(1)$.(Here $o_n(1)\rightarrow 0$ as $n\rightarrow\infty$.) 
 Moreover,
\begin{eqnarray*}
\mathrm{Var} \mathcal{Z}_{ab}&\leq & 32 h^2
\end{eqnarray*}
which is $o(n)$. Hence
\begin{eqnarray*}
\max_{a\in H_1,b\in H_2}\langle\mathbf{W}_c,\mathbf{P}(y^{(ab)})-\mathbf{P}(y) \rangle&\geq& 2(1-0.01\epsilon-o(1))\sqrt{2\log n\cdot 8(n-2h)\left(1-\cos \frac{2\pi}{k}\right)}\\
&\geq &8(1-0.01\epsilon-o(1))\sqrt{\left(1-\cos\frac{2\pi}{k}\right)n\log n}
\end{eqnarray*}
with probability $1-o_n(1)$. Since $\sigma^2>\frac{(1+\delta)\left[n(1-\cos\frac{2\pi}{k})\right]}{2\log n}$, we have
\begin{eqnarray*}
\mathrm{Pr}\left(\mathrm{max}_{a\in H_0,b\in H_1}\sigma\langle\mathbf{W},\mathbf{P}(y^{(ab)})-\mathbf{P}(y) \rangle>E(n)\right)\geq 1-o_n(1)
\end{eqnarray*}
Then the lemma follows.
$\hfill\Box$

\subsection{Algorithm: Complex Semi-Definite Programming and GOE perturbation}\label{pm5}

In this section, we prove Theorem \ref{m5}.

Assume
\begin{eqnarray*}
\mathbf{V}=V_1+\mathbf{i}V_2; \qquad \mathbf{X}=X_1+\mathbf{i}X_2
\end{eqnarray*}
where $U_1$ and $X_1$ are $n\times n$ real symmetric matrices, and $U_2$ and $X_2$ are $n\times n$ real anti-symmetric matrices.

Let
\begin{eqnarray*}
\tilde{X}=\left(\begin{array}{cc}X_1&-X_2\\X_2&X_1\end{array}\right)\qquad \tilde{V}=\left(\begin{array}{cc}V_1&V_2\\-V_2&V_1\end{array}\right),
\end{eqnarray*}
then the complex optimization problem (\ref{co}) is equivalent to the following real optimization problem
\begin{eqnarray}
&&\max \langle \tilde{V},\tilde{X} \rangle\label{ro}\\
\mathrm{subject\ to}\ &&\tilde{X}_{ii}=1,\ \mathrm{for}\ 1\leq i\leq 2n\notag\\
&&\tilde{X}_{p,q}=\tilde{X}_{p+n,q+n},\ \mathrm{for}\ 1\leq p\leq n,\ 1\leq q\leq n\notag\\
&&\tilde{X}_{p,q+n}=-\tilde{X}_{p+n,q},\ \mathrm{for}\ 1\leq p\leq n,\ 1\leq q\leq n. \notag\\
\mathrm{and}\ &&\tilde{X}\succeq 0\notag
\end{eqnarray}
The dual program of (\ref{ro}) is 
\begin{eqnarray}
&&\min \mathrm{tr}(Z)\label{dro}\\
\mathrm{subject\ to}\ &&Z-\tilde{V}+\left(\begin{array}{cc}A&0\\0 &-A\end{array}\right)+\left(\begin{array}{cc}0&B\\B &0\end{array}\right) \succeq 0\notag\\
&& Z\ \mathrm{is\ diagonal}\notag
\end{eqnarray}
By complementary slackness
\begin{eqnarray*}
X=X_1+iX_2=\mathbf{y}\ol{\mathbf{y}}^t=(\mathbf{y}_1+i\mathbf{y}_2)(\mathbf{y}_1-i\mathbf{y}_2)^t
\end{eqnarray*}
is the unique optimum solution of (\ref{co}) if and only if there exists a dual feasible solution $(Z,A,B)$, such that
\begin{eqnarray*}
&&\left\langle Z-\tilde{V}+\left(\begin{array}{cc}A&0\\0 &-A\end{array}\right)+\left(\begin{array}{cc}0&B\\B &0\end{array}\right),
\left(\begin{array}{cc}\mathbf{y}_1\mathbf{y}_1^t+\mathbf{y}_2\mathbf{y}_2^t&\mathbf{y}_1\mathbf{y}_2^t-\mathbf{y}_2\mathbf{y}_1^t\\ \mathbf{y}_2\mathbf{y}_1^t-\mathbf{y}_1\mathbf{y}_2^t & \mathbf{y}_1\mathbf{y}_1^t+\mathbf{y}_2\mathbf{y}_2^t\end{array}\right)\right \rangle=0;
\end{eqnarray*}
which is equivalent to
\begin{eqnarray}
&&\left\langle Z-\tilde{V},\left(\begin{array}{cc}\mathbf{y}_1\mathbf{y}_1^t+\mathbf{y}_2\mathbf{y}_2^t&\mathbf{y}_1\mathbf{y}_2^t-\mathbf{y}_2\mathbf{y}_1^t\\ \mathbf{y}_2\mathbf{y}_1^t-\mathbf{y}_1\mathbf{y}_2^t & \mathbf{y}_1\mathbf{y}_1^t+\mathbf{y}_2\mathbf{y}_2^t\end{array}\right) \right\rangle=0.\label{rf}
\end{eqnarray}

Assume
\begin{eqnarray*}
Z=\left(\begin{array}{cc}Z_1&0\\0&Z_2\end{array}\right),
\end{eqnarray*}
where $Z_1$, $Z_2$ are $n\times n$ real diagonal matrices. Then (\ref{rf}) is equivalent to
\begin{eqnarray}
\Re\left[\mathbf{y}^t(Z_1+Z_2-2\mathbf{V})\ol{\mathbf{y}}\right]=0\label{eqs}
\end{eqnarray}
Note that
\begin{eqnarray}
&&\Im\left[\mathbf{y}^t(Z_1+Z_2-2\mathbf{V})\ol{\mathbf{y}}\right]\label{imp}=\langle Z_1+Z_2-2V_1,\mathbf{y}_2\mathbf{y}_1^t-\mathbf{y}_1\mathbf{y}_2^t\rangle-\langle 2V_2,\mathbf{y}_1\mathbf{y}_1^t+\mathbf{y}_2\mathbf{y}_2^t\rangle\notag
\end{eqnarray}
which is identically zero since each term on the right hand side of (\ref{imp}) is the inner product of a symmetric matrix and an anti-symmetric matrix.
Moreover, if the minimizer of (\ref{dro}) is unique, then as a Hermitian matrix, the second smallest eigenvalue of 
\begin{eqnarray*}
S:=Z_1+Z_2-2\mathbf{V}
\end{eqnarray*}
is strictly positive.

From (\ref{eqs}) and the fact that $S$ is positive semi-definite, we obtain $S\mathbf{y}=0$. Hence
\begin{eqnarray*}
Z_1(i,i)+Z_2(i,i)=2\sum_{j=1}^{n}\ol{\mathbf{y}(i)}[\mathbf{V}(i,j)]\mathbf{y}(j);
\end{eqnarray*}
and therefore,
\begin{eqnarray}
S(i,i)=2\sum_{j\in[n],j\neq i}\ol{\mathbf{y}(i)}\mathbf{V}(i,j)\mathbf{y}(j)-2\mathbf{V}(i,i)\label{sii}
\end{eqnarray}
For $i\neq j$,
\begin{eqnarray}
S(i,j)=-2\mathbf{V}(i,j)\label{sij}
\end{eqnarray}
For a Hermitian matrix $M$, we define the Laplacian $\Delta(M)$ of $M$ by
\begin{eqnarray*}
\Delta(M):=\mathrm{diag}(M\mathbf{1})-M
\end{eqnarray*}
where $\mathbf{1}$ is the column vector all of whose entries are 1. Then from (\ref{sii}), (\ref{sij}), explicit computations show that
\begin{eqnarray*}
S=2\mathrm{diag}(\mathbf{y})[\Delta(\mathrm{diag}(\ol{\mathbf{y}})\mathbf{V}\mathrm{diag}(\mathbf{y}))]\mathrm{diag}(\ol{\mathbf{y}})
\end{eqnarray*}
Moreover,
\begin{eqnarray*}
\mathrm{diag}(\ol{\mathbf{y}})\mathbf{V}\mathrm{diag}(\mathbf{y})&=&\mathrm{diag}(\ol{\mathbf{y}})[\mathbf{y}\ol{\mathbf{y}}^t+\sigma\mathrm{diag}(\mathbf{y})\mathbf{W}_s\mathrm{diag}(\ol{\mathbf{y}})]\mathrm{diag}(\mathbf{y})\\
&=&\mathbf{1}\mathbf{1}^t+\sigma\mathbf{W}_s
\end{eqnarray*}
Therefore
\begin{eqnarray}
\Delta[\mathrm{diag}(\ol{\mathbf{y}})\mathbf{U}\mathrm{diag}(\mathbf{y})]=n\left(I_{n\times n}-\frac{1}{n}\mathbf{1}\mathbf{1}^t\right)+\sigma\Delta[\mathbf{W}_s].\label{dd}
\end{eqnarray}
The matrix $n\left(I_{n\times n}-\frac{1}{n}\mathbf{1}\mathbf{1}^t\right)$ has rank $(n-1)$ and two distinct eigenvalues: $0$ and $n$. The eigenvalue $0$ has multiplicity $1$ and $n$ has multiplicity $(n-1)$. Note that $0$ is also an eigenvalue of the matrix $\Delta[\mathbf{W}_s]$, and the $n$-dimensional vector $\mathbf{1}$ is an eigenvector with respect to the eigenvalue 0 for both the matrix $n\left(I_{n\times n}-\frac{1}{n}\mathbf{1}\mathbf{1}^t\right)$ and the matrix $\Delta[\mathbf{W}_s]$. Therefore the matrix (\ref{dd}) is positive definite if
\begin{eqnarray}
\sigma\|\Delta[\mathbf{W}_s]\|\leq n,\label{c1}
\end{eqnarray}
where $\|\cdot\|$ is the spectral norm of a matrix defined to be the largest modulus of its eigenvalues. We have, by the triangle inequality,
\begin{eqnarray}
\|\Delta\mathbf{W}_s\|\leq \max_{i\in [n]}\left|\sum_{j=1}^{n}\mathbf{W}_s(i,j)\right|+\|\mathbf{W}_s\|.\label{c2}
\end{eqnarray}
The matrix $\mathbf{W}_s$ is a standard GOE matrix. Recall the following proposition about the largest eigenvalue of the standard GOE matrix.
\begin{proposition}\label{tw}(Tracy-Widom \cite{TW}) Let $\lambda_{\max}^{(n)}$ be the largest eigenvalue of an $n\times n$ GOE matrix, then when $n$ is large,
\begin{eqnarray*}
\lambda_{\max}^{(n)}\sim \sqrt{2n}+\frac{n^{-\frac{1}{6}}\xi_1}{\sqrt{2}}
\end{eqnarray*}
where $\xi_1$ is a random variable independent of $n$ with the GOE Tracy-Widom distribution.
\end{proposition}

By Proposition \ref{tw}, for any fixed $\delta>0$,
\begin{eqnarray}
\lim_{n\rightarrow\infty}\mathbb{P}(\|\mathbf{W}_s\|\geq (1+\delta)\sqrt{2n})=0.\label{c3}
\end{eqnarray}

Now we consider the distribution of $\max_{i\in [n]}\left|\sum_{j=1}^{n}\mathbf{W}_s(i,j)\right|$. 
Since we require that $\mathbf{W}_s$ is a symmetric matrix, the identically distributed Gaussion random variables $\{\sum_{j=1}^{n}\mathbf{W}_s(i,j)\}_{i\in [n]}$ are no longer independent. 
In our case the random vector $(\sum_{j=1}^{n}\mathbf{W}_s(1,j),\ldots,\sum_{j=1}^{n}\mathbf{W}_s(n,j) )$ is a Gaussian random vector with mean 0 and covariance matrix given by
\begin{eqnarray*}
\Sigma=\left(\begin{array}{cccc}n&1&\ldots&1\\ 1&n&\ldots&1\\ &\ldots&&\\ 1&\ldots&1&n\end{array}\right)
\end{eqnarray*}

The following proposition was proved in \cite{HHR}.
\begin{proposition}\label{p6}(Theorem 2.2 of \cite{HHR})Consider a triangular array of normal random variables $\xi_{n,i}$, $i=1,2,\ldots,$, and $n=1,2,\ldots,$ such that for each $n$, $\{\xi_{n,i},i\geq 0\}$ is a stationary normal sequence. Assume $\xi_{n,i}\sim \mathcal{N}(0,1)$. Let $\rho_{n,j}:=\mathbb{E}(\xi_{n,i},\xi_{n,i+j})$ and assuming that
\begin{eqnarray*}
(1-\rho_{n,j})\log n\rightarrow \delta_j\in(0,\infty],\ \mathrm{for\ all}\ j\geq 1\ \mathrm{as}\ n\rightarrow\infty
\end{eqnarray*}
Assume that there exist positive integers $l_n$ satisfying $l_n=o(n)$ and for which
\begin{eqnarray*}
\lim_{n\rightarrow\infty}\sup_{j\geq l_n}|\rho_{n,j}|\log n=0.
\end{eqnarray*}
and
\begin{eqnarray*}
\lim_{m\rightarrow\infty}\limsup_{n\rightarrow\infty}\sum_{j=m}^{l_n}n^{-\frac{1-\rho_{n,j}}{1+\rho_{n,j}}}\frac{(\log n)^{-\frac{\rho_{n,j}}{1+\rho_{n,j}}}}{(1-\rho_{n,j}^2)^{\frac{1}{2}}}=0
\end{eqnarray*}
Then
\begin{eqnarray*}
\lim_{n\rightarrow\infty}\mathrm{Pr}(\max_{1\leq i\leq n}\xi_{n,i}\leq u_n(x))=\exp[-\theta\exp(-x)]
\end{eqnarray*}
where
\begin{eqnarray*}
u_n(x)=\frac{x}{a_n}+b_n;
\end{eqnarray*}
and
\begin{eqnarray*}
a_n&=&\sqrt{2\log n}\\
b_n&=&\sqrt{2\log n}+\frac{\log\log n+\log 4\pi}{2\sqrt{2\log n}}
\end{eqnarray*}
and $\theta\in[0,1]$ is a constant. In particular $\theta=1$ if $\delta_j=\infty$ for all $j\geq 1$.
\end{proposition}

By Proposition \ref{p6},
\begin{eqnarray*}
&& \max_{i\in [n]}\left|\sum_{j=1}^{n}\mathbf{W}_s(i,j)\right|\\
 &\geq &\left(1-\delta\right)\sqrt{2\log n\left(\max_{i\in[n]} \left[\mathrm{Var}\left(\sum_{j\in[n]}\mathbf{W}_s(i,j)\right)\right]\right)}\\
& =&\left(1-\delta\right)\sqrt{2n\log n}
\end{eqnarray*}
with probability $1-o_n(1)$. Moreover, by Lemma \ref{mg}, for any fixed $\epsilon>0$
\begin{eqnarray}
&& \max_{i\in [n]}\left|\sum_{j=1}^{n}\mathbf{W}_s(i,j)\right|\leq\left(1+\epsilon\right)\sqrt{2n\log n}\label{c4}
\end{eqnarray}
with probability $1-o(1)$.

Then Theorem \ref{m5} follows from (\ref{c1}), (\ref{c2}), (\ref{c3}) and (\ref{c4}).

\appendix

\section{Proof of Lemma \ref{mg}}

It is known that for a standard Gaussian random variable $G_i$ and $x>0$,
\begin{eqnarray}
\frac{xe^{-\frac{x^2}{2}}}{\sqrt{2\pi}(1+x^2)}\leq \mathbf{Pr}(G_i>x)\leq \frac{e^{-\frac{x^2}{2}}}{x\sqrt{2\pi}}\label{gd}
\end{eqnarray}

Let $G_1,\ldots, G_N$ be $N$ standard Gaussian random variables. Then by (\ref{gd}) we have
\begin{eqnarray*}
\mathrm{Pr}\left(\max_{i\in[N]}G_i\geq (1+\epsilon)\sqrt{2\log N}\right)&\leq& \sum_{i\in[N]}\mathrm{Pr}\left(G_i\geq (1+\epsilon)\sqrt{2\log N}\right)\\
&\leq&\frac{N e^{-(1+\epsilon)^2\log N}}{2(1+\epsilon)\sqrt{\pi\log N}}\\
&\leq & N^{-\epsilon}
\end{eqnarray*}
If we further assume that $G_i$'s are independent, then
\begin{eqnarray*}
\mathrm{Pr}\left(\max_{i\in[N]}G_i< (1-\epsilon)\sqrt{2\log N}\right)&=&\prod_{i\in[N]}\mathrm{Pr}\left(G_i<(1-\epsilon)\sqrt{2\log N}\right)\\
&=&\prod_{i\in[N]}\left[1-\mathrm{Pr}\left(G_i>(1-\epsilon)\sqrt{2\log N}\right)\right]
\end{eqnarray*}
By (\ref{gd}) we obtain
\begin{eqnarray*}
\mathrm{Pr}\left(\max_{i\in[N]}G_i< (1-\epsilon)\sqrt{2\log N}\right)&\leq& \left(1-\frac{(1-\epsilon)\sqrt{2\log N}}{\sqrt{2\pi}(1+2(1-\epsilon)^2\log N)}\frac{1}{N^{(1-\epsilon)^2}}\right)^N
\end{eqnarray*}
When (\ref{epn}) holds, we have 
\begin{eqnarray*}
\mathrm{Pr}\left(\max_{i\in[N]}G_i< (1-\epsilon)\sqrt{2\log N}\right)\leq \left(1-\frac{1}{N^{1-\epsilon}}\right)^{N^{1-\epsilon}\cdot N^{\epsilon}}\leq  e^{-N^{\epsilon}}
\end{eqnarray*}
Then the lemma follows.

\bigskip
\bigskip
\noindent{\textbf{Acknowledgements.}} ZL's research is supported by National Science Foundation grant 1608896 and Simons Foundation grant 638143. 
\bibliography{cdbd}

\providecommand{\bysame}{\leavevmode\hbox to3em{\hrulefill}\thinspace}
\providecommand{\MR}{\relax\ifhmode\unskip\space\fi MR }
\providecommand{\MRhref}[2]{%
  \href{http://www.ams.org/mathscinet-getitem?mr=#1}{#2}
}
\providecommand{\href}[2]{#2}
\begin{thebibliography}{1}

\bibitem{EA18}
Emmanuel Abbe, \emph{Community detection and stochastic block models: Recent
  developments}, Journal of Machine Learning Research \textbf{18} (2018),
  1--86.

\bibitem{ABH15}
Emmanuel Abbe, Afonso~S. Bandeira, and Georgina Hall, \emph{Exact recovery in
  the stochastic block model}, IEEE Transactions on Information Theory
  \textbf{62} (2016), 471--487.

\bibitem{HHR}
Tailen Hsing, J\"urg H\"usler, and Rolf-Dieter Reiss, \emph{The extremes of a
  triangular array of normal random variables}, The Annals of Applied
  Probability \textbf{6} (1996), 671--686.

\bibitem{JMT16S}
A.~Javanmard, A.~Montanari, and F.~Ricci-Tersenghi, \emph{Performance of a
  community detection algorithm based on semidefinite programming}, J.~Phys.:
  Conf. Ser. 699 (2016), 012015.

\bibitem{JMRT16}
\bysame, \emph{Phase transitions in semidefinite relaxations}, Proceedings of
  the National Academy of Sciences \textbf{113(16)} (2016), E2218--2223.

\bibitem{KBG}
C.~Kim, A.~Bandeira, and M.~Goemans, \emph{Community detection in hypergraphs,
  spikedtensor models, and sum-of-squares}, 2017 12th International Conference
  on Sampling Theory and Applications (2017), 124--128.

\bibitem{LM14}
L.~Massouli\'e, \emph{Community detection thresholds and the weak {R}amanujan
  property}, Proceedings of the 46th Annual ACM Symposium on Theory of
  Computing (2014), 694--703.

\bibitem{MNS13}
Elchanan Mossel, Joe Neeman, and Allan Sly, \emph{A proof of the blockmodel
  threshold conjecture}, Combinatorica \textbf{38} (2018), 665--708.

\bibitem{TW}
Craig~A. Tracy and Harold Widom, \emph{Distribution functions for largest
  eigenvalues and their applications}, Proceedings of the International
  Congress of Mathematicians \textbf{I} (2002), 587--596.

\end{thebibliography}
\bibliographystyle{amsplain}

\end{document}